\documentclass[11pt,letterpaper]{article}
\usepackage[margin=1in]{geometry}
\usepackage[utf8]{inputenc} 
\usepackage[T1]{fontenc}    
\usepackage{dsfont}
\usepackage[maxfloats=100]{morefloats}[2015/07/22]
\usepackage{multirow}
\usepackage{wrapfig}
\usepackage[T1]{fontenc}
\usepackage{lmodern}
\usepackage{booktabs}       
\usepackage{amsfonts}       
\usepackage{nicefrac}       
\usepackage{microtype}      
\usepackage{enumerate}
\usepackage{lipsum}
\usepackage{mathtools}
\usepackage{cuted}
\usepackage{float}
\usepackage[dvipsnames,table,xcdraw]{xcolor}
\usepackage{bbm}
\usepackage{varioref}
\usepackage{hyperref}
                                
\usepackage{amssymb} 
\usepackage{rotating}
\usepackage{graphicx}
\usepackage{subcaption}
\usepackage[normalem]{ulem}

\usepackage{amsthm}
\DeclareUnicodeCharacter{00A0}{~}
\usepackage{algorithm,algorithmic}
\usepackage{tcolorbox}

\newtheorem{assumption}{Assumption}
\newtheorem{theorem}{Theorem}
\newtheorem{lemma}{Lemma}
\newtheorem{proposition}{Proposition}
\newtheorem{definition}{Definition}
\newtheorem{corollary}{Corollary}
\theoremstyle{plain}
\newtheorem{remark}{Remark}

 \newcommand{\fa}[1]{{\color{black}#1}}
\newcommand{\fy}[1]{{\color{black}#1}}
\newcommand{\fyy}[1]{{\color{black}#1}}

\newcommand{\afj}[1]{{\color{black}#1}}
\newcommand{\me}[1]{{\color{black}#1}}

\usepackage[textsize=small]{todonotes}

\usepackage{tcolorbox}
\title{\LARGE \bf  Stochastic Approximation for Estimating the Price of Stability in Stochastic Nash Games}

\author{Afrooz~Jalilzadeh$^{1}$ \and Farzad~Yousefian$^{2}$ \and Mohammadjavad Ebrahimi$^{3}$
\thanks{$^1$Afrooz Jalilzadeh is an Assistant Professor of Systems and Industrial Engineering at The University of Arizona, Tucson, AZ 85721, USA. ({\tt\small afrooz@arizona.edu})}
\thanks{$^2$Farzad Yousefian is an Assistant Professor of Industrial and Systems Engineering at Rutgers University, Piscataway, NJ 08854, USA. {(\tt\small farzad.yousefian@rutgers.edu})}
\thanks{$^3$Mohammadjavad Ebrahimi is currently a PhD student of Industrial and Systems Engineering at Rutgers University, Piscataway, NJ 08854, USA. ({\tt\small mohammadjavad.ebrahimi@rutgers.edu})}}

\begin{document}
\sloppy
\maketitle
\thispagestyle{empty}
\pagestyle{plain}

\begin{abstract}
 The goal in this paper is to approximate the Price of Stability (PoS) \fa{in} stochastic Nash games using stochastic approximation (SA) schemes. PoS is amongst the most popular metrics in game theory and provides an avenue for estimating the efficiency of Nash games. In particular, knowing the value of PoS can help with designing efficient networked systems, including transportation networks and power market mechanisms. Motivated by the lack of efficient methods for computing the PoS, first we consider stochastic optimization problems with a nonsmooth and merely convex objective function and a merely monotone stochastic variational inequality (SVI) constraint. This problem appears in the numerator of the PoS \fa{ratio}. We develop a randomized block-coordinate stochastic extra-(sub)gradient method where we employ a novel iterative penalization scheme to account for the mapping of the SVI in each of the two gradient updates of the algorithm. We obtain an iteration complexity of the order $\epsilon^{-4}$ that appears to be best known result for this class of constrained stochastic optimization problems, \fa{where $\epsilon$ denotes an arbitrary bound on suitably defined infeasibility and suboptimality metrics}. Second, we develop an SA-based scheme for approximating the PoS and derive lower and upper bounds on the approximation error. To validate \fa{the} theoretical findings, we provide preliminary simulation results on a \fa{networked} stochastic Nash Cournot competition.
\end{abstract}

\section{Introduction}
\fa{The} goal in this paper lies in the development of a stochastic approximation method, equipped with performance guarantees, for computing the price of stability (PoS) ratio in monotone stochastic Nash games. Nash equilibrium (NE) is \fa{a fundamental concept} in game theory \fa{and captures} a wide range of phenomena in engineering, economics, and finance~\cite{FacchineiPang2003}. Consider a stochastic Nash game with $N$ players, each \fa{associated} with a strategy set $X_i \subseteq \mathbb{R}^{n_i}$ and a cost function $f_i$. Player $i$'s objective is to determine, for any collection of arbitrary strategies of the other players, denoted by $x^{(-i)}$, an optimal strategy $x^{(i)}$ that solves the stochastic minimization problem
\begin{align}\label{nash}\tag{P$_i(x^{(-i)})$}
 &\hbox{minimize}_{x^{(i)}} \qquad \mathbb{E}\left[f_i\left(\left(x^{(i)};x^{(-i)}\right),\xi\right)\right],\\
& \hbox{subject to}  \qquad x^{(i)} \in X_i,\notag
 \end{align} 
where $f_i\left(\left(x^{(i)};x^{(-i)}\right),\xi\right)$ denotes a random cost function associated with the $i$th player that is parameterized in terms of the action of the player $x^{(i)}$, actions of other players denoted by $x^{(-i)}$, and a random variable $\xi$, where $\xi: \Omega\to\mathbb{R}^d$ denotes a random variable associated with the probability space $(\Omega, \mathcal{F},\mathbb{P})$.
\fa{\begin{remark}\em
Throughout, similar to \cite{juditsky2011solving,Nem04,yousefian2017smoothing}, we focus on settings where the stochasticity is only present in the objective function of the players. In particular, we assume that the strategy sets are deterministic.
\end{remark}}

An NE is described as a collection of specific strategies chosen by all the players, denoted by the tuple $x\triangleq \left(x^{(1)};\ldots;x^{(N)}\right)$ where no player can reduce her cost by unilaterally changing her strategy within her feasible strategy set. Mathematically, NE can be described as a vector $x$ that satisfies, for all $i=1,\ldots,N$, \fa{the inequality given as}
\begin{align}\label{eqn:NE_ineq}
\mathbb{E}\left[f_i\left(\left(x^{(i)};x^{(-i)}\right),\xi\right)\right] \leq \mathbb{E}\left[f_i\left(\left(y^{(i)};x^{(-i)}\right),\xi\right)\right], \qquad \hbox{for all }y^{(i)} \in X_i.
\end{align}
Suppose $n$ denotes the total number of dimensions associated with an NE, i.e., $n\triangleq \sum_{i=1}^N n_i$. Let us define the set $X \subseteq \mathbb{R}^n$ as the Cartesian product of the players' strategy sets, i.e., $X\triangleq \prod_{i=1}^N X_i$. Also, under a \fa{differentiability} assumption, define the stochastic mapping $F:\mathbb{R}^n\times \mathbb{R}^d \to \mathbb{R}^n$ and its deterministic counterpart $F:\mathbb{R}^n \to \mathbb{R}^n$ as the collection of players' gradient mappings as  
$$F(x) \triangleq \mathbb{E}[F(x,\xi)], \quad \hbox{where } F(x,\xi)\triangleq \left(\nabla_{x^{(1)}}f_1(x,\xi),\ldots,\nabla_{x^{(N)}}f_N(x,\xi)\right).$$ 
Note that for expository ease, we use $F$ in naming both deterministic and stochastic mappings. Then, under the convexity of the players' objective functions, the problem of seeking an NE to the game characterized by problems \eqref{nash} for $i=1,\ldots,N$, can be compactly formulated as a stochastic variational inequalities (VI) problem, denoted by $\mbox{VI}(X,F)$. Recall that a vector $x^*\in X$ solves $\mbox{VI}(X,F)$ if 
$$(y-x^*)^T F(x^*)\geq 0, \qquad \hbox{for all } y\in X.$$
Indeed, it can be observed that the inequality above compactly captures the optimality conditions of the convex programs \eqref{eqn:NE_ineq} written for all $i=1,\ldots,N$. To this end, computing a solution to $\mbox{VI}(X,F)$ leads to finding an NE to the described stochastic Nash game.
Generally, a VI problem may admit multiple solutions leading to a collection of NEs. Throughout, we let $\mbox{SOL}(X,F)$ denote the solution set of the $\mbox{VI}(X,F)$.

In this paper, our aim is to develop a provably convergent scheme for estimating the efficiency in stochastic Nash games with monotone mappings. The notion of efficiency in Nash games is a storied area of research and dates back to the celebrated Prisoner's Dilemma. In fact, Nash equilibrium is provably known to be inefficient~\cite{Dubey86}, in the sense that the competition among the players often leads to a degradation of the overall performance of the system \fa{of} players. In view of this, understanding the efficiency of an NE has received much attention in game theory. Among, the popular measures of the efficiency of NE is a metric called \emph{price of stability} (PoS) \cite{roughgarden2010algorithmic}. Given an arbitrary cost metric for quantifying the overall performance of the system, PoS is defined as the ratio between the following two quantities: (1) the minimal cost attained by the best Nash equilibrium (among possibly many NEs); (2) the optimal cost when the competition among the players is (hypothetically) suppressed. Let stochastic function $f:\mathbb R^n\times \mathbb{R}^d\to\mathbb R$ denote the system's overall performance metric. Mathematically and following our notation, PoS can be formulated \fa{as}
\begin{tcolorbox}
    \begin{align}\label{pos}\mbox{PoS} \ \triangleq \ \frac{\min_{x\in \tiny{\mbox{SOL}}(X,\, \mathbb E[F(\bullet,\xi)])}\mathbb E[f(x,\xi)]}{\min_{x\in X}\mathbb E[f(x,\xi)]}.
    \end{align}
     \end{tcolorbox}
\begin{remark}\em  We note that the function $f$ may or may not relate to the individual objective functions of the players denoted by $f_i$. In the literature~\cite{PoS08,JohariThesis}, different choices have been considered. Two common examples include the {\it utilitarian} approach where $f$ is defined as the summation of all players' objectives, and the {\it egalitarian} approach where $f$ is defined as the maximum of the individual objective functions.
\end{remark}

\fy{Evaluating the PoS ratio, even in deterministic problems, is a \fa{computationally} challenging task. To elaborate on this, we provide a simple example in the following.\hfill}
       \begin{wrapfigure}{r}{0.4\textwidth}
	\hspace{-.1in}
	\begin{center}
	\includegraphics[width=0.4\textwidth]{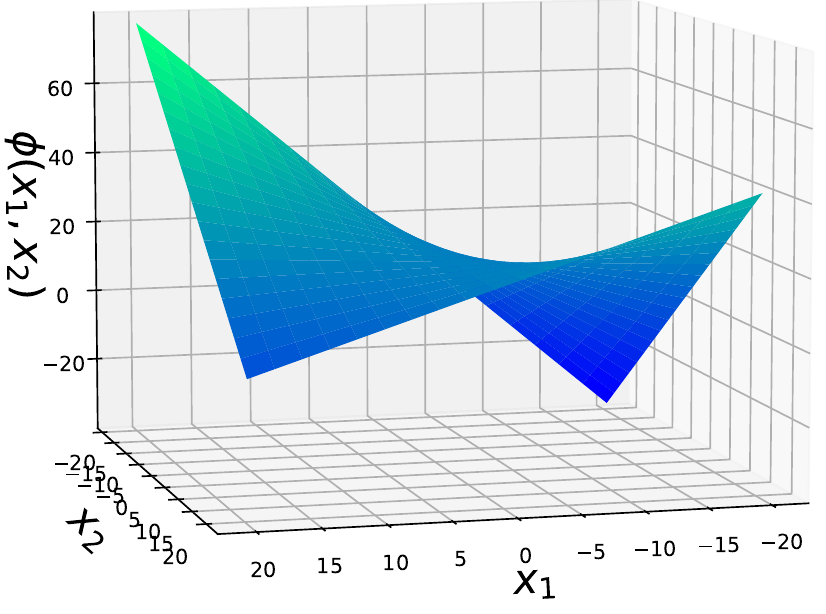} 
	\end{center}
	\vspace{-0.2in} 
	\caption{Function $\phi$ in problem \eqref{eqn:minimax_toy}}\label{fig:minmax_toy}
\end{wrapfigure}

\noindent {\bf Example (PoS in saddle-point problems).} The problem of seeking a saddle-point in minmax optimization is an important class of equilibrium problems that has received considerable attention in game theory~\cite{FacchineiPang2003,korpelevich1976extragradient,Nem04,ned09} and more recently, in adversarial learning \cite{GAN14}, fairness in machine learning \cite{fairgan18}, and distributionally robust federated learning \cite{distrobustfl20}. In fact, the canonical minmax problem can be viewed as a subclass of two-person zero-sum games. The existence of equilibrium in such a game was established by the von Neumann's minmax theorem in 1928~\cite{minimax28}. To elaborate, consider a minmax problem given as 
\begin{align}\label{eqn:minimax_toy}
\min_{11 \leq x_1\leq 60} \ \max_{10 \leq x_2\leq 50} \phi(x_1,x_2)\triangleq \ 20-0.1x_1x_2+x_1.
\end{align}
Figure~\ref{fig:minmax_toy} shows the saddle-shaped function $\phi$. Associated with problem \eqref{eqn:minimax_toy}, we can consider a pair of optimization problems as 

\noindent\begin{minipage}{.5\linewidth}
\begin{equation}\label{eqn:player1problem}
\Bigg\{\begin{aligned}
 &\hbox{minimize}_{x_1} \qquad  f_1(x_1,x_2) \triangleq 20-0.1x_1x_2+x_1 \\
& \hbox{subject to}  \qquad  x_1 \in X_1 \triangleq [11,60], 
\end{aligned}
\end{equation}
\end{minipage} 
\begin{minipage}{.5\linewidth}
\begin{equation}\label{eqn:player2problem}
\Bigg\{\begin{aligned}
 &\hbox{minimize}_{x_2} \qquad  f_2(x_1,x_2) \triangleq -20+0.1x_1x_2-x_1 \\
& \hbox{subject to}  \qquad  x_2 \in X_2 \triangleq [10,50]. 
\end{aligned}
\end{equation}
\end{minipage}
\vspace{.1in}

\noindent Problems~\eqref{eqn:player1problem} and~\eqref{eqn:player2problem} together construct a  two-person zero-sum Nash game. From~\cite[1.4.2 Proposition]{FacchineiPang2003}, the set of saddle-points are the solutions to the variational inequality problem $\mbox{VI}(X,F)$ where we define $$F(x_1,x_2) \triangleq (\nabla_{x_1} f_1(x), \nabla_{x_2} f_2(x))= (-0.1x_2+1, 0.1x_1) \quad \hbox{and} \quad X\triangleq X_1\times X_2.$$ 
Note that the mapping $F$ is merely monotone, in view of $(F(x)-F(y))^T(x-y) =0$ for all $x \in \mathbb{R}^2$ and $y \in \mathbb{R}^2$. We observe that the set of all the saddle-points is given by $\mbox{SOL}(X,F) = \{(x_1,x_2)\mid x_1 \in [11,60], \ x_2=10\}$, implying that there are infinitely many Nash equilibria to this game characterized by the convex set $\mbox{SOL}(X,F)$. To measure \fa{the} PoS, let us consider the global metric defined as $f(x_1,x_2) \triangleq 20+|x_1-x_2|$ for instance. This implies that the numerator of \fa{the} PoS in~\eqref{pos} is equal to $21$, while its denominator is equal to $20$. As such, we obtain $\hbox{PoS} = 1.05$, implying that the competition in the game leads to an $\%5$ loss in the metric $f$. Although in this simple example, we are able to evaluate the PoS, in practice, we often encounter several challenges that \fa{may} make this impossible. Two main challenges are explained as follows: (i) The solution set of the VI is often unknown. Even in deterministic settings, it is often impossible to determine the entire set $\mbox{SOL}(X,F)$; (ii) Nash games might be afflicted by the presence of uncertainty which motivates the need for leveraging Monte Carlo sampling schemes, such as stochastic approximation, for contenting with stochasticity and the large-scale of the problem. \fa{For example, in distributionally robust federated learning \cite{distrobustfl20}, the problem is cast a stochastic minmax problem where the stochasticity emerges from the probability distribution of the local data sets, privately maintained by the clients.}

To estimate the PoS with guarantees, first, we need to solve the numerator of the right-hand side of \eqref{pos} that is characterized as a stochastic optimization with a {stochastic VI} constraint. Naturally, addressing the presence of VI constraints is a challenging task in optimization. This is mainly because VI constraints do not appear to lend themselves to standard Lagrangian relaxation schemes. In this work, this challenge is exacerbated due to the presence of uncertainty in the mapping of the VI constraint. 
To this end, our goal is to employ stochastic approximation (SA) schemes. SA is an iterative scheme that has been widely employed for solving problems in which the objective function is corrupted by a random noise. In the context of optimization problems, the function values and/or higher-order information are estimated from noisy samples in a Monte Carlo simulation procedure \cite{broadie2014multidimensional}. The SA scheme, first introduced by Robbins and Monro \cite{robbins51sa}, has been studied extensively in recent years for addressing stochastic optimization and stochastic variational inequality problems~\cite{Polyak92acceleration,Farzad1,juditsky2011solving,koshal12regularized}.

In addressing constrained stochastic formulations, the majority of the SA schemes in the existing literature address the standard cases where the  constraints are in the form of functional inequalities, equalities, or easy-to-project sets. However, motivated by the need for efficiency estimation in stochastic Nash games, we aim at devising a provably convergent SA method for estimation of the PoS. To this end, our primary interest lies in solving the following stochastic optimization problem whose constraint set is characterized as the solution set of a stochastic VI problem. This optimization problem is given \fa{as}
\begin{tcolorbox}
    \vspace{-0.15in}
\begin{align}\label{prob:sopt_svi} 
&\hbox{minimize} \qquad \mathbb{E}[f(x,\xi)]\\
& \hbox{subject to}  \qquad x \in \mbox{SOL}(X, \mathbb{E}[F(\bullet,\xi)]),\notag
\end{align}
\end{tcolorbox}
\noindent where  $f:\mathbb R^n\times \mathbb{R}^d\to\mathbb R$ is a convex function, $X \subseteq \mathbb{R}^n$ is the Cartesian product of the component sets $X_i \subseteq \mathbb{R}^{n_i}$ where $\sum_{i=1}^N n_i =n$, i.e., $X \triangleq \prod\nolimits_{i=1}^NX_i$. We let {the $i$th block-coordinate of the mapping $F(\bullet,\xi)$ be denoted by} $F_i:\mathbb R^n\times \mathbb{R}^d\to \mathbb R^{n_i}$ for any $i\in [N]\triangleq \{1,\ldots,N\}$. 
{As noted earlier, f}or the ease of presentation, throughout we 
define $f(x) \triangleq \mathbb{E}[f(x,\xi)]$ and $F(x) \triangleq \mathbb{E}[F(x,\xi)]$. 

%
{{\bf Existing literature on VIs.}} The variational inequality problem has been {extensively} studied in the literature due to its versatility in {capturing} a wide range of problems including optimization, equilibrium and complementarity problems, amongst others \cite{FacchineiPang2003}. The extra-gradient method, initially proposed by Korpelevich \cite{korpelevich1976extragradient} and its extensions \cite{censor2011subgradient,censor2012extensions,chen2017accelerated,iusem2011korpelevich,juditsky2011solving,yousefian2014optimal,yousefian2018stochastic}, is a classical method for solving VI problems which requires weaker assumptions than {standard gradient schemes} \cite{bertsekas1997nonlinear,sibony1970methodes}. In stochastic {problems}, amongst the earliest schemes for resolving stochastic variational inequalities via stochastic approximation was presented
by Jiang and Xu \cite{jiang2008stochastic} {under the strong monotonicity and smoothness assumptions of the mapping}. Regularized variants {of SA schemes} were developed by Koshal et al.~\cite{koshal12regularized} for {addressing stochastic VIs with} merely monotone {mappings. Further, smoothness requirements were weakened by leveraging randomized smoothing in~\cite{yousefian2013regularized,yousefian2017smoothing}}. In the absence of {strong monotonicity}, extra-gradient approaches that rely on two projections per iteration provide an avenue for resolving merely monotone problems~\cite{jalilzadeh2019proximal}. The per-iteration complexity can be reduced to a single projection via projected reflected gradient and
splitting techniques as examined in~\cite{cui2016analysis,cui2021analysis} (also see~\cite{hsieh2019convergence}). When the assumption on the
mapping is {weakened} to pseudomonotonicity and its variants, rate statements have been provided
in \cite{iusem2017extragradient,kannan2014pseudomonotone,kannan2019optimal} via a stochastic extra-gradient framework. 

{{\bf Gap in the literature.}} Despite these advances in {addressing VIs and their stochastic variants}, solving problem \eqref{prob:sopt_svi} remains challenging. {In fact, we are unaware of any provably convergent stochastic approximation method for solving problem \eqref{prob:sopt_svi} that appears to be essential in estimating the PoS, defined as~\eqref{pos}}. One main approach to solve \eqref{prob:sopt_svi}, when the constraint set is {the} solution {set} of {a} deterministic VI and {the} objective function is also deterministic, is {the} sequential regularization (SR) approach which is a two-loop framework (see \cite[Chapter 12]{FacchineiPang2003}). In each iteration of the SR scheme, a regularized VI is required to be solved and convergence has been shown under the monotonicity of the mapping $F$ and closedness and convexity of the set $X$. However, the iteration complexity of the SR algorithm is unknown and it requires solving a series of increasingly more difficult VI problems. To resolve these shortcomings, recently, Kaushik and Yousefian \cite{kaushik2021method} developed a more efficient first-order method called averaging randomized block iteratively regularized gradient. Non-asymptotic suboptimality and infeasibility convergence {rates} of $\mathcal O(1/K^{0.25})$ {have} been obtained where $K$ is the total number of iterations. Here, we consider a more general problem with a stochastic objective function and a stochastic VI constraint. {Employing} a novel iterative penalization technique, we propose an extra-(sub)gradient-based {SA method and we derive convergence results in expectation, of the same order of magnitude} as in \cite{kaushik2021method}, despite the presence of stochasticity in \fa{the} {both} levels of the problem. 

{\bf Main contributions.} In this paper, we study a stochastic optimization problem with a nonsmooth and {merely} convex objective function and a {constraint set characterized as the solution set of a stochastic variational inequality problem}. {Motivated by the absence of efficient and scalable SA methods for addressing this class of constrained stochastic optimization problems, we} develop {a single-timescale first-order stochastic approximation method with block-coordinate updates}, called Averaging Randomized Iteratively Penalized Stochastic Extra-Gradient Method (aR-IP-SeG). We {derive convergence rates in terms of suitably defined metrics for} suboptimality and infeasibility. In particular, in {Theorem~\ref{thm:rates}}, we obtain an iteration complexity of the order \fa{of} $\epsilon^{-4}$ {where $\epsilon$ denotes a user-specified bound on both the objective function's error and a suitably defined infeasibility metric (i.e., dual gap function). This iteration complexity} appears to be best known result for this class of constrained stochastic optimization problems. Moreover, utilizing the proposed extra-(sub)gradient-based method, we derive lower and upper bounds, {both of the order \fa{$1/K^{0.25}$}, for approximating the price of stability}. Such guarantees appear to be new in computing the PoS (see Lemma~\ref{prop:pos_bounds}).

{\bf Outline of the paper.} Next, we introduce the notation that we use throughout the paper. In the next section, we precisely state the main definitions and assumptions that we need for the convergence analysis. In Section~\ref{sec:alg}, we describe the {aR-IP-SeG} algorithm  to solve problem~\eqref{prob:sopt_svi} and the complexity analysis is provided in Section~\ref{sec:rate}. Additionally, in Section~\ref{sec:pos}, we propose {a scheme} to  approximate the price of stability in \eqref{pos} with guarantees. Finally, some empirical experiments are presented in Section~\ref{sec:num} {for addressing  a stochastic Nash Cournot competition over a network where we compare} our proposed scheme {with the few existing} schemes {that can be employed for} estimating the PoS.

{{\bf Notation.}} {Throughout, we often use column vectors to present the algorithms and discuss the convergence analysis. For a convex function $h:\mathbb{R}^n\rightarrow\mathbb{R}$ with the domain $\text{dom}(h)$ and  any $x\in\text{dom}(h)$, a vector $ \tilde{\nabla}h(x) \in\mathbb{R}^n$ is called a subgradient of $h$ at $x$ if $h(x) +\fa{\tilde{\nabla}h(x)^T(y-x)}\leq h(y)$ for  all $y\in \text{dom}(h)$. We let $\partial h(x)$ denote the subdifferential set of  function $h$ at $x$. Given a vector $x \in \mathbb{R}^n$, we use $x^{(i)} \in \mathbb{R}^{n_i}$ to denote its $i$th block-coordinate. We let $\tilde{\nabla_i}h(x)$ denote the $i$th block-coordinate of $\tilde{\nabla}h(x)$. We use similar notation for referring to the $i$th block-coordinate of mappings. We let} $\mathbb{E}[\bullet]$ {denote} the expectation with respect to the all probability distributions under study. We use \fa{filtration} to take conditional expectations with respect to a subgroup of probability distributions. We denote the optimal objective value of the problem~\eqref{prob:sopt_svi} by $f^*$. The Euclidean projection of vector $x$ onto a convex set $X$ is denoted by $\mathcal P_X(x)$, where $\mathcal P_X(x)\triangleq \mbox{argmin}_{y\in X}\|y-x\|^2$. Throughout the paper, unless specified otherwise, $k$ denotes the iteration counter while $K$ represents the total number of steps employed in the proposed methods. {Moreover, we define $\mbox{dist}(x,X)\triangleq \min_{y\in X}\|y-x\|$.}

\section{Algorithm Outline}\label{sec:alg}
{Our goal in this section is to devise an SA scheme for solving problem~\eqref{prob:sopt_svi}. To this end, we develop a method, called} Averaging Randomized Iteratively Penalized Stochastic Extra-Gradient (aR-IP-SeG) presented {by} Algorithm~\ref{algorithm:aR-IP-SeG}. Compared with standard extra-gradient methods, a key novelty in the design of {aR-IP-SeG} lies in how we iteratively penalize the stochastic mapping of the VI using the parameter $\rho_k$. Intuitively, this is done to penalize the infeasibility of the generated iterate in terms of the stochastic VI constraint in problem~\eqref{prob:sopt_svi}. At each iteration $k$, we select indices $i_k$ and $\tilde i_k$ uniformly at random and update only the corresponding blocks of the variables $y_k$ and $x_k$ by taking a step in a negative direction of the partial sample subgradient $\tilde\nabla_i f(\bullet,\xi_k)$ and sample map $F_i(\bullet, \xi_k)$ for $i=i_k$ and $\tilde i_k$. Then, we {compute} the projection onto sets $X_{i_k}$ and $X_{\tilde i_k}$. \fa{Note that each player is associated with multi-dimensional strategies, denoted by $n_i$ for $i=1,\ldots, N$, where $\sum_{i=1}^N n_i=n$. Also, at each iteration, a player is randomly chosen to update her/his full block of strategy. Also,} $\gamma_k$ and $\rho_k$ denote the stepsize and the {penalty} parameter, respectively. Finally, the output of the proposed algorithm is a weighted average of the generated sequence {$\{y_k\}$}. {This is done in a novel way through incorporating both the stepsize and the penalty parameter into averaging weights.}
\begin{algorithm}
  \caption{Averaging Randomized Iteratively Penalized Stochastic Extra-Gradient Method (aR-IP-SeG)}
\label{algorithm:aR-IP-SeG}
    \begin{algorithmic}[1]
    \STATE \textbf{initialization:} Set random initial points $x_0, y_0\in X$, {an initial} stepsize $\gamma_0>0$, {an initial penalty parameter $\rho_0>0$} a scalar $r<1$, $\bar{y}_0= y_0$, and $\Gamma_0=0$.
    \FOR {$k=0,1,\ldots,K-1$}
     \STATE Generate $i_k$ and $\tilde i_k$ uniformly from $\{1,\ldots,N\}$.
     \STATE Generate $\xi_k$ and $\tilde \xi_k$ as realizations of the random vector $\xi$.
      \STATE Update the variables $y_k$ and $x_k$ \fa{as}
\begin{align}
y_{k+1}^{(i)}&:= \left\{\begin{array}{ll}\mathcal{P}_{X_i}\left(x_{k}^{(i)}-\gamma_k(\tilde \nabla_i f(x_k,\tilde \xi_k) + \rho_k F_{i}(x_{k},\tilde \xi_k))\right)
&\hbox{if } i=\tilde i_k,\cr \hbox{} &\hbox{}\cr
x_k^{(i)}& \hbox{if } i\neq \tilde i_k,\end{array}\right.\label{eqn:y_k_update_rule}
\\
&\hbox{} \notag\\
x_{k+1}^{(i)}&:=\left\{\begin{array}{ll}\mathcal{P}_{X_i}\left(x_{k}^{(i)} - \gamma_k(\tilde \nabla_i f(y_{k+1}, \xi_k) + \rho_k F_{i}(y_{k+1}, \xi_k))\right)
&\hbox{if } i=i_k,\cr \hbox{} &\hbox{}\cr
x_k^{(i)}& \hbox{if } i\neq i_k.\end{array}\right.\label{eqn:x_k_update_rule}
\end{align}
     \STATE Update $\Gamma_k$ and $\bar y_{k}$ using the following recursions\fa{:}
\begin{align}
&\Gamma_{k+1}:=\Gamma_k+(\gamma_k\rho_k)^r,\label{eqn:averaging_eq1}\\
&\bar y_{k+1}:=\frac{\Gamma_k \bar y_k+(\gamma_k\rho_k)^r y_{k+1}}{\Gamma_{k+1}}.\label{eqn:averaging_eq2}
\end{align}
        \ENDFOR
     \STATE Return $\bar y_{K}$.
   \end{algorithmic}
\end{algorithm}
 
Throughout the paper, we consider the following assumptions on the map $F$, objective function $f$ and set $X$ in problem~\eqref{prob:sopt_svi}. 
\begin{assumption}[Problem properties]\label{assum:problem} \em Consider problem \eqref{prob:sopt_svi}. Let the following holds.

\noindent (i) Mapping $F(\bullet):\mathbb{R}^n \to \mathbb{R}^{n}$ is {vector-valued}, continuous, and merely monotone on its domain, {i.e., for all $x,y\in \mbox{dom}(F),$ ${(F(x)-F(y))^T(x-y)}\geq0 $}. 

\noindent (ii) Function $f(\bullet):\mathbb{R}^n \to \mathbb{R}$ is {closed, proper, and} merely convex on its domain. 

\noindent (iii) Set $X \subseteq \mbox{int}\left(\mbox{dom}(F)\cap\mbox{dom}(f)\right)$ is nonempty, compact, and convex. 
\end{assumption}

\begin{remark}\label{rem:bounds}\em
In view of Assumption \ref{assum:problem}, the subdifferential set 	$\partial f(x)$ is nonempty for all $x \in \mbox{int}(\mbox{dom}(f))$. Also, $f$ has bounded subgradients over $X$. Throughout, we let scalars $D_X$ and $D_f$ be defined as $D_X\triangleq \sup_{x \in X} \|x\|$ and $D_f\triangleq \sup_{x \in X} |f(x)|$, respectively. Also, we let $C_F>0$ and $C_f>0$ be scalars such that $\|F(x)\|\leq C_F$, and $\|\tilde \nabla f(x)\|\leq C_f$ for all $\tilde \nabla f(x)\in \partial f(x)$, for all $x \in X$. 
\end{remark}
Next, we impose {some standard conditions} on the conditional bias and the conditional second moment on
the sampled subgradient $\tilde \nabla f(\bullet,{\xi})$ and sampled map $F(\bullet,{\xi})$ produced by the oracle.
\begin{assumption}[Random samples]\label{assum:rnd_vars}\em
\noindent (a) The random samples $\tilde \xi_k$ and $\xi_k$ are i.i.d., and $\tilde i_k$ and $i_k$ are i.i.d. from the range $\{1,\ldots,N\}$. Also, all these random variables are independent from each other. 

\noindent (b) For all $k\geq 0$ the stochastic mappings $F(\bullet,\tilde \xi_k)$ and $F(\bullet,\xi_k)$ are both unbiased estimators of $F(\bullet)$. Similarly, 
{ $\tilde \nabla f(\bullet,\tilde \xi_k)$ and $\tilde \nabla f(\bullet,\xi_k)$ are both unbiased estimators of $\tilde \nabla f(\bullet)$}.

\noindent (c) For all \fa{$x \in X$}, there exist $\fa{\nu_{F}},\nu_{f}>0$ such that $\mathbb{E}[\|F(x,\xi) - F(x)\|^2 \mid x] \leq\nu_{F}^2$ and $\mathbb{E}[\|\tilde \nabla f(x, \xi) - \tilde \nabla f(x)\|^2 \mid x] \leq\nu_{f}^2$.
\end{assumption}
{\begin{remark}\label{rem:bounds2}\em
Under Assumption~\ref{rem:bounds}, we can write 
$\mathbb{E}[\|F(x,\xi)\|^2 \mid x] = \mathbb{E}[\|F(x,\xi) - F(x)\|^2 \mid x] +\|F(x)\|^2  \leq \nu^2_F +C_F^2$, where we use Remark~\ref{rem:bounds}. Similarly, we have that $\mathbb{E}[\|\tilde \nabla f(x,\xi)\|^2 \mid x]  \leq \nu^2_f +C_f^2$.
\end{remark}}
\fa{\begin{remark} \em
In the case when the stochastic VI represents a Nash game, we assume that each player has access to stochastic gradient of its objective as well as stochastic gradient of the global function $f$.
\end{remark}}
\section{Preliminaries and Background}\label{sec:assump}

\begin{definition}\label{def:hist}
We denote the history of the method by $\mathcal{F}_k$ for $k \geq 0$ defined as
\begin{align*}
\mathcal{F}_k \triangleq \cup_{t=0}^k\{\tilde \xi_t,\tilde i_t,  \xi_t,  i_t\} \cup \{x_0,y_0\}.
\end{align*}
\end{definition}
Next, we define the errors for stochastic approximation of objective function $f$ and operator $F$, and block-coordinate sampling. In particular, we use {the terms} $w_{\bullet,k}$ and $\tilde w_{\bullet,k}$ {to denote} the errors of stochastic approximation {involved at} iteration $k$ and {similarly, the terms} $e_{\bullet,k}$ and $\tilde e_{\bullet,k}$ for the errors of block-coordinate sampling. 
\begin{definition}[Stochastic errors]\label{def:errs}
For all $k \geq 0$ we define 

\noindent
\begin{minipage}[t]{.4\linewidth}
 \begin{itemize}
\item [] $\tilde{w}_{F,k} \triangleq F(x_k,\tilde \xi_k) - F(x_k)$,
\item  [] ${w}_{F,k} \triangleq F(y_{k+1}, \xi_k) - F(y_{k+1})$, 
\item  [] $\tilde{e}_{F,k} \triangleq N\mathbf{U}_{\tilde i_k}F_{\tilde i_k}(x_k,\tilde \xi_k) - F(x_k,\tilde \xi_k)$, 
\item [] ${e}_{F,k} \triangleq N\mathbf{U}_{ i_k}F_{ i_k}(y_{k+1},\xi_k) - F(y_{k+1}, \xi_k)$.
\end{itemize}
\end{minipage}\hfill
\begin{minipage}[t]{.6\linewidth}
 \begin{itemize}
\item [] $\tilde{w}_{f,k} \triangleq \tilde \nabla f(x_k,\tilde \xi_k) -\tilde \nabla f(x_k)$,
\item [] ${w}_{f,k} \triangleq \tilde \nabla f(y_{k+1}, \xi_k) -\tilde \nabla f(y_{k+1})$, 
\item [] $\tilde{e}_{f,k} \triangleq N\mathbf{U}_{\tilde i_k}\tilde \nabla_{\tilde i_k} f(x_k,\tilde \xi_k) -\tilde \nabla f(x_k,\tilde \xi_k)$, 
\item [] ${e}_{f,k} \triangleq N\mathbf{U}_{ i_k}\tilde \nabla_{ i_k} f(y_{k+1}, \xi_k) -\tilde \nabla f(y_{k+1}, \xi_k)$.
\end{itemize}
\end{minipage}
where $\mathbf{U}_\ell \in \mathbb{R}^{n\times n_\ell}$ for $\ell \in [N]$ such that $[\mathbf{U}_1,\ldots,\mathbf{U}_N]=\mathbf{I}_n$ where $\mathbf{I}_n$ denotes the $n \times n$ identity matrix. 
\end{definition}
{Based on the above definitions,} we state some standard properties of the errors. 
\begin{lemma}[Properties of stochastic approximation and random blocks]\label{lemma:prop_rnd_blcks}\em
Consider $\tilde{e}_{F,k}$, $\tilde{e}_{f,k}$, ${e}_{F,k}$, and ${e}_{f,k}$ given by Definition \ref{def:errs}. Let Assumption \ref{assum:rnd_vars} hold. Then, the following statements hold almost surely for all $k \geq 0$

\noindent
\begin{minipage}[t]{.4\linewidth}
 \begin{itemize}
\item [(a-i)]  $\mathbb{E}[\tilde{w}_{F,k}\mid \mathcal{F}_{k-1}]=0$,
\item [(a-ii)] $\mathbb{E}[\tilde{w}_{f,k}\mid \mathcal{F}_{k-1}]=0$, 
\item [(a-iii)] $\mathbb{E}[{w}_{F,k}\mid \mathcal{F}_{k-1}\cup\{\tilde \xi_k,\tilde{i}_k\}]=0$, 
\item [(a-iv)] $\mathbb{E}[{w}_{f,k}\mid \mathcal{F}_{k-1}\cup\{\tilde \xi_k,\tilde{i}_k\}]=0$.
\end{itemize}

 \begin{itemize}
\item [(b-i)] $\mathbb{E}[\|\tilde{w}_{F,k}\|^2 \mid \mathcal{F}_{k-1}] \leq\nu_{F}^2$,
\item [(b-ii)] $\mathbb{E}[\|\tilde{w}_{f,k}\|^2\mid \mathcal{F}_{k-1}] \leq \nu_{f}^2$, 
\item [(b-iii)] $\mathbb{E}[\|{w}_{F,k}\|^2\mid \mathcal{F}_{k-1}\cup\{\tilde \xi_k,\tilde{i}_k\}]\leq \nu_{F}^2$, 
\item [(b-iv)] $\mathbb{E}[\|{w}_{f,k}\|^2\mid \mathcal{F}_{k-1}\cup\{\tilde \xi_k,\tilde{i}_k\}]\leq \nu_{f}^2$.
\end{itemize}
\end{minipage}\hfill
\begin{minipage}[t]{.6\linewidth}
 \begin{itemize}
\item [(c-i)] $\mathbb{E}[\tilde{e}_{F,k}\mid \mathcal{F}_{k-1}\cup\{\tilde \xi_k\}]=0 $,
\item [(c-ii)] $\mathbb{E}[\tilde{e}_{f,k}\mid \mathcal{F}_{k-1}\cup\{\tilde \xi_k\}]=0 $, 
\item [(c-iii)] $ \mathbb{E}[{e}_{F,k}\mid \mathcal{F}_{k-1}\cup\{\tilde \xi_k,\tilde{i}_k,\xi_k\}]=0$, 
\item [(c-iv)] $ \mathbb{E}[{e}_{f,k}\mid \mathcal{F}_{k-1}\cup\{\tilde \xi_k,\tilde{i}_k,\xi_k\}]=0 $.
\end{itemize}
 \begin{itemize}
\item [(d-i)] $\mathbb{E}[\|\tilde{e}_{F,k}\|^2\mid \mathcal{F}_{k-1}\cup\{\tilde \xi_k\}] = (N-1)\|F(x_k,\tilde \xi_k)\|^2$,
\item [(d-ii)] $\mathbb{E}[\|\tilde{e}_{f,k}\|^2\mid \mathcal{F}_{k-1}\cup\{\tilde \xi_k\}] = (N-1)\|\tilde{\nabla} f(x_k,\tilde \xi_k)\|^2$, 
\item [(d-iii)] $\mathbb{E}[\|{e}_{F,k}\|^2\mid \mathcal{F}_{k-1}\cup\{\tilde \xi_k,\tilde{i}_k,\xi_k\}]=(N-1)\|F(y_{k+1}, \xi_k)\|^2$, 
\item [(d-iv)] $\mathbb{E}[\|{e}_{f,k}\|^2\mid \mathcal{F}_{k-1}\cup\{\tilde \xi_k,\tilde{i}_k,\xi_k\}] =(N-1)\|\tilde{\nabla} f(y_{k+1}, \xi_k)\|^2$.
\end{itemize}
\end{minipage}
\end{lemma}
\begin{proof}\em 
(a) {From assumption} that $\tilde \nabla f(\bullet,\tilde \xi)$ and $F(\bullet,\tilde \xi)$ are unbiased estimators of $\tilde \nabla f(\bullet)$ and $F(\bullet)$, respectively, we have that $\mathbb{E}[\tilde{w}_{F,k}\mid \mathcal{F}_{k-1}]= \mathbb{E}[\tilde{w}_{f,k}\mid \mathcal{F}_{k-1}]=0$. Moreover, from Assumption \ref{assum:problem} {(i)}, since random samples $\tilde \xi_i$ and $\tilde i_k$ are independent from $\xi_k$, one can {conclude} that $\mathbb{E}[{w}_{F,k}\mid \mathcal{F}_{k-1}\cup\{\tilde \xi_k,\tilde{i}_k\}]=\mathbb{E}[{w}_{f,k}\mid \mathcal{F}_{k-1}\cup\{\tilde \xi_k,\tilde{i}_k\}]=0$.

\noindent (b) Using the same argument in part (a) and {invoking} Assumption~\ref{assum:problem} {(iii)}, the results follow.   

\noindent (c) {Note that} $\tilde e_{F,k}$ is the error of block-coordinate sampling of $\tilde i_k$ and since $\tilde\xi_k$ and $\tilde i_k$ are independent, {we have} that $$\mathbb{E}\left[N\mathbf{U}_{\tilde i_k}F_{\tilde i_k}(x_k,\tilde \xi_k) \mid \mathcal{F}_{k-1}\cup\{\tilde \xi_k\}\right]={1\over N}\sum_{i=1}^N N{\mathbf{U}}_{{i}}F_{{i}}(x_k,\tilde \xi_k)={F(x_k,\tilde \xi_k)}.$$ \fyy{Hence, we have} $\mathbb{E}[\tilde{e}_{F,k}\mid \mathcal{F}_{k-1}\cup\{\tilde \xi_k\}]=0$. Similarly, \fyy{we have} $\mathbb{E}[\tilde{e}_{f,k}\mid \mathcal{F}_{k-1}\cup\{\tilde \xi_k\}]=0$. Moreover, using the same argument and the fact that $i_k$ is independent \fyy{from} $\tilde\xi_k,\tilde i_k$ and $\xi_k$, \fyy{we obtain}   $$\mathbb{E}[{e}_{F,k}\mid \mathcal{F}_{k-1}\cup\{\tilde \xi_k,\tilde{i}_k,\xi_k\}]= \mathbb{E}[{e}_{f,k}\mid \mathcal{F}_{k-1}\cup\{\tilde \xi_k,\tilde{i}_k,\xi_k\}]=0.$$

\noindent (d) \fyy{We can write}  $$\mathbb{E}[\|\tilde{e}_{F,k}\|^2\mid \mathcal{F}_{k-1}\cup\{\tilde \xi_k\}]= \|F(x_k,\tilde\xi_k)\|^2+N\sum_{i=1}^N \|\fyy{\bf U_i}F_i(x_k,\tilde\xi_k)\|^2-2\fyy{\|F(x_k,\tilde\xi_k)\|^2}=(N-1)\|F(x_k,\tilde\xi_k)\|^2.$$ 
\fyy{The other relations} in part (d) can be shown using the same approach. 
\end{proof}
\begin{corollary}\label{cor:exp_terms}\em
Consider $\tilde{e}_{F,k}$, $\tilde{e}_{f,k}$, ${e}_{F,k}$, and ${e}_{f,k}$ given by Definition \ref{def:errs}. Let Assumption \ref{assum:rnd_vars} hold. Then, the following statements hold almost surely for all $k \geq 0$ 

\noindent
\begin{minipage}[t]{.4\linewidth}
 \begin{itemize}
\item [(a-i)]  $\mathbb{E}[\tilde{w}_{F,k} ]=0$,
\item [(a-ii)] $\mathbb{E}[\tilde{w}_{f,k} ]=0$, 
\item [(a-iii)] $\mathbb{E}[{w}_{F,k} ]=0$, 
\item [(a-iv)] $\mathbb{E}[{w}_{f,k} ]=0$.
\end{itemize}

 \begin{itemize}
\item [(b-i)] $\mathbb{E}[\|\tilde{w}_{F,k}\|^2  ] \leq\nu_{F}^2$,
\item [(b-ii)] $\mathbb{E}[\|\tilde{w}_{f,k}\|^2 ] \leq \nu_{f}^2$, 
\item [(b-iii)] $\mathbb{E}[\|{w}_{F,k}\|^2 ]\leq \nu_{F}^2$, 
\item [(b-iv)] $\mathbb{E}[\|{w}_{f,k}\|^2 ]\leq \nu_{f}^2$.
\end{itemize}
\end{minipage}\hfill
\begin{minipage}[t]{.6\linewidth}
 \begin{itemize}
\item [(c-i)] $\mathbb{E}[\tilde{e}_{F,k} ]=0 $,
\item [(c-ii)] $\mathbb{E}[\tilde{e}_{f,k}]=0 $, 
\item [(c-iii)] $ \mathbb{E}[{e}_{F,k} ]=0$, 
\item [(c-iv)] $ \mathbb{E}[{e}_{f,k} ]=0 $.
\end{itemize}
 \begin{itemize}
\item [(d-i)] $\mathbb{E}[\|\tilde{e}_{F,k}\|^2\ ] \leq (N-1)\fyy{(\nu_F^2+C_F^2)}$,
\item [(d-ii)] $\mathbb{E}[\|\tilde{e}_{f,k}\|^2 ] \leq (N-1)\fyy{(\nu_f^2+C_f^2)}$, 
\item [(d-iii)] $\mathbb{E}[\|{e}_{F,k}\|^2 ]\leq (N-1)\fyy{(\nu_F^2+C_F^2)}$, 
\item [(d-iv)] $\mathbb{E}[\|{e}_{f,k}\|^2 ] \leq (N-1)\fyy{(\nu_f^2+C_f^2)	}$.
\end{itemize}
\end{minipage}
\end{corollary}
\begin{proof}\em 
\fyy{The inequalities} (a-c) follow from \fyy{taking expectations on both sides of} the results in parts (a-c) of Lemma \ref{lemma:prop_rnd_blcks} \fyy{and invoking the law of total expectation}. \fyy{We can show (d-i) as follows: (i) taking expectations with respect to $\tilde{\xi_k}$ on both sides of (d-i) in Lemma \ref{lemma:prop_rnd_blcks}; (ii) applying Remark~\ref{rem:bounds2}; (iii) lastly, \fa{taking} expectations with respect to $\mathcal{F}_{k-1}$ on both sides of the resulting inequality in (ii). This will complete the proof of (d-i) in Corollary~\ref{cor:exp_terms}. Similarly, we can show (d-ii), (d-iii), and (d-iv) in Corollary~\ref{cor:exp_terms}.}  
\end{proof}

In the following lemma, we show that the update rules \eqref{eqn:y_k_update_rule} and \eqref{eqn:x_k_update_rule} in Algorithm \ref{algorithm:aR-IP-SeG} can be written compactly \fyy{in terms of} the full subgradient $\tilde \nabla f$ and map $F$ \fyy{following the terms introduced in} Definition \ref{def:errs}. 
\begin{lemma}[Compact representation of the scheme]\label{lemma:compact_alg}\em Consider Algorithm~\ref{algorithm:aR-IP-SeG}. The update rules \eqref{eqn:y_k_update_rule} and \eqref{eqn:x_k_update_rule} can be compactly written as 
\begin{align*}
&y_{k+1} = \mathcal{P}_X\left(x_k-N^{-1}\gamma_k\left(\tilde \nabla f(x_k)+\tilde{w}_{f,k}+\tilde{e}_{f,k}+\rho_kF(x_k)+\rho_k\tilde{w}_{F,k}+\rho_k\tilde{e}_{F,k}\right)\right)\\
&x_{k+1} = \mathcal{P}_X\left(x_k-N^{-1}\gamma_k\left(\tilde \nabla f(y_{k+1})+{w}_{f,k}+e_{f,k}+\rho_kF(y_{k+1})+\rho_k{w}_{F,k}+\rho_k{e}_{F,k}\right)\right).
\end{align*}
\end{lemma}

\begin{proof}\em 
Note that \fyy{in view of $X=\prod_{i=1}^N X_i$, using the definition of the Euclidean projection operator, we have that $ \mathcal{P}_X(\bullet)=\left( \mathcal{P}_{X_1}(\bullet),\ldots, \mathcal{P}_{X_N}(\bullet)\right)$}, then update rule \eqref{eqn:y_k_update_rule} can be written as $$y_{k+1}=\mathcal{P}_X\left(x_k-\gamma_k(\fyy{\bf U}_{i}\tilde \nabla_i f(x_k,\tilde \xi_k) + \rho_k F_{i}(x_{k},\tilde \xi_k))\right), \afj{\quad i=\tilde i_k.}$$ \fyy{Then} the result follows using Definition \ref{def:errs}. Similarly, one can obtain the compact form of the update rule \eqref{eqn:x_k_update_rule}. 
\end{proof}

In our analysis, we use the following properties of projection map. 
\begin{lemma}[Properties of projection mapping \cite{bertsekas2003convex}]\label{lemma:proj}\em
Let $X \subseteq \mathbb{R}^n$ be a nonempty closed convex set. 

\noindent (a) $\|\mathcal{P}_X(u) - \mathcal{P}_X(v)\| \leq \|u-v\|$ for all $u,v \in \mathbb{R}^n$.

\noindent (b) $\left(\mathcal{P}_X(u)-u\right)^T\left(x-\mathcal{P}_X(u)\right)\geq 0$ for all $u \in \mathbb{R}^n$ and $x  \in X$.
\end{lemma}
We will adopt the following \fyy{error function} to measure the quality of solution generated by Algorithm \ref{algorithm:aR-IP-SeG} \fyy{in terms of infeasibility}. 
\begin{definition}[The dual gap function \cite{marcotte1998weak}]\label{def:dual_gap}
Let $X\subseteq \mathbb{R}^n$ be a \fyy{nonempty, closed, and convex} set and $F:X\rightarrow\mathbb{R}^n$ be a \fyy{vector-valued} mapping. Then, for any $x \in X$, the dual gap function $\mathrm{Gap}^*:X\rightarrow \mathbb{R}\cup \{+\infty\} $ is defined as $ \mathrm{Gap}^*(x) \triangleq \sup_{y\in X} F(y)^T(x-y)$.
\end{definition}
\begin{remark} \em
Note that when $X \neq \emptyset$, the dual gap function is nonnegative over $X$. It is also known that when $F$ is continuous and monotone and $X$ is closed and convex, $\mathrm{Gap}^*(x^*)=0$ if and only if $x^* \in \mbox{SOL}(X,F)$ (cf.~\cite{juditsky2011solving}).
\end{remark}
\begin{lemma}[Bounds on the harmonic series~\cite{kaushik2021method}]\label{lemma:harmonic_bnds}\em
Let $0\leq \alpha <1$ be a given scalar. Then, for any integer $K \geq 2^{\frac{1}{1-\alpha}}$, we have 
\begin{align*}
\frac{K^{1-\alpha}}{2(1-\alpha)}  \leq \sum_{k=0}^{K-1}\frac{1}{(k+1)^\alpha}\leq \frac{K^{1-\alpha}}{1-\alpha}.
\end{align*}
\end{lemma}

\section{Performance analysis}\label{sec:rate}
In this section, we develop a rate and complexity analysis for \fyy{Algorithm~\ref{algorithm:aR-IP-SeG}}. We begin with showing that $\bar y_k$ generated by Algorithm \ref{algorithm:aR-IP-SeG}  is a well-defined weighted average.
\begin{lemma}[Weighted averaging]\label{lemma:ave}\em
Let $\{\bar y_k\}$ be generated by Algorithm \ref{algorithm:aR-IP-SeG}. Let us define the weights $\lambda_{k,K} \triangleq \frac{(\gamma_k\rho_k)^r}{\sum_{j=0}^{K-1} (\gamma_j\rho_j)^r}$ for $k \in \{0,\ldots, K-1\}$ and $K\geq 1$. Then, for any $K\geq 1$, we have $\bar{y}_{K} = \sum_{k=0}^{K-1} \lambda_{k,K} y_{k+1}$. Also, when $X$ is a convex set, we have $\bar y_K \in X$.
\end{lemma}
\begin{proof}\em 
We employ induction to show $\bar{y}_{K} = \sum_{k=0}^{K-1} \lambda_{k,K} y_{k+1}$ for any $K\geq 1$. For $K=1$ we have 
\begin{align*}
\sum_{k=0}^0 \fyy{\lambda_{k,1}} y_{k+1} = \lambda_{0,1}y_{1} = y_{1}, 
\end{align*}
where we used $\lambda_{0,1}=1$. Also, from the equations \eqref{eqn:averaging_eq1}--\eqref{eqn:averaging_eq2} and the initialization $\Gamma_0 =0$, we have
\begin{align*}
&\bar y_{1}:=\frac{\Gamma_0 \bar y_0+(\gamma_0\rho_0)^r y_{1}}{\Gamma_{1}}=\frac{0+(\gamma_0\rho_0)^r y_{1}}{\Gamma_{0}+\gamma_{0}^r} = y_{1}.
\end{align*}
The preceding two relations imply that the hypothesis statement holds for $K=1$. Next, suppose the relation holds for some $K\geq 1$. From the hypothesis, equations \eqref{eqn:averaging_eq1}--\eqref{eqn:averaging_eq2}, and that $\Gamma_{K}=\sum_{k=0}^{K-1}\gamma_k^r$ for all $K\geq 1$, we have
\begin{align*}
	\bar{y}_{K+1} &= \frac{\Gamma_K\bar{y}_K + (\gamma_K\rho_K)^r y_{K+1}}{\Gamma_{K+1}} 
	 = \frac{\left(\sum_{k=0}^{K-1}(\gamma_k\rho_k)^r\right)\sum_{k=0}^{K-1} \lambda_{k,K} y_{k+1}+ (\gamma_K\rho_K)^r y_{K+1}}{\Gamma_{K+1}}\\
	 &= \frac{\sum_{k=0}^{K}(\gamma_k\rho_k)^r y_{k+1}}{\sum_{j = 0}^{K}(\gamma_j\rho_j)^r} = \sum_{k=0}^{K}\left(\tfrac{(\gamma_k\rho_k)^r }{\sum_{j = 0}^{K}(\gamma_j\rho_j)^r}\right)y_{k+1}= \sum_{k=0}^{K} \lambda_{k,K+1} y_{k+1},
	\end{align*}
implying that the induction hypothesis holds for $K+1$. Thus, we conclude that the averaging formula holds for {all} $K\geq 1$. Note that since $\sum_{k=0}^{K-1} \lambda_{k,K}=1$, under the convexity of the set $X$, we have $\bar y_K \in X$. This completes the proof.
\end{proof}
Next, we prove a one-step lemma to obtain an upper bound for $F(y)^T(y_{k+1}-y) + \rho_k^{-1}(f(y_{k+1})-f(y))$ in terms of consecutive iterates and error terms. \fyy{this result will later help us obtain} upper bounds for \fyy{both the} suboptimality of the objective function and the \fyy{dual} gap function in Proposition \ref{prop:bounds}. 
\begin{lemma}[An error bound]\label{lemma:main_ineq}\em Consider Algorithm~\ref{algorithm:aR-IP-SeG} for solving problem~\eqref{prob:sopt_svi}. Let Assumptions~\ref{assum:problem} and~\ref{assum:rnd_vars} hold. Let the auxiliary stochastic sequence $\{u_k\}$ be defined recursively as
\begin{align}\label{eqn:aux_seq}
u_{k+1}\triangleq \mathcal{P}_X\left(u_k+N^{-1}\gamma_k({w}_{f,k}+e_{f,k}+\rho_k{w}_{F,k}+\rho_k{e}_{F,k})\right),
\end{align}
 where $u_0:=x_0$. Then for any arbitrary $y \in X$ and $k \geq 0$ we have
\begin{align}\label{eqn:main_07}
& (\gamma_k\rho_k)^rF(y)^T(y_{k+1}-y) + (\gamma_k\rho_k)^{r}\rho_k^{-1}(f(y_{k+1})-f(y)) \notag\\ & \leq 0.5N(\gamma_k\rho_k)^{r-1}\left(\|x_k-y\|^2  -\|x_{k+1}-y\|^2+\|u_k-y\|^2-\|u_{k+1}-y\|^2\right) \notag\\
&+2N^{-1}(\gamma_k\rho_k)^{r+1}\rho_k^{-2}\left(6C_f^2+3\|\tilde{w}_{f,k}\|^2+3\|\tilde{e}_{f,k}\|^2+4\|{w}_{f,k}\|^2+4\|e_{f,k}\|^2\right)\notag \\
&+2N^{-1}(\gamma_k\rho_k)^{r+1}\left(6C_F^2+3\|\tilde{w}_{F,k}\|^2+3\|\tilde{e}_{F,k}\|^2+4\|{w}_{F,k}\|^2+4\|{e}_{F,k}\|^2\right)\notag \\
&+\gamma_k^r\rho_k^{r-1}\left({w}_{f,k}+e_{f,k}+\rho_k{w}_{F,k}+\rho_k{e}_{F,k}\right)^T(u_k-y_{k+1}).
\end{align}
\end{lemma}
\begin{proof}\em 
Let $y \in X$ and $k\geq 0$ be arbitrary fixed values. From Lemma~\ref{lemma:compact_alg} we have
\begin{align}\label{eqn:main_01}
\|x_{k+1}-y\|^2 &= \|x_{k+1}-x_k\|^2+\|x_k-y\|^2 +2(x_{k+1}-x_k)^T(x_k-y)\notag\\
& = \|x_{k+1}-x_k\|^2+\|x_k-y\|^2 +2(x_{k+1}-x_k)^T(x_k-x_{k+1})+2(x_{k+1}-x_k)^T(x_{k+1}-y)\notag\\
& = \|x_k-y\|^2-\|x_{k+1}-x_k\|^2 +2(x_{k+1}-x_k)^T(x_{k+1}-y),
\end{align}
where the first equation is obtained by adding and subtracting $x_k$ while the third equation is implied by adding and subtracting $x_{k+1}$. In view of Lemma~\ref{lemma:proj} (b), by setting
$$u:= x_k-N^{-1}\gamma_k\left(\tilde \nabla f(y_{k+1})+{w}_{f,k}+e_{f,k}+\rho_kF(y_{k+1})+\rho_k{w}_{F,k}+\rho_k{e}_{F,k}\right),$$ 
and $x:=y$, and that we have $x_{k+1} =\mathcal{P}_X(u)$, we can write
\begin{align*}
&0 \leq \left(x_{k+1} - \left(x_k-N^{-1}\gamma_k\left(\tilde \nabla f(y_{k+1})+{w}_{f,k}+e_{f,k}+\rho_kF(y_{k+1})+\rho_k{w}_{F,k}+\rho_k{e}_{F,k}\right)\right)\right)^T(y-x_{k+1})\\
\Rightarrow \ & \afj{(x_{k+1}-x_k)^T(x_{k+1}-y)} \leq N^{-1}\gamma_k\left(\tilde \nabla f(y_{k+1})+{w}_{f,k}+e_{f,k}+\rho_kF(y_{k+1})+\rho_k{w}_{F,k}+\rho_k{e}_{F,k}\right)^T(y-x_{k+1}).
\end{align*}
Combining the preceding inequality with \eqref{eqn:main_01} we obtain
\begin{align*}
\|x_{k+1}-y\|^2 & \leq \|x_k-y\|^2-\|x_{k+1}-x_k\|^2 \notag\\
&+2N^{-1}\gamma_k\left(\tilde \nabla f(y_{k+1})+{w}_{f,k}+e_{f,k}+\rho_kF(y_{k+1})+\rho_k{w}_{F,k}+\rho_k{e}_{F,k}\right)^T(y-x_{k+1}).
\end{align*}
Note that we have
\begin{align*} 
\|x_{k+1}-x_k\|^2 &= \|x_{k+1}-y_{k+1}\|^2 +\|y_{k+1}-x_k\|^2 + 2(x_{k+1}-y_{k+1})^T(y_{k+1}-x_k).
\end{align*}
From the two preceding relations we obtain
\begin{align}\label{eqn:main_02}
\|x_{k+1}-y\|^2 & \leq \|x_k-y\|^2-\|x_{k+1}-y_{k+1}\|^2 -\|y_{k+1}-x_k\|^2 - 2(x_{k+1}-y_{k+1})^T(y_{k+1}-x_k) \notag\\
&+2N^{-1}\gamma_k\left(\tilde \nabla f(y_{k+1})+{w}_{f,k}+e_{f,k}+\rho_kF(y_{k+1})+\rho_k{w}_{F,k}+\rho_k{e}_{F,k}\right)^T(y-x_{k+1}).
\end{align}
Next we find an upper bound on the term $- 2(x_{k+1}-y_{k+1})^T(y_{k+1}-x_k)$. In view of Lemma \ref{lemma:proj} (b), by setting
$$u:= x_k-N^{-1}\gamma_k\left(\tilde \nabla f(x_k)+\tilde{w}_{f,k}+\tilde{e}_{f,k}+\rho_kF(x_k)+\rho_k\tilde{w}_{F,k}+\rho_k\tilde{e}_{F,k}\right),$$ 
and $x:=x_{k+1}$, and \fyy{in view of} $y_{k+1} =\mathcal{P}_X(u)$, we have
\begin{align*}
&0 \leq \left(y_{k+1} - \left(x_k-N^{-1}\gamma_k\left(\tilde \nabla f(x_k)+\tilde{w}_{f,k}+\tilde{e}_{f,k}+\rho_kF(x_k)+\rho_k\tilde{w}_{F,k}+\rho_k\tilde{e}_{F,k}\right)\right)\right)^T(x_{k+1}-y_{k+1})\\
\Rightarrow \ & - (x_{k+1}-y_{k+1})^T(y_{k+1}-x_k) \leq N^{-1}\gamma_k\left(\tilde \nabla f(x_k)+\tilde{w}_{f,k}+\tilde{e}_{f,k}+\rho_kF(x_k)+\rho_k\tilde{w}_{F,k}+\rho_k\tilde{e}_{F,k}\right)^T(x_{k+1}-y_{k+1}).
\end{align*}
From the preceding inequality and \eqref{eqn:main_02} we obtain
\begin{align*}
\|x_{k+1}-y\|^2 & \leq \|x_k-y\|^2-\|x_{k+1}-y_{k+1}\|^2 -\|y_{k+1}-x_k\|^2  \notag\\
& +2N^{-1}\gamma_k\left(\tilde \nabla f(x_k)+\tilde{w}_{f,k}+\tilde{e}_{f,k}+\rho_kF(x_k)+\rho_k\tilde{w}_{F,k}+\rho_k\tilde{e}_{F,k}\right)^T(x_{k+1}-y_{k+1})\notag \\
&+2N^{-1}\gamma_k\left(\tilde \nabla f(y_{k+1})+{w}_{f,k}+e_{f,k}+\rho_kF(y_{k+1})+\rho_k{w}_{F,k}+\rho_k{e}_{F,k}\right)^T(y-x_{k+1}).
\end{align*}
We further obtain
\begin{align*} 
\|x_{k+1}-y\|^2 & \leq \|x_k-y\|^2-\|x_{k+1}-y_{k+1}\|^2 -\|y_{k+1}-x_k\|^2  \notag\\
& +2N^{-1}\gamma_k\left(\tilde \nabla f(x_k)+\tilde{w}_{f,k}+\tilde{e}_{f,k}+\rho_kF(x_k)+\rho_k\tilde{w}_{F,k}+\rho_k\tilde{e}_{F,k}\right.\notag\\
&\left.-\tilde \nabla f(y_{k+1})-{w}_{f,k}-e_{f,k}-\rho_kF(y_{k+1})-\rho_k{w}_{F,k}-\rho_k{e}_{F,k}\right)^T(x_{k+1}-y_{k+1})\notag \\
&+2N^{-1}\gamma_k\left(\tilde \nabla f(y_{k+1})+{w}_{f,k}+e_{f,k}+\rho_kF(y_{k+1})+\rho_k{w}_{F,k}+\rho_k{e}_{F,k}\right)^T(y-y_{k+1}).
\end{align*}
Recall that for any {$a,b \in \mathbb{R}^n$, we have $2a^Tb \leq \|a\|^2+\|b\|^2$}. We obtain
\begin{align}\label{eqn:main_03}
\|x_{k+1}-y\|^2 & \leq \|x_k-y\|^2 -\|y_{k+1}-x_k\|^2  \notag\\
& +N^{-2}\gamma_k^2\left\|\tilde \nabla f(x_k)+\tilde{w}_{f,k}+\tilde{e}_{f,k}+\rho_kF(x_k)+\rho_k\tilde{w}_{F,k}+\rho_k\tilde{e}_{F,k}\right.\notag\\
&\left.-\tilde \nabla f(y_{k+1})-{w}_{f,k}-e_{f,k}-\rho_kF(y_{k+1})-\rho_k{w}_{F,k}-\rho_k{e}_{F,k}\right\|^2\notag \\
&+2N^{-1}\gamma_k\left(\tilde \nabla f(y_{k+1})+{w}_{f,k}+e_{f,k}+\rho_kF(y_{k+1})+\rho_k{w}_{F,k}+\rho_k{e}_{F,k}\right)^T(y-y_{k+1}).
\end{align}
Note that we can write 
\begin{align*}
&\left\|\tilde \nabla f(x_k)+\tilde{w}_{f,k}+\tilde{e}_{f,k}+\rho_kF(x_k)+\rho_k\tilde{w}_{F,k}+\rho_k\tilde{e}_{F,k}-\tilde \nabla f(y_{k+1})-{w}_{f,k}-e_{f,k}-\rho_kF(y_{k+1})-\rho_k{w}_{F,k}-\rho_k{e}_{F,k}\right\|^2 \\
&\leq 12\|\tilde \nabla f(x_k)\|^2+12\|\tilde \nabla f(y_{k+1})\|^2 + 12\rho_k^2\|F(x_k)\|^2+12\rho_k^2\|F(y_{k+1})\|^2+12\Delta_f +12\rho_k^2\Delta_F,
\end{align*}
where $\Delta_f \triangleq \|\tilde{w}_{f,k}\|^2+\|\tilde{e}_{f,k}\|^2+\|{w}_{f,k}\|^2+\|e_{f,k}\|^2$ and $\Delta_F \triangleq \|\tilde{w}_{F,k}\|^2+\|\tilde{e}_{F,k}\|^2+\|{w}_{F,k}\|^2+\|{e}_{F,k}\|^2$. In view of Remark~\ref{rem:bounds} we have
\begin{align*}
&\left\|\tilde \nabla f(x_k)+\tilde{w}_{f,k}+\tilde{e}_{f,k}+\rho_kF(x_k)+\rho_k\tilde{w}_{F,k}+\rho_k\tilde{e}_{F,k}-\tilde \nabla f(y_{k+1})-{w}_{f,k}-e_{f,k}-\rho_kF(y_{k+1})-\rho_k{w}_{F,k}-\rho_k{e}_{F,k}\right\|^2 \\
&\leq 24C_f^2+24\rho_k^2C_F^2+12\Delta_f +12\rho_k^2\Delta_F.
\end{align*}
From the preceding inequality and \eqref{eqn:main_03}, dropping the non-positive term $-\|y_{k+1}-x_k\|^2$ we have
\begin{align}\label{eqn:main_04}
\|x_{k+1}-y\|^2 & \leq \|x_k-y\|^2  +N^{-2}\gamma_k^2\left(24C_f^2+24\rho_k^2C_F^2+12\Delta_f +12\rho_k^2\Delta_F\right)\notag \\
&+2N^{-1}\gamma_k\left(\tilde \nabla f(y_{k+1})+{w}_{f,k}+e_{f,k}+\rho_kF(y_{k+1})+\rho_k{w}_{F,k}+\rho_k{e}_{F,k}\right)^T(y-y_{k+1}).
\end{align}
Note that from the convexity of $f$ we have that $\tilde \nabla f(y_{k+1})^T(y-y_{k+1}) \leq f(y)-f(y_{k+1})$. Also, the monotonicity of $F$ implies that $F(y_{k+1})^T(y-y_{k+1}) \leq F(y)^T(y-y_{k+1})$. Multiplying both sides of~\eqref{eqn:main_04} by $0.5N$, for all $y \in X$ and $k \geq 0$ we have
\begin{align}\label{eqn:main_05}
 \gamma_k\rho_kF(y)^T(y_{k+1}-y) + \gamma_k(f(y_{k+1})-f(y)) &\leq 0.5N\left(\|x_k-y\|^2  -\|x_{k+1}-y\|^2\right) \notag\\
&+N^{-1}\gamma_k^2\left(12C_f^2+12\rho_k^2C_F^2+6\Delta_f +6\rho_k^2\Delta_F\right)\notag \\
&+\gamma_k\left({w}_{f,k}+e_{f,k}+\rho_k{w}_{F,k}+\rho_k{e}_{F,k}\right)^T(y-y_{k+1}),
\end{align}
Let us now consider the auxiliary sequence $\{u_k\}$ given by Lemma~\ref{lemma:main_ineq}. Invoking Lemma~\ref{lemma:proj} (a) we can write
\begin{align*}
\|u_{k+1}-y\|^2 &= \left\|\mathcal{P}_X\left(u_k+N^{-1}\gamma_k({w}_{f,k}+e_{f,k}+\rho_k{w}_{F,k}+\rho_k{e}_{F,k})\right)-\mathcal{P}_X(y)\right\|^2\\
&\leq \| u_k+N^{-1}\gamma_k({w}_{f,k}+e_{f,k}+\rho_k{w}_{F,k}+\rho_k{e}_{F,k})-y\|^2\\
& = \|u_k-y\|^2+N^{-2}\gamma_k^2\|{w}_{f,k}+e_{f,k}+\rho_k{w}_{F,k}+\rho_k{e}_{F,k}\|^2\\ 
&+2N^{-1}\gamma_k({w}_{f,k}+e_{f,k}+\rho_k{w}_{F,k}+\rho_k{e}_{F,k})^T(u_k-y)\\
&\leq  \|u_k-y\|^2+4N^{-2}\gamma_k^2\|{w}_{f,k}\|^2+4N^{-2}\gamma_k^2\|{e}_{f,k}\|^2+4N^{-2}\gamma_k^2\rho_k^2\|{w}_{F,k}\|^2+4N^{-2}\gamma_k^2\rho_k^2\|{e}_{F,k}\|^2\\
&+2N^{-1}\gamma_k({w}_{f,k}+e_{f,k}+\rho_k{w}_{F,k}+\rho_k{e}_{F,k})^T(u_k-y).
\end{align*}
Rearranging the terms in the preceding inequality and multiplying the both sides by $0.5N$ we obtain
\begin{align}\label{eqn:main_06}
0&\leq  0.5N\left(\|u_k-y\|^2-\|u_{k+1}-y\|^2\right) +2N^{-1}\gamma_k^2\|{w}_{f,k}\|^2+2N^{-1}\gamma_k^2\|{e}_{f,k}\|^2+2N^{-1}\gamma_k^2\rho_k^2\|{w}_{F,k}\|^2\notag\\
&+2N^{-1}\gamma_k^2\rho_k^2\|{e}_{F,k}\|^2+\gamma_k({w}_{f,k}+e_{f,k}+\rho_k{w}_{F,k}+\rho_k{e}_{F,k})^T(u_k-y).
\end{align}
Summing the inequities \eqref{eqn:main_05} and \eqref{eqn:main_06} we have
\begin{align*}
 \gamma_k\rho_kF(y)^T(y_{k+1}-y) + \gamma_k(f(y_{k+1})-f(y)) &\leq 0.5N\left(\|x_k-y\|^2  -\|x_{k+1}-y\|^2+\|u_k-y\|^2-\|u_{k+1}-y\|^2\right) \notag\\
&+2N^{-1}\gamma_k^2\left(6C_f^2+3\|\tilde{w}_{f,k}\|^2+3\|\tilde{e}_{f,k}\|^2+4\|{w}_{f,k}\|^2+4\|e_{f,k}\|^2\right)\notag \\
&+2N^{-1}\gamma_k^2\rho_k^2\left(6C_F^2+3\|\tilde{w}_{F,k}\|^2+3\|\tilde{e}_{F,k}\|^2+4\|{w}_{F,k}\|^2+4\|{e}_{F,k}\|^2\right)\notag \\
&+\gamma_k\left({w}_{f,k}+e_{f,k}+\rho_k{w}_{F,k}+\rho_k{e}_{F,k}\right)^T(u_k-y_{k+1}).
\end{align*}
Multiplying both sides of the preceding inequality by $(\gamma_k\rho_k)^{r-1}$, we obtain the inequality \eqref{eqn:main_07}.
\end{proof}
\fyy{In the following result, we show that one of the error terms that appear in the inequality \eqref{eqn:main_07} has a zero mean. This result will help us with obtaining the convergence rates for Algorithm~\ref{algorithm:aR-IP-SeG}.} 
\begin{lemma}\label{lemma:aux_exp_zero}\em
Consider the auxiliary sequence defined by \eqref{eqn:aux_seq}. Let Assumptions~\ref{assum:problem} and~\ref{assum:rnd_vars} hold. Then for any $k\geq 0$ we have
\begin{align*}
\mathbb{E}\left[\left({w}_{f,k}+e_{f,k}+\rho_k{w}_{F,k}+\rho_k{e}_{F,k}\right)^T(u_k-y_{k+1})\right]=0.
\end{align*}
\end{lemma}
\begin{proof}\em 
Consider $\{u_k\}$ defined by \eqref{eqn:aux_seq}. From this definition and Algorithm~\ref{algorithm:aR-IP-SeG} we observe that $u_k$ is \fyy{$\mathcal{F}_{k-1}$-measurable}. Also, note that $y_{k+1}$ is $\mathcal{F}_{k-1}\cup\{\tilde{\xi}_k,\tilde{i}_k\}$-measurable. We can write
\begin{align}\label{eqn:lem_u_zero_1}
&\mathbb{E}\left[\left({w}_{f,k}+e_{f,k}+\rho_k{w}_{F,k}+\rho_k{e}_{F,k}\right)^T(u_k-y_{k+1})\mid \mathcal{F}_{k-1}\cup\{\tilde{\xi}_k,\tilde{i}_k\}\right]\notag\\
&=\mathbb{E}\left[\left({w}_{f,k}+e_{f,k}+\rho_k{w}_{F,k}+\rho_k{e}_{F,k}\right)\mid \mathcal{F}_{k-1}\cup\{\tilde{\xi}_k,\tilde{i}_k\}\right]^T(u_k-y_{k+1}).
\end{align}
Note that from Lemma \ref{lemma:prop_rnd_blcks} (a) we have
\begin{align}\label{eqn:lem_u_zero_2}
\mathbb{E}[{w}_{f,k}+\rho_k{w}_{F,k}\mid \mathcal{F}_{k-1}\cup\{\tilde{\xi}_k,\tilde{i}_k\}]=0.
\end{align}
 We also have from Lemma \ref{lemma:prop_rnd_blcks} (c) that
\begin{align*}
\mathbb{E}[e_{f,k}+\rho_k{e}_{F,k}\mid \mathcal{F}_{k-1}\cup\{\tilde{\xi}_k,\tilde{i}_k,\xi_k\}]=0.
\end{align*}
Taking conditional \fyy{expectations} with respect to $\xi_k$ \fyy{on} both sides of the preceding equation, we obtain
\begin{align*}
\mathbb{E}[e_{f,k}+\rho_k{e}_{F,k}\mid \mathcal{F}_{k-1}\cup\{\tilde{\xi}_k,\tilde{i}_k\}]=0.
\end{align*}
Combining the preceding relation with \eqref{eqn:lem_u_zero_1} and \eqref{eqn:lem_u_zero_2}, we have that
\begin{align*}
&\mathbb{E}\left[\left({w}_{f,k}+e_{f,k}+\rho_k{w}_{F,k}+\rho_k{e}_{F,k}\right)^T(u_k-y_{k+1})\mid \mathcal{F}_{k-1}\cup\{\tilde{\xi}_k,\tilde{i}_k\}\right]=0.
\end{align*}
Taking conditional \fyy{expectations} with respect to $\mathcal{F}_{k-1}\cup\{\tilde{\xi}_k,\tilde{i}_k\}$ \fyy{on}  both sides of the preceding relation, we obtain the result. 
\end{proof}
\fyy{In the following, we employ the results of Lemmas~\ref{lemma:main_ineq} and \ref{lemma:aux_exp_zero} to obtain upper bounds on the suboptimality of the objective function and the dual gap function associated with the stochastic VI constraint in problem~\eqref{prob:sopt_svi}. This will prepare us to analyze the convergence speed of Algorithm~\ref{algorithm:aR-IP-SeG} later in Theorem~\ref{thm:rates}.}
\begin{proposition}[Error bounds]\label{prop:bounds}\em
Consider Algorithm~\ref{algorithm:aR-IP-SeG} for solving problem~\eqref{prob:sopt_svi}. Let Assumptions~\ref{assum:problem} and~\ref{assum:rnd_vars} hold. Suppose $\{\gamma_k\rho_k\}$ is nonincreasing, $\{\rho_k\}$ is nondecreasing, and $0\leq r<1$ is a scalar. The following results hold for \fyy{all} $K\geq 2$
\begin{align}
&\mathbb{E}[f(\bar{y}_K)]-  f^*\leq \frac{4ND_X^2(\gamma_{K-1}\rho_{K-1})^{r-1}\rho_{K-1}+2N^{-1}\sum_{k=0}^{K-1}(\gamma_k\rho_k)^{1+r}\rho_k\left(\theta_F+\theta_f\rho_k^{-2}\right)}{\sum_{k=0}^{K-1}(\gamma_k\rho_k)^r},\label{prop:subopt_bound}\\
&\mathbb{E}[\mbox{Gap}^*(\bar{y}_K)]\leq \frac{4ND_X^2(\gamma_{K-1}\rho_{K-1})^{r-1}+2N^{-1}\sum_{k=0}^{K-1}(\gamma_k\rho_k)^{r}\left(\theta_F\gamma_k\rho_k+\theta_f\gamma_k\rho_k^{-1}+2ND_f\rho_k^{-1}\right)}{\sum_{k=0}^{K-1}(\gamma_k\rho_k)^r},\label{prop:infeas_bound}
\end{align}
where $\theta_F\triangleq (7N-1)C_F^2+\fyy{7N\nu_F^2}$ and $\theta_f\triangleq (7N-1)C_f^2+\fyy{7N\nu_f^2}$.
\end{proposition}
\begin{proof}\em 
First we show \fyy{the relation} \eqref{prop:subopt_bound}. Consider the inequality \eqref{eqn:main_07}. Let $y:=x^*$ where $x^* \in X$ is an optimal solution to the problem \eqref{prob:sopt_svi}. This implies that $x^* \in \mbox{SOL}(X, \mathbb{E}[F(\bullet,\fyy{\xi})])$ or equivalently, $F(x^*)^T(y_{k+1}-x^*) \geq 0$. We obtain
\begin{align}\label{eqn:prop_bounds_01}
(\gamma_k\rho_k)^{r}\rho_k^{-1}(f(y_{k+1})-f^*) &\leq 0.5N(\gamma_k\rho_k)^{r-1}\left(\|x_k-x^*\|^2  -\|x_{k+1}-x^*\|^2+\|u_k-x^*\|^2-\|u_{k+1}-x^*\|^2\right) \notag\\
&+2N^{-1}(\gamma_k\rho_k)^{r+1}\rho_k^{-2}\left(6C_f^2+3\|\tilde{w}_{f,k}\|^2+3\|\tilde{e}_{f,k}\|^2+4\|{w}_{f,k}\|^2+4\|e_{f,k}\|^2\right)\notag \\
&+2N^{-1}(\gamma_k\rho_k)^{r+1}\left(6C_F^2+3\|\tilde{w}_{F,k}\|^2+3\|\tilde{e}_{F,k}\|^2+4\|{w}_{F,k}\|^2+4\|{e}_{F,k}\|^2\right)\notag \\
&+\gamma_k^r\rho_k^{r-1}\left({w}_{f,k}+e_{f,k}+\rho_k{w}_{F,k}+\rho_k{e}_{F,k}\right)^T(u_k-y_{k+1}).
\end{align}
Multiplying the both sides by $\rho_k$ and then, adding and subtracting \fyy{the term} $$0.5N(\gamma_{k-1}\rho_{k-1})^{r-1}\rho_{k-1}\left(\|x_k-x^*\|^2+\|u_k-x^*\|^2 \right),$$ \fyy{we have  for all $k \geq 1$}
\begin{align}\label{eqn:prop_bounds_02}
(\gamma_k\rho_k)^{r}(f(y_{k+1})-f^*) &\leq 0.5N(\gamma_{k-1}\rho_{k-1})^{r-1}\rho_{k-1}\left(\|x_k-x^*\|^2  +\|u_k-x^*\|^2\right) \notag\\
&-0.5N(\gamma_{k}\rho_{k})^{r-1}\rho_{k}\left(\|x_{k+1}-x^*\|^2+\|u_{k+1}-x^*\|^2\right)\notag\\
&+ 0.5N\left((\gamma_{k}\rho_{k})^{r-1}\rho_k-(\gamma_{k-1}\rho_{k-1})^{r-1}\rho_{k-1}\right)\left(\|x_k-x^*\|^2  +\|u_k-x^*\|^2\right)\notag\\
&+2N^{-1}(\gamma_k\rho_k)^{r+1}\rho_k^{-1}\left(6C_f^2+3\|\tilde{w}_{f,k}\|^2+3\|\tilde{e}_{f,k}\|^2+4\|{w}_{f,k}\|^2+4\|e_{f,k}\|^2\right)\notag \\
&+2N^{-1}(\gamma_k\rho_k)^{r+1}\rho_k\left(6C_F^2+3\|\tilde{w}_{F,k}\|^2+3\|\tilde{e}_{F,k}\|^2+4\|{w}_{F,k}\|^2+4\|{e}_{F,k}\|^2\right)\notag \\
&+(\gamma_k\rho_k)^r\left({w}_{f,k}+e_{f,k}+\rho_k{w}_{F,k}+\rho_k{e}_{F,k}\right)^T(u_k-y_{k+1}).
\end{align}
Note that because $r<1$ and that $\{\gamma_k\rho_k\}$ is nonincreasing and $\{\rho_k\}$ is nondecreasing, we have $$\gamma_{k}^{r-1}\rho_k-\gamma_{k-1}^{r-1}\rho_{k-1} \geq 0.$$ Thus, in view of Remark~\ref{rem:bounds} we have
\begin{align*}
 &0.5N\left((\gamma_{k}\rho_{k})^{r-1}\rho_k-(\gamma_{k-1}\rho_{k-1})^{r-1}\rho_{k-1}\right)\left(\|x_k-x^*\|^2  +\|u_k-x^*\|^2\right) \\
& \leq 4ND_X^2\left((\gamma_{k}\rho_{k})^{r-1}\rho_k-(\gamma_{k-1}\rho_{k-1})^{r-1}\rho_{k-1}\right).
\end{align*}
Substituting the preceding bound in \eqref{eqn:prop_bounds_02} and then, summing the resulting inequality for $k=1,\ldots,K-1$ we obtain
\begin{align}\label{eqn:prop_bounds_02}
\sum_{k=1}^{K-1}(\gamma_k\rho_k)^{r}(f(y_{k+1})-f^*) &\leq 0.5N(\gamma_{0}\rho_{0})^{r-1}\rho_{0}\left(\|x_1-x^*\|^2  +\|u_1-x^*\|^2\right) \notag\\
&+4ND_X^2\left((\gamma_{K-1}\rho_{K-1})^{r-1}\rho_{K-1}-(\gamma_{0}\rho_0)^{r-1}\rho_{0}\right)\notag\\
&+2N^{-1}\sum_{k=1}^{K-1}(\gamma_k\rho_k)^{r+1}\rho_k^{-1}\left(6C_f^2+3\|\tilde{w}_{f,k}\|^2+3\|\tilde{e}_{f,k}\|^2+4\|{w}_{f,k}\|^2+4\|e_{f,k}\|^2\right)\notag \\
&+2N^{-1}\sum_{k=1}^{K-1}(\gamma_k\rho_k)^{r+1}\rho_k\left(6C_F^2+3\|\tilde{w}_{F,k}\|^2+3\|\tilde{e}_{F,k}\|^2+4\|{w}_{F,k}\|^2+4\|{e}_{F,k}\|^2\right)\notag \\
&+\sum_{k=1}^{K-1}(\gamma_k\rho_k)^{r}\left({w}_{f,k}+e_{f,k}+\rho_k{w}_{F,k}+\rho_k{e}_{F,k}\right)^T(u_k-y_{k+1}).
\end{align}
From \eqref{eqn:prop_bounds_01} for $k=0$ we have 
\begin{align}\label{eqn:prop_bounds_03}
(\gamma_0\rho_0)^{r}(f(y_{1})-f^*) &\leq 0.5N(\gamma_0\rho_0)^{r-1}\rho_0\left(\|x_0-x^*\|^2  -\|x_{1}-x^*\|^2+\|u_0-x^*\|^2-\|u_{1}-x^*\|^2\right) \notag\\
&+2N^{-1}(\gamma_0\rho_0)^{1+r}\rho_0^{-1}\left(6C_f^2+3\|\tilde{w}_{f,0}\|^2+3\|\tilde{e}_{f,0}\|^2+4\|{w}_{f,0}\|^2+4\|e_{f,0}\|^2\right)\notag \\
&+2N^{-1}(\gamma_0\rho_0)^{1+r}\rho_0\left(6C_F^2+3\|\tilde{w}_{F,0}\|^2+3\|\tilde{e}_{F,0}\|^2+4\|{w}_{F,0}\|^2+4\|{e}_{F,0}\|^2\right)\notag \\
&+(\gamma_0\rho_0)^r\left({w}_{f,0}+e_{f,0}+\rho_k{w}_{F,0}+\rho_k{e}_{F,0}\right)^T(u_0-y_{1}).
\end{align}
Summing the preceding two relations we obtain
\begin{align}\label{eqn:prop_bounds_04}
\sum_{k=0}^{K-1}(\gamma_k\rho_k)^{r}(f(y_{k+1})-f^*)  &\leq 0.5N(\gamma_0\rho_0)^{r-1}\rho_0\left(\|x_0-x^*\|^2  +\|u_0-x^*\|^2\right) \notag\\
&+4ND_X^2\left((\gamma_{K-1}\rho_{K-1})^{r-1}\rho_{K-1}-(\gamma_{0}\rho_0)^{r-1}\rho_{0}\right)\notag\\
&+2N^{-1}\sum_{k=0}^{K-1}(\gamma_k\rho_k)^{r+1}\rho_k^{-1}\left(6C_f^2+3\|\tilde{w}_{f,k}\|^2+3\|\tilde{e}_{f,k}\|^2+4\|{w}_{f,k}\|^2+4\|e_{f,k}\|^2\right)\notag \\
&+2N^{-1}\sum_{k=0}^{K-1}(\gamma_k\rho_k)^{r+1}\rho_k\left(6C_F^2+3\|\tilde{w}_{F,k}\|^2+3\|\tilde{e}_{F,k}\|^2+4\|{w}_{F,k}\|^2+4\|{e}_{F,k}\|^2\right)\notag \\
&+\sum_{k=0}^{K-1}(\gamma_k\rho_k)^{r}\left({w}_{f,k}+e_{f,k}+\rho_k{w}_{F,k}+\rho_k{e}_{F,k}\right)^T(u_k-y_{k+1}).
\end{align}
Note that from the convexity of $f$ and Lemma \ref{lemma:ave}, we have
\begin{align*}
\frac{\sum_{k=0}^{K-1}(\gamma_k\rho_k)^rf(y_{k+1})}{\sum_{k=0}^{K-1}(\gamma_k\rho_k)^r} = \sum_{k=0}^{K-1}\left(\frac{(\gamma_k\rho_k)^r}{\sum_{j=0}^{K-1}(\gamma_j\rho_j)^r}\right)f(y_{k+1}) =  \sum_{k=0}^{K-1}\lambda_{k,K}f(y_{k+1}) \geq f\left(\sum_{k=0}^{K-1} \lambda_{k,K} y_{k+1}\right) = f(\bar{y}_K).
\end{align*}
Dividing the both \fyy{sides} of \eqref{eqn:prop_bounds_04} by $\sum_{k=0}^{K-1}(\gamma_k\rho_k)^r$, using the preceding relation, and $\|x_0-x^*\|^2  +\|u_0-x^*\|^2 \leq 8D_X^2$, we obtain
\begin{align}\label{eqn:prop_bounds_05}
f(\bar{y}_K) -f^*&\leq \left(\sum_{k=0}^{K-1}(\gamma_k\rho_k)^r\right)^{-1}\left(4ND_X^2(\gamma_0\rho_0)^{r-1}\rho_0 +4ND_X^2\left((\gamma_{K-1}\rho_{K-1})^{r-1}\rho_{K-1}-(\gamma_{0}\rho_0)^{r-1}\rho_{0}\right)\right.\notag\\
&\left.+2N^{-1}\sum_{k=0}^{K-1}(\gamma_k\rho_k)^{r+1}\rho_k^{-1}\left(6C_f^2+3\|\tilde{w}_{f,k}\|^2+3\|\tilde{e}_{f,k}\|^2+4\|{w}_{f,k}\|^2+4\|e_{f,k}\|^2\right)\right.\notag \\
&\left.+2N^{-1}\sum_{k=0}^{K-1}(\gamma_k\rho_k)^{r+1}\rho_k\left(6C_F^2+3\|\tilde{w}_{F,k}\|^2+3\|\tilde{e}_{F,k}\|^2+4\|{w}_{F,k}\|^2+4\|{e}_{F,k}\|^2\right)\right.\notag \\
&\left.+\sum_{k=0}^{K-1}(\gamma_k\rho_k)^{r}\left({w}_{f,k}+e_{f,k}+\rho_k{w}_{F,k}+\rho_k{e}_{F,k}\right)^T(u_k-y_{k+1})\right).
\end{align}
Taking expectations \fyy{on} the both sides and applying Corollary \ref{cor:exp_terms} and Lemma \ref{lemma:aux_exp_zero}, we obtain
\begin{align*}
\mathbb{E}[f(\bar{y}_K)]-  f^*&\leq  \left(\sum_{k=0}^{K-1}(\gamma_k\rho_k)^r\right)^{-1}\left(4ND_X^2(\gamma_{K-1}\rho_{K-1})^{r-1}\rho_{K-1}\right. \\ & \left. +2N^{-1}\sum_{k=0}^{K-1}(\gamma_k\rho_k)^{r+1}\rho_k^{-1}\left(6C_f^2+7\nu_f^2+7(N-1)\fyy{(\nu_f^2+C_f^2)}\right)\right.\notag \\
&\left.+2N^{-1}\sum_{k=0}^{K-1}(\gamma_k\rho_k)^{r+1}\rho_k\left(6C_F^2+7\nu_F^2+7(N-1)\fyy{(\nu_F^2+C_F^2)}\right)\right).
\end{align*}
This implies that the inequality \eqref{prop:subopt_bound} holds for all $K\geq 2$. Next we show the inequality \eqref{prop:infeas_bound}. Consider the inequality \eqref{eqn:main_07} again for an arbitrary $y \in X$. In view of Remark~\ref{rem:bounds} we have $f(y_{k+1})-f(y) \leq 2D_f$. Rearranging the terms in \eqref{eqn:main_07} we obtain
\begin{align} \label{eqn:prop_bound2_01}
 (\gamma_k\rho_k)^rF(y)^T(y_{k+1}-y)  & \leq 0.5N(\gamma_k\rho_k)^{r-1}\left(\|x_k-y\|^2  -\|x_{k+1}-y\|^2+\|u_k-y\|^2-\|u_{k+1}-y\|^2\right) \notag\\
&+2N^{-1}(\gamma_k\rho_k)^{r+1}\rho_k^{-2}\left(6C_f^2+3\|\tilde{w}_{f,k}\|^2+3\|\tilde{e}_{f,k}\|^2+4\|{w}_{f,k}\|^2+4\|e_{f,k}\|^2\right)\notag \\
&+2N^{-1}(\gamma_k\rho_k)^{r+1}\left(6C_F^2+3\|\tilde{w}_{F,k}\|^2+3\|\tilde{e}_{F,k}\|^2+4\|{w}_{F,k}\|^2+4\|{e}_{F,k}\|^2\right)\notag \\
&+\gamma_k^r\rho_k^{r-1}\left({w}_{f,k}+e_{f,k}+\rho_k{w}_{F,k}+\rho_k{e}_{F,k}\right)^T(u_k-y_{k+1}) +2(\gamma_k\rho_k)^{r}\rho_k^{-1}D_f.
\end{align}
Adding and subtracting $(\gamma_k\rho_k)^{r-1}\left(\|x_k-y\|^2+\|u_k-y\|^2 \right)$, for all $k \geq 1$ we have
\begin{align}\label{eqn:prop_bound2_02}
(\gamma_k\rho_k)^{r}F(y)^T(y_{k+1}-y) &\leq 0.5N(\gamma_{k-1}\rho_{k-1})^{r-1}\left(\|x_k-y\|^2  +\|u_k-y\|^2\right) \notag\\ 
&-0.5N(\gamma_k\rho_k)^{r-1}\left(\|x_{k+1}-y\|^2+\|u_{k+1}-y\|^2\right)\notag\\
&+ 0.5N\left((\gamma_k\rho_k)^{r-1}-(\gamma_{k-1}\rho_{k-1})^{r-1}\right)\left(\|x_k-y\|^2  +\|u_k-y\|^2\right)\notag\\
&+2N^{-1}(\gamma_k\rho_k)^{r+1}\rho_k^{-2}\left(6C_f^2+3\|\tilde{w}_{f,k}\|^2+3\|\tilde{e}_{f,k}\|^2+4\|{w}_{f,k}\|^2+4\|e_{f,k}\|^2\right)\notag \\
&+2N^{-1}(\gamma_k\rho_k)^{r+1}\left(6C_F^2+3\|\tilde{w}_{F,k}\|^2+3\|\tilde{e}_{F,k}\|^2+4\|{w}_{F,k}\|^2+4\|{e}_{F,k}\|^2\right)\notag \\
&+\gamma_k^r\rho_k^{r-1}\left({w}_{f,k}+e_{f,k}+\rho_k{w}_{F,k}+\rho_k{e}_{F,k}\right)^T(u_k-y_{k+1}) +2(\gamma_k\rho_k)^{r}\rho_k^{-1}D_f.
\end{align}
Note that because $r<1$ and that $\{\gamma_k\rho_k\}$ is nonincreasing, we have $(\gamma_k\rho_k)^{r-1}-(\gamma_{k-1}\rho_{k-1})^{r-1} \geq 0$. Thus, in view of Remark~\ref{rem:bounds} we have
\begin{align*}
 0.5N\left((\gamma_k\rho_k)^{r-1}-(\gamma_{k-1}\rho_{k-1})^{r-1}\right)\left(\|x_k-x^*\|^2  +\|u_k-x^*\|^2\right) \leq 4ND_X^2\left((\gamma_k\rho_k)^{r-1}-(\gamma_{k-1}\rho_{k-1})^{r-1}\right).
\end{align*}
Substituting the preceding bound in \eqref{eqn:prop_bound2_02} and then, summing the resulting inequality for $k=1,\ldots,K-1$ we obtain
\begin{align}\label{eqn:prop_bound2_03}
&\sum_{k=1}^{K-1}(\gamma_k\rho_k)^{r}F(y)^T(y_{k+1}-y) \leq 0.5N(\gamma_{0}\rho_{0})^{r-1}\left(\|x_1-y\|^2  +\|u_1-y\|^2\right) \notag\\
& + 4ND_X^2\left((\gamma_{K-1}\rho_{K-1})^{r-1}-(\gamma_{0}\rho_{0})^{r-1}\right)\notag\\
&+2N^{-1}\sum_{k=1}^{K-1}(\gamma_k\rho_k)^{r+1}\rho_k^{-2}\left(6C_f^2+3\|\tilde{w}_{f,k}\|^2+3\|\tilde{e}_{f,k}\|^2+4\|{w}_{f,k}\|^2+4\|e_{f,k}\|^2\right)\notag \\
&+2N^{-1}\sum_{k=1}^{K-1}(\gamma_k\rho_k)^{r+1}\left(6C_F^2+3\|\tilde{w}_{F,k}\|^2+3\|\tilde{e}_{F,k}\|^2+4\|{w}_{F,k}\|^2+4\|{e}_{F,k}\|^2\right)\notag \\
&+\sum_{k=1}^{K-1}\gamma_k^r\rho_k^{r-1}\left({w}_{f,k}+e_{f,k}+\rho_k{w}_{F,k}+\rho_k{e}_{F,k}\right)^T(u_k-y_{k+1}) +2D_f\sum_{k=1}^{K-1}(\gamma_k\rho_k)^{r}\rho_k^{-1}.
\end{align}
Consider \eqref{eqn:prop_bound2_01} for $k=0$. Summing that relation with \eqref{eqn:prop_bound2_03} we have
\begin{align}\label{eqn:prop_bound2_04}
&F(y)^T\left(\sum_{k=0}^{K-1}(\gamma_k\rho_k)^{r}y_{k+1}-y\right) \leq 0.5N(\gamma_0\rho_0)^{r-1}\left(\|x_0-y\|^2  +\|u_0-y\|^2\right)  \notag\\ &   +4ND_X^2\left((\gamma_{K-1}\rho_{K-1})^{r-1}-(\gamma_{0}\rho_0)^{r-1}\right)\notag\\
&+2N^{-1}\sum_{k=0}^{K-1}(\gamma_k\rho_k)^{r+1}\rho_k^{-2}\left(6C_f^2+3\|\tilde{w}_{f,k}\|^2+3\|\tilde{e}_{f,k}\|^2+4\|{w}_{f,k}\|^2+4\|e_{f,k}\|^2\right)\notag \\
&+2N^{-1}\sum_{k=0}^{K-1}(\gamma_k\rho_k)^{r+1}\left(6C_F^2+3\|\tilde{w}_{F,k}\|^2+3\|\tilde{e}_{F,k}\|^2+4\|{w}_{F,k}\|^2+4\|{e}_{F,k}\|^2\right)\notag \\
&+\sum_{k=0}^{K-1}\gamma_k^r\rho_k^{r-1}\left({w}_{f,k}+e_{f,k}+\rho_k{w}_{F,k}+\rho_k{e}_{F,k}\right)^T(u_k-y_{k+1}) +2D_f\sum_{k=0}^{K-1}(\gamma_k\rho_k)^{r}\rho_k^{-1}.
\end{align}
Dividing the both side of \eqref{eqn:prop_bound2_04} by $\sum_{k=0}^{K-1}(\gamma_k\rho_k)^r$, invoking Lemma~\ref{lemma:ave}, and $\|x_0-y\|^2  +\|u_0-y\|^2 \leq 8D_X^2$, we obtain
\begin{align}\label{eqn:prop_bound2_05}
F(y)^T(\bar{y}_K-y) &\leq \left(\sum_{k=0}^{K-1}(\gamma_k\rho_k)^r\right)^{-1}\left(4ND_X^2(\gamma_{K-1}\rho_{K-1})^{r-1}\right.\notag\\
&\left.+2N^{-1}\sum_{k=0}^{K-1}(\gamma_k\rho_k)^{r+1}\rho_k^{-2}\left(6C_f^2+3\|\tilde{w}_{f,k}\|^2+3\|\tilde{e}_{f,k}\|^2+4\|{w}_{f,k}\|^2+4\|e_{f,k}\|^2\right)\right.\notag \\
&\left.+2N^{-1}\sum_{k=0}^{K-1}(\gamma_k\rho_k)^{r+1}\left(6C_F^2+3\|\tilde{w}_{F,k}\|^2+3\|\tilde{e}_{F,k}\|^2+4\|{w}_{F,k}\|^2+4\|{e}_{F,k}\|^2\right)\right.\notag \\
&\left.+\sum_{k=0}^{K-1}\gamma_k^r\rho_k^{r-1}\left({w}_{f,k}+e_{f,k}+\rho_k{w}_{F,k}+\rho_k{e}_{F,k}\right)^T(u_k-y_{k+1}) +2D_f\sum_{k=0}^{K-1}(\gamma_k\rho_k)^{r}\rho_k^{-1}\right).
\end{align}
Taking the supremum \fyy{on} the both sides of \eqref{eqn:prop_bound2_05} with respect to $y$ over the set $X$ and invoking Definition~\ref{def:dual_gap}, we have
\begin{align*}
\mbox{Gap}^*(\bar{y}_K) &\leq \left(\sum_{k=0}^{K-1}(\gamma_k\rho_k)^r\right)^{-1}\left(4ND_X^2(\gamma_{K-1}\rho_{K-1})^{r-1}\right.\notag\\
&\left.+2N^{-1}\sum_{k=0}^{K-1}(\gamma_k\rho_k)^{r+1}\rho_k^{-2}\left(6C_f^2+3\|\tilde{w}_{f,k}\|^2+3\|\tilde{e}_{f,k}\|^2+4\|{w}_{f,k}\|^2+4\|e_{f,k}\|^2\right)\right.\notag \\
&\left.+2N^{-1}\sum_{k=0}^{K-1}(\gamma_k\rho_k)^{r+1}\left(6C_F^2+3\|\tilde{w}_{F,k}\|^2+3\|\tilde{e}_{F,k}\|^2+4\|{w}_{F,k}\|^2+4\|{e}_{F,k}\|^2\right)\right.\notag \\
&\left.+\sum_{k=0}^{K-1}\gamma_k^r\rho_k^{r-1}\left({w}_{f,k}+e_{f,k}+\rho_k{w}_{F,k}+\rho_k{e}_{F,k}\right)^T(u_k-y_{k+1}) +2D_f\sum_{k=0}^{K-1}(\gamma_k\rho_k)^{r}\rho_k^{-1}\right).
\end{align*}
Taking expectations \fyy{on} the both sides and applying Corollary \ref{cor:exp_terms} and Lemma \ref{lemma:aux_exp_zero}, we obtain
\begin{align*}
\mathbb{E}[\mbox{Gap}^*(\bar{y}_K)] &\leq \left(\sum_{k=0}^{K-1}(\gamma_k\rho_k)^r\right)^{-1}\left(4ND_X^2(\gamma_{K-1}\rho_{K-1})^{r-1}\right. \\
& \left. +2N^{-1}\sum_{k=0}^{K-1}(\gamma_k\rho_k)^{r+1}\rho_k^{-2}\left(6C_f^2+7\nu_f^2+7(N-1)\fyy{(\nu_f^2+C_f^2)}\right)\right.\notag \\
&\left.+2N^{-1}\sum_{k=0}^{K-1}(\gamma_k\rho_k)^{r+1}\left(6C_F^2+7\nu_F^2+7(N-1)\fyy{(\nu_F^2+C_F^2)}\right)+2D_f\sum_{k=0}^{K-1}(\gamma_k\rho_k)^{r}\rho_k^{-1}\right).
\end{align*}
Hence, we obtain the infeasibility \fyy{bound given by} \eqref{prop:infeas_bound}. 
\end{proof}
\fyy{The main result of this section is presented in the following theorem where we obtain convergence rates for solving problem~\eqref{prob:sopt_svi}. In particular, we} specify \fyy{update rules for} \fyy{stepsize} $\gamma_k$ and \fyy{penalty} parameter $\rho_k$ to \fyy{guarantee this performance for} Algorithm \ref{algorithm:aR-IP-SeG}. 
\begin{theorem}[Rate statements and iteration complexity guarantees]\label{thm:rates}\em
Consider Algorithm~\ref{algorithm:aR-IP-SeG} applied to problem~\eqref{prob:sopt_svi}. Suppose $r \in [0,1)$ is an arbitrary scalar. Let Assumptions~\ref{assum:problem} and~\ref{assum:rnd_vars} hold. Suppose, for any $k\geq 0$, the stepsize and the penalty sequence are given by
\begin{align*}
\gamma_k \triangleq \frac{\gamma_0}{\sqrt[4]{(k+1)^3}} \quad \hbox{and } \quad \rho_k \triangleq \rho_0\sqrt[4]{k+1}. 
\end{align*}
Then, for all $K\geq 2^{\frac{2}{1-r}}$ the following statements hold.

\noindent {\bf (i)}  The convergence rate in terms of the suboptimality is \fa{given as}
\begin{align*}
\mathbb{E}[f(\bar{y}_K)]-  f^*&\leq \left(\tfrac{D_X^2}{\gamma_0\rho_0}+\tfrac{\gamma_0\rho_0\left((7-N^{-1})C_F^2+\fyy{7\nu_F^2}+\tfrac{(7-N^{-1})C_f^2+\fyy{7\nu_f^2}}{\rho_0^2}\right)}{(1.5-r)N}\right)\frac{4\rho_0(2-r)N }{\sqrt[4]{K}}.
\end{align*}
 
\noindent {\bf (ii)}  The convergence rate in terms of the infeasibility is \fa{given as}
\begin{align*}
\mathbb{E}[\mbox{Gap}^*(\bar{y}_K)]&\leq \left(\tfrac{D_X^2}{\gamma_0\rho_0\sqrt[4]{K}}+\tfrac{\gamma_0\rho_0\left((7-N^{-1})C_F^2+\fyy{7\nu_F^2}+\tfrac{(7-N^{-1})C_f^2+\fyy{7\nu_f^2}}{\rho_0^2}\right)}{(1-r)N\sqrt[4]{K}}+\tfrac{D_fN^{-1}}{\rho_0(0.75-0.5r)}\right)\frac{4(2-r)N }{\sqrt[4]{K}}.
\end{align*}
\noindent {\bf (iii)} \fyy{Given $\epsilon>0$, let} $K_\epsilon$ \fyy{denote} a \fyy{deterministic} integer to achieve $\mathbb{E}[f(\bar{y}_{K_\epsilon})]-  f^* \leq \epsilon$ and $\mathbb{E}[\mbox{Gap}^*(\bar{y}_{K_\epsilon})] \leq \epsilon$. Then the total iteration complexity and also, the total sample complexity of Algorithm~\ref{algorithm:aR-IP-SeG} are the same and \fyy{are} $\mathcal{O}(N^4\epsilon^{-4})$ where $N$ denotes the number of blocks \fyy{(In particular, in the Nash game, $N$ denotes the number of players)}.
\end{theorem}
\begin{proof}\em  (i) Substituting the update rules of $\gamma_k$ and $\rho_k$ in \eqref{prop:subopt_bound}, we obtain
\begin{align*}
\mathbb{E}[f(\bar{y}_K)]-  f^*&\leq \frac{4ND_X^2(\gamma_{K-1}\rho_{K-1})^{r-1}\rho_{K-1}+2N^{-1}\sum_{k=0}^{K-1}(\gamma_k\rho_k)^{1+r}\rho_k\left(\theta_F+\theta_f\rho_k^{-2}\right)}{\sum_{k=0}^{K-1}(\gamma_k\rho_k)^r}\\
&\leq 
\frac{4ND_X^2\rho_0(\gamma_0\rho_0)^{r-1}K^{0.75-0.5r}+2N^{-1}\rho_0\left(\theta_F+\theta_f\rho_0^{-2}\right)(\gamma_0\rho_0)^{1+r}\sum_{k=0}^{K-1}(k+1)^{-(0.25+0.5r)}}{(\gamma_0\rho_0)^r\sum_{k=0}^{K-1}(k+1)^{-0.5r}}.
\end{align*}
Because $0\leq r<1$, note that both the terms $0.25+0.5r$ and $0.5r$ are nonnegative and smaller than $1$. This implies that the conditions of Lemma~\ref{lemma:harmonic_bnds} are met. Employing the bounds provided by Lemma~\ref{lemma:harmonic_bnds}, from the preceding inequality we have
\begin{align*}
\mathbb{E}[f(\bar{y}_K)]-  f^*&\leq  \frac{4ND_X^2\rho_0(\gamma_0\rho_0)^{r-1}K^{0.75-0.5r}+2N^{-1}\rho_0\left(\theta_F+\theta_f\rho_0^{-2}\right)(\gamma_0\rho_0)^{1+r}(0.75-0.5r)^{-1}K^{0.75-0.5r}}{0.5(1-0.5r)^{-1}(\gamma_0\rho_0)^rK^{1-0.5r}}\\
&=\frac{(2-r)\left(4ND_X^2\rho_0(\gamma_0\rho_0)^{-1}+2N^{-1}\rho_0\left(\theta_F+\theta_f\rho_0^{-2}\right)(\gamma_0\rho_0)(0.75-0.5r)^{-1}\right)}{K^{0.25}}.
\end{align*}
Substituting $\theta_f$ and $\theta_F$ by their values and then, rearranging the terms we obtain the desired rate statement in (i).

\noindent (ii) Next we derive the non-asymptotic rate statement in terms of the infeasibility. Substituting the update rules of $\gamma_k$ and $\rho_k$ in \eqref{prop:infeas_bound}, and noting that $\gamma_k$ and $\rho_k^{-1}$ are nonincreasing, we obtain
\begin{align*}
\mathbb{E}[\mbox{Gap}^*(\bar{y}_K)]&\leq \frac{4ND_X^2(\gamma_{K-1}\rho_{K-1})^{r-1}+2N^{-1}\sum_{k=0}^{K-1}(\gamma_k\rho_k)^{r}\left(\theta_F\gamma_k\rho_k+\theta_f\gamma_k\rho_k^{-1}+2ND_f\rho_k^{-1}\right)}{\sum_{k=0}^{K-1}(\gamma_k\rho_k)^r}\\
&\leq \frac{4ND_X^2(\gamma_{K-1}\rho_{K-1})^{r-1}+2N^{-1}(\theta_F+\theta_f\rho_0^{-2})\sum_{k=0}^{K-1}(\gamma_k\rho_k)^{r+1}+4D_f\sum_{k=0}^{K-1}(\gamma_k\rho_k)^{r}\rho_k^{-1}}{\sum_{k=0}^{K-1}(\gamma_k\rho_k)^r}\\
&\leq \frac{4ND_X^2(\gamma_0\rho_0K^{-0.5})^{r-1}+2N^{-1}\left(\theta_F+\theta_f\rho_0^{-2}\right)(\gamma_0\rho_0)^{1+r}\sum_{k=0}^{K-1}(k+1)^{-0.5(1+r)}}{(\gamma_0\rho_0)^r\sum_{k=0}^{K-1}(k+1)^{-0.5r}}\\
&+\frac{4D_f(\gamma_0\rho_0)^{r}\rho_0^{-1}\sum_{k=0}^{K-1}(k+1)^{-0.5r-0.25}}{(\gamma_0\rho_0)^r\sum_{k=0}^{K-1}(k+1)^{-0.5r}}.
\end{align*}
Employing the bounds provided by Lemma~\ref{lemma:harmonic_bnds}, from the preceding inequality we have
\begin{align*}
\mathbb{E}[\mbox{Gap}^*(\bar{y}_K)]&\leq \frac{4ND_X^2(\gamma_0\rho_0)^{-1}K^{-0.5(r-1)}+2N^{-1}\left(\theta_F+\theta_f\rho_0^{-2}\right)(\gamma_0\rho_0)(1-0.5(1+r))^{-1}K^{1-0.5(1+r)}}{0.5(1-0.5r)^{-1}K^{1-0.5r}}\\
&+\frac{4D_f\rho_0^{-1}(1-0.5r-0.25)^{-1}K^{1-0.5r-0.25}}{0.5(1-0.5r)^{-1}K^{1-0.5r}}\\
& \leq (2-r)\frac{4ND_X^2(\gamma_0\rho_0)^{-1}+4N^{-1}\left(\theta_F+\theta_f\rho_0^{-2}\right)(\gamma_0\rho_0)(1-r)^{-1}}{K^{0.5}}\\
&+(2-r)\frac{4D_f\rho_0^{-1}(0.75-0.5r)^{-1}}{K^{0.25}}.
\end{align*}
The rate statement in (ii) can be obtained by substituting $\theta_f$ and $\theta_F$ by their values and then, rearranging the terms.

\noindent (iii) The result of part (iii) holds directly from the rate statements in parts (i) and (ii). 
\end{proof}

\section{Approximating the price of stability}\label{sec:pos}
Our goal in this section lies in devising a stochastic scheme for \fyy{approximating} the price of stability, defined by \eqref{pos}, in monotone stochastic Nash games. The proposed scheme includes three main steps described as follows\fa{:} 

\noindent {\bf (i)} Employing Algorithm~\ref{algorithm:aR-IP-SeG} for approximating a solution to the optimization problem \eqref{prob:sopt_svi}. 

\noindent {\bf (ii)} Employing a stochastic approximation method for approximating a solution to the nonsmooth stochastic optimization problem $\min_{x\in X}\mathbb{E}[f(x,\xi)]$. This can be done through a host of well-known methods including the stochastic \fa{subgradient}~\cite{nemirovski_robust_2009,Farzad1} and its accelerated smoothed variants~\cite{jalilzadeh2018smoothed}. Another avenue for solving this class of problems is stochastic \fyy{extra-subgradient} methods~\cite{juditsky2011solving,Nem04,yousefian2014optimal,yousefian2018stochastic,iusem2017extragradient}.

\noindent {\bf (iii)} Lastly, given the two approximate optimal solutions in (i) and (ii), we estimate the objective function value $\mathbb{E}[f(x,\xi)]$ at each solution. The PoS is then \fyy{approximated} by dividing the \fyy{sample average approximation of optimal} objective value of problem \eqref{prob:sopt_svi} by that of $\min_{x\in X}\mathbb{E}[f(x,\xi)]$.

An example of this scheme is presented by Algorithm~\ref{algorithm:pos}. Here, vectors $y_{k,1}$ and $x_{k,1}$ are generated by Algorithm~\ref{algorithm:aR-IP-SeG}, while $y_{k,2}$ and $x_{k,2}$ are generated by a standard stochastic extra-subgradient method for solving $\min_{x\in X}\mathbb{E}[f(x,\xi)]$. \fyy{We provide} the following remark to make clarifications about this \fyy{scheme}. 
\begin{remark}\em 
As mentioned earlier, we do have several options in employing a method for solving the canonical nonsmooth stochastic optimization problem $\min_{x\in X}\mathbb{E}[f(x,\xi)]$. Here, we use the stochastic extra-subgradient method that is known to achieve the convergence rate of the order $\frac{1}{\sqrt{K}}$ when employing a suitable weighted averaging scheme specified by \eqref{eqn:averaging_pos_eq2} (cf.~\cite{yousefian2018stochastic}). We also note that Algorithm~\ref{algorithm:pos} can be compactly presented \fyy{by the} two extra-subgradient schemes, separately. However, we note that there are different groups of random samples generated in Algorithm~\ref{algorithm:pos} and the analysis \fyy{of the scheme} relies on what assumptions we make on these samples, \fyy{presented in the following}. 
\end{remark}
\begin{assumption}\label{assum:vars_alg_pos}\em Let the following statements hold.

\noindent (i) The random samples $\{ \xi_{k,1}\}_{k=0}^{K-1}$, $\{\tilde \xi_{k,1}\}_{k=0}^{K-1}$, $\{\xi_{k,2}\}_{k=0}^{K-1}$, $\fyy{\{\tilde \xi_{k,2}\}_{k=0}^{K-1}}$, \fyy{and $\{ \zeta_{t}\}_{t=0}^{M-1}$} are i.i.d. associated with the probability space $(\Omega, \mathcal{F},\mathbb{P})$. Also, $\{\tilde i_{k,1}\}_{k=0}^{K-1}$, $\{\tilde i_{k,1}\}_{k=0}^{K-1}$, $\{i_{k,2}\}_{k=0}^{K-1}$, and $\{\tilde i_{k,2}\}_{k=0}^{K-1}$ are i.i.d. uniformly distributed within the range $\{1,\ldots,N\}$. Additionally, all the aforementioned random variables are independent from each other. 
 
\noindent (ii) $f(\bullet,\xi)$ is an unbiased estimator of the deterministic function $f(\bullet)$.
\end{assumption}

{To approximate the \fa{PoS}, we need upper and lower bounds for suboptimality of problem~ \eqref{prob:sopt_svi}. We established the upper bound in Theorem \ref{thm:rates}. Now we obtain the lower bound considering the following weak sharpness assumption. 

\begin{assumption}[Weak Sharpness \cite{cui2016analysis}]\label{sharp}\em
The variational inequality problem VI(X,F) satisfies the weak sharpness property implying that there exists an $\alpha > 0$ such that $(x-x^*)^T F (x^*) \geq \alpha \mbox{dist} (x, X^*)$ \fa{ for any $x \in X^*$, where $X^*$ denotes the solution set of VI$(X,F)$}.
\end{assumption}

\begin{corollary}\label{lower bound}\em 
Under the premises of Theorem \ref{thm:rates} and considering Assumption \ref{sharp}, \fa{we have for all $K\geq 2$}
$$-\frac{\mathcal O(N)}{\sqrt[4]{K}}\leq \fa{\mathbb{E}}[f(\bar y_K)-f^*]\leq \frac{\mathcal O(N)}{\sqrt[4]{K}}.$$
\end{corollary}
\begin{proof}\em
From Assumption \ref{sharp}, we know that there exists $\alpha>0$ such that $\mathbb E[\mbox{dist}(\bar y_K,X^*)]\leq \frac{1}{\alpha}\mathbb{E}[\mbox{Gap}^*(\bar{y}_K)]$. Moreover, since $X^*$
 is a compact set, there exists $\hat y^*\in X^*$ such that $\mbox{dist}(\bar y_K,X^*)=\min_{y\in X^*}\|y-\bar y_k\|=\|\hat y^*-\bar y_K\|$. Therefore, using the result of Theorem \ref{thm:rates}, we have
 \begin{align}\label{lower gap}\mathbb E[\|\hat y^*-\bar y_K\|]\leq \frac{1}{\alpha}\mathbb{E}[\mbox{Gap}^*(\bar{y}_K)]\leq \frac{\mathcal O(N)}{\sqrt[4]{K}}.\end{align} Moreover, using convexity of $f$ and Cauchy-Schwartz inequality, we conclude that
 $$\mathbb E[f(\bar y_k)]-f^*\geq\mathbb E[f(\bar y_k)]-f(\hat y^*)\geq \mathbb E[  \fyy{\nabla f(\hat y^*)^T\left(\bar y_K-\hat y^*\right)}]\geq -\|\nabla f(\hat y^*)\| \mathbb E[\|\bar y_K-\hat y^*\|]\geq -\frac{\mathcal O(N)}{\sqrt[4]{K}},$$
 where in the first inequality we used the fact that $f^*\leq f(\hat y^*)$ and the last inequality follows from \eqref{lower gap} and the fact that the gradient is bounded. 
 \end{proof}
}

The main result in this section is presented in the following
\begin{lemma}[\fyy{Error bounds in approximating the PoS}]\label{prop:pos_bounds}\em
Consider Algorithm~\ref{algorithm:pos}.  Let Assumptions~\ref{assum:problem},~\ref{assum:rnd_vars},~\ref{assum:vars_alg_pos}, \afj{and ~\ref{sharp}} hold. Suppose, $r_1, r_2 \in[0,1)$ be fixed scalars and for any $k\geq 0$, let us define
\begin{align*}
\gamma_{k,1} \triangleq \frac{\gamma_{0,1}}{\sqrt[4]{(k+1)^3}}, \quad \rho_k \triangleq \rho_0\sqrt[4]{k+1}, \quad \gamma_{k,2} \triangleq \frac{\gamma_{0,2}}{\sqrt{k+1}}. 
\end{align*}
Then the following holds
\begin{align}\label{eqn:bounds_pos}
\afj{-\mathcal{O}\left(\frac{1}{\sqrt[4]{K}}\right)} 
\leq \frac{\mathbb{E}[\hat{f}(\bar{y}_{\me{K},1})]}{\mathbb{E}[\hat{f}(\bar{y}_{\me{K},2})]}-\mbox{PoS}
 \leq \mathcal{O}\left(\frac{1}{\sqrt[4]{K}}\right). 
\end{align}
\end{lemma}
\begin{algorithm}[H]
  \caption{\footnotesize Approximating PoS using randomized stochastic \fa{extra-gradient} schemes}
\label{algorithm:pos}
    \begin{algorithmic}[1]
    \STATE \textbf{initialization:} Set random initial points $x_{0,1},x_{0,2}, y_{0,1},y_{0,2}\in X$, initial stepsizes $\gamma_{0,1}, \gamma_{0,2}>0$, scalar $0 \leq r_1, r_2<1$, $\bar{y}_{0,1} = \bar{y}_{0,2}:= y_0$, $\Gamma_{0,1}=\Gamma_{0,2}:=0$, $S_{0,1}=S_{0,2} : = 0$.
    \FOR {$k=0,1,\ldots,K-1$}
     \STATE Generate $i_{k,1}$, $\tilde i_{k,1}$, $i_{k,2}$, and $\tilde{i}_{k,2}$ uniformly from $\{1,\ldots,N\}$.
     \STATE Generate $\xi_{k,1}$, $\tilde \xi_{k,1}$, $\xi_{k,2}$, and $\tilde \xi_{k,2}$ as realizations of the random vector $\xi$.
      \STATE Update the variables $y_{k,1}$, $x_{k,1}$, $y_{k,2}$, and $x_{k,2}$ \fa{as}
\begin{align}
y_{k+1,1}^{(i)}&:= \left\{\begin{array}{ll}\mathcal{P}_{X_i}\left(x_{k,1}^{(i)}-\gamma_{k,1}(\tilde \nabla_i f(x_{k,1},\tilde \xi_{k,1}) + \rho_k F_{i}(x_{k,1},\tilde \xi_{k,1}))\right)
&\hbox{if } i=\tilde i_{k,1},\cr \hbox{} &\hbox{}\cr
x_{k,1}^{(i)}& \hbox{if } i\neq \tilde i_{k,1},\end{array}\right.\label{eqn:y_k_update_rule_pos}
\\
&\hbox{} \notag\\
x_{k+1,1}^{(i)}&:=\left\{\begin{array}{ll}\mathcal{P}_{X_i}\left(x_{k,1}^{(i)} - \gamma_{k,1}(\tilde \nabla_i f(y_{k+1,1}, \xi_{k,1}) + \rho_k F_{i}(y_{k+1,1}, \xi_{k,1}))\right)
&\hbox{if } i=i_{k,1},\cr \hbox{} &\hbox{}\cr
x_{k,1}^{(i)}& \hbox{if } i\neq i_{k,1},\end{array}\right.\label{eqn:x_k_update_rule_pos}\\
&\hbox{} \notag\\
y_{k+1,2}^{(i)}&:= \left\{\begin{array}{ll}\mathcal{P}_{X_i}\left(x_{k,2}^{(i)}-\gamma_{k,2}\tilde \nabla_i f(x_{k,2},\tilde \xi_{k,2}) \right)
&\hbox{if } i=\tilde i_{k,2},\cr \hbox{} &\hbox{}\cr
x_{k,2}^{(i)}& \hbox{if } i\neq \tilde i_{k,2},\end{array}\right.\label{eqn:y_k_update_rule_pos_2}
\\
&\hbox{} \notag\\
x_{k+1,2}^{(i)}&:=\left\{\begin{array}{ll}\mathcal{P}_{X_i}\left(x_{k,2}^{(i)} - \gamma_{k,2} \tilde \nabla_i f(y_{k+1,2}, \xi_{k,2}) \right)
&\hbox{if } i=i_{k,2},\cr \hbox{} &\hbox{}\cr
x_{k,2}^{(i)}& \hbox{if } i\neq i_{k,2}.\end{array}\right.\label{eqn:x_k_update_rule_pos_2}
\end{align}
     \STATE Update $\Gamma_{k,1}$, $\Gamma_{k,2}$, $\bar y_{k,1}$, and $\bar y_{k,2}$ using the following recursions.
\begin{align}
&\Gamma_{k+1,1}:=\Gamma_{k,1}+(\gamma_{k,1}\rho_k)^{r_1}, \quad \bar y_{k+1,1}:=\frac{\Gamma_{k,1} \bar y_{k,1}+(\gamma_{k,1}\rho_k)^{r_1} y_{k+,1}}{\Gamma_{k+1,1}},\label{eqn:averaging_pos_eq1}\\
&\Gamma_{k+1,2}:=\Gamma_{k,2}+\gamma_{k,2}^{r_2}, \quad \bar y_{k+1,2}:=\frac{\Gamma_{k,2} \bar y_{k,2}+\gamma_{k,2}^{r_2} y_{k+1,2}}{\Gamma_{k+1,2}}\label{eqn:averaging_pos_eq2}.
\end{align}
    \ENDFOR

       \STATE \fyy{Generate the batch of samples $\{\zeta_{t}\}$ as i.i.d realizations of $\xi$, for $t=0,\ldots, M-1$}
       \STATE \fyy{Evaluate sample average approximations $\hat{f}_{M}(\bar y_{K,1}):= \frac{1}{M}\sum_{t=0}^{M-1}f\left(\bar y_{K,1}, {\zeta_{t}}\right)$ and $\hat{f}_{M}(\bar y_{K,2}):=\frac{1}{M}\sum_{t=0}^{M-1}f\left(\bar y_{K,2}, \zeta_{t}\right)$}
     \STATE \fyy{Return $\frac{\hat{f}_{M}(\bar y_{K,1})}{\hat{f}_{M}(\bar y_{K,2})}$.}
   \end{algorithmic}
\end{algorithm}

\begin{proof}\em \fyy{We utilize the following notation in the proof
\begin{align*}
\mathcal{F}_{k,1}& \triangleq \cup_{t=0}^k\{\tilde \xi_{t,1},\tilde i_{t,1},  \xi_{t,1},  i_{t,1}\} \cup \{x_{0,1},y_{0,1}\},\qquad \hbox{for all } k \in\{0,\ldots,K-1\},\\
\mathcal{F}_{k,2}& \triangleq \cup_{t=0}^k\{\tilde \xi_{t,2},\tilde i_{t,2},  \xi_{t,2},  i_{t,2}\} \cup \{x_{0,2},y_{0,2}\},\qquad \hbox{for all } k \in\{0,\ldots,K-1\}.
\end{align*}
}
\fyy{Recall the definitions $\hat{f}_{M}(\bar y_{K,1}):= \frac{1}{M}\sum_{t=0}^{M-1}f\left(\bar y_{K,1}, {\zeta_{t}}\right)$ and $\hat{f}_{M}(\bar y_{K,2}):=\frac{1}{M}\sum_{t=0}^{M-1}f\left(\bar y_{K,2}, \zeta_{t}\right)$}. \fyy{Then, we can write}
\begin{align*}
   \mathbb{E}\left[\fyy{\hat{f}}_{M}(\bar y_{K,1})\right]  = \fyy{\mathbb{E}\left[\mathbb{E}\left[\fyy{\hat{f}}_{M}(\bar y_{K,1})\mid \mathcal{F}_{K-1,1}\right]\right] =      \mathbb{E}\left[\mathbb{E}\left[\tfrac{1}{M}\sum_{t=0}^{M-1}f\left(\bar y_{K,1},\zeta_{t}\right)\mid \mathcal{F}_{K-1,1}\right]\right] = }\mathbb{E}[f(\bar y_{K,1})]. 
   \end{align*}
From the preceding relation and \fyy{Theorem \ref{thm:rates}} we have
\begin{align*}
  \afj{-\frac{\mathcal O(N)}{\sqrt[4]{K}}}   \leq \fyy{\mathbb{E}\left[\hat{f}_{M}(\bar y_{K,1})\right]- \fa{f^*} }\leq  \frac{\fyy{\mathcal{O}(N)}}{\sqrt[4]{K}},
\end{align*}
\fyy{where $\fa{f^*}$ denotes the optimal objective value of problem~\eqref{prob:sopt_svi}. Let us define $f^*_{Opt}\triangleq \min_{x\in X}\mathbb E[f(x,\xi)]$. Similarly, 
\begin{align*}
   \mathbb{E}\left[\fyy{\hat{f}}_{M}(\bar y_{K,2})\right]  = \mathbb{E}\left[\mathbb{E}\left[\fyy{\hat{f}}_{M}(\bar y_{K,2})\mid \mathcal{F}_{K-1,2}\right]\right] =      \mathbb{E}\left[\mathbb{E}\left[\tfrac{1}{M}\sum_{t=0}^{M-1}f\left(\bar y_{K,2},\xi_t\right)\mid \mathcal{F}_{K-1,2}\right]\right] = \mathbb{E}[f(\bar y_{K,2})]. 
\end{align*}
and we also have that  
\begin{align*}
  0   \leq \fyy{\mathbb{E}\left[\hat{f}_{M}(\bar y_{K,2})\right]}- f^*_{Opt} \leq  \frac{\fyy{\mathcal{O}(N)}}{\sqrt{K}}.
\end{align*}
\afj{We show the result holds when $\fa{f^*},f^*_{Opt}\geq 0$ and one can verify that the result also holds for other cases}. \fyy{From the definition of PoS given by \eqref{pos} and the two preceding inequalities, we can write
\begin{align*}
 \frac{\mathbb{E}[\hat{f}(\bar{y}_{\me{K},1})]}{\mathbb{E}[\hat{f}(\bar{y}_{\me{K},2})]} \leq \frac{\fa{f^*}+\frac{\fyy{\mathcal{O}(N)}}{\sqrt[4]{K}}}{f^*_{Opt}} 
 = \frac{\fa{f^*}}{f^*_{Opt}} +\frac{\fyy{\mathcal{O}(N)}}{\sqrt[4]{K}} = \mbox{PoS}  +\frac{\fyy{\mathcal{O}(N)}}{\sqrt[4]{K}}. 
\end{align*}
We can also write
\begin{align*}
 \frac{\mathbb{E}[\hat{f}(\bar{y}_{\me{K},1})]}{\mathbb{E}[\hat{f}(\bar{y}_{\me{K},2})]} \geq \frac{\fa{f^*}\afj{-\frac{\mathcal O(N)}{\sqrt[4]{K}}} }{f^*_{Opt}+\frac{\fyy{\mathcal{O}(N)}}{\sqrt{K}}} 
  = \left(\frac{1\afj{-\frac{\mathcal O(N)}{\sqrt[4]{K}}} }{1+\frac{\fyy{\mathcal{O}(N)}}{\sqrt{K}}}\right)\times \mbox{PoS} 
  &\implies \frac{\mathbb{E}[\hat{f}(\bar{y}_{\me{K},1})]}{\mathbb{E}[\hat{f}(\bar{y}_{\me{K},2})]} -\mbox{PoS}\geq -\frac{\mathcal{O}\left(N\right)}{\afj{\sqrt[4]{K}}}.
\end{align*}
Thus, in view of the two preceding inequalities, the result holds.
} 
}

\end{proof}

 \fyy{\begin{remark}\em We note that in Algorithm~\ref{algorithm:pos}, in using the extra-gradient method employed for solving $\min_{x\in X} \mathbb{E}[f(x,\xi)]$, we do not use any penalization. However, in solving $\min_{x\in {\tiny \hbox{SOL}}(X,\mathbb{E}[F(\bullet,\xi)])} \mathbb{E}[f(x,\xi)]$, we employ Algorithm 1 where we utilize iterative penalization. Intuitively speaking, problem $\min_{x\in X} \mathbb{E}[f(x,\xi)]$ can be viewed as a special case of $\min_{x\in {\tiny \hbox{SOL}}(X,\mathbb{E}[F(\bullet,\xi)])} \mathbb{E}[f(x,\xi)]$ where the mapping $F(x)$ is zero for all $x$. As such, we suppress the penalization in solving $\min_{x\in X} \mathbb{E}[f(x,\xi)]$. This allows us to use larger stepsizes in solving $\min_{x\in X} \mathbb{E}[f(x,\xi)]$ and obtain faster convergence for the optimality metric.

Moreover, in Algorithm~\ref{algorithm:pos}, in solving $\min_{x\in X} \mathbb{E}[f(x,\xi)]$, we employ the averaging weights $  \frac{(\gamma_{k,2})^r}{\sum_{j=0}^{K-1} (\gamma_{j,2})^r}$. However, in solving $\min_{x\in {\tiny \hbox{SOL}}(X,\mathbb{E}[F(\bullet,\xi)])} \mathbb{E}[f(x,\xi)]$, we use the averaging weights $  \frac{(\gamma_{k,1}\rho_k)^r}{\sum_{j=0}^{K-1} (\gamma_{j,1}\rho_j)^r}$. We note that in view of the choices of the stepsizes and penalty parameter in Lemma~\ref{prop:pos_bounds}, the averaging weights of the two schemes are indeed almost identical. This is because in Lemma~\ref{prop:pos_bounds}, assuming that $\gamma_{0,1}\rho_0 = \gamma_{0,2}$, we have $\gamma_{k,1}\rho_k = \gamma_{k,2}$ for all $k$.
 \end{remark}}

\section{Numerical Experiments}\label{sec:num}
In this section we present the performance of the proposed schemes in estimating the price of stability for a stochastic Nash Cournot competition over a network. Cournot game is one of the most popular and amongst the first economic models for formulating the competition among multiple firms (see~\cite{JohariThesis,FacchineiPang2003} \fyy{for} the applications of Cournot models in imperfectly competitive power markets and also, rate control in communication
networks). The Cournot model is described as follows. Consider a collection of $N$ firms who compete over a network with $J$ nodes to sell a product. The strategy of firm $i \in\{1, \dots, N\}$ is characterized by the decision variables $y_{ij}$ and $s_{ij}$, denoting the generation and sales of firm $i$ at the node $j$, respectively.  Compactly, the decision variables of the $i^{th}$ firm is denoted by $x^{(i)} \triangleq \left(y_i, s_i\right) \in \mathbb{R}^{2J}$ where we assume that $\ y_i \triangleq \left(y_{i1}, \dots, y_{iJ}\right)$ and $s_i \triangleq \left(s_{i1}, \dots, s_{iJ}\right)$. The goal of the $i^{th}$ firm lies in minimizing the expected value of a net cost function $f_i\left(x^{(i)}, x^{(-i)},\xi\right)$ over the network over the strategy set $X_i$. This optimization problem for the firm $i$ is defined \fa{as}
\begin{align*}
\text{minimize} \qquad &\mathbb{E}\left[f_i\left(x^{(i)}, x^{(-i)},\xi\right)\right]   \triangleq \mathbb{E}\left[\sum_{j=1}^{{J}} c_{ij} (y_{ij})- \sum_{j=1}^{{J}} s_{ij}p_j\left(\bar{s}_j,\xi\right)\right]\\
\text{Subject to.} \qquad & x^{(i)} \in X_i \triangleq \left \{  \left(y_i, s_i\right) \mid    y_{ij} \leq \mathcal{B}_{ij},\sum_{j=1}^{J} y_{ij} = \sum_{j=1}^{J} s_{ij}, \quad y_{ij}, s_{ij} \geq 0,\ \text{ for all } j = 1, \dots, J \right \}. 
\end{align*}
Here, $\bar{s}_j \triangleq \sum_{i=1}^{d}s_{ij}$ denotes the aggregate sales from all the firms at node $j$, $p_j:\mathbb{R}\times \Omega\to \mathbb{R}$ denotes the price function characterized in terms of the aggregate sales at the node $j$ and a random variable $\xi$, and $c_{ij}:\mathbb{R}\to \mathbb{R}$ denotes the production cost function of firm $i$ at node $j$. The price functions are given as $p_j \left(\bar{s}_j,\xi\right) \triangleq \alpha_j(\xi) - \beta_j \left(\bar{s}_j\right)^{\sigma}$, where $\alpha_j(\xi)$ is a random positive variable, $ \beta_j$ is a positive scalar, and $\sigma \geq 1$. We assume that cost \fyy{functions} are linear and the transportation costs are zero. The constraint $y_{ij} \leq \mathcal{B}_{ij}$ states that the generation is capacitated where $\mathcal{B}_{ij}$ is a positive scalar for all $i$ and $j$. Similar to~\cite{kaushik2021method}, in defining a global objective function for the price of stability, we consider \fyy{the} Marshallian aggregate surplus function defined as $$\mathbb{E}[f(x,\xi)]\triangleq \sum_{i=1}^N \mathbb{E}\left[f_i\left( x^{(i)}, x^{(-i)},\xi   \right)\right]. $$

It has been shown~\cite{KannanShanbhag2012} that when $\sigma \geq 1$, $f$ is convex and also, when either $\sigma =1$ or $1<\sigma\leq 3$ and $N\leq \frac{3\sigma-1}{\sigma-1}$, the mapping associated with the Cournot game, i.e., $F(x)\triangleq \left(\nabla_{x^{(1)}} \mathbb{E}[f_1(x,\xi)], \ldots,\nabla_{x^{(N)}} \mathbb{E}[f_N(x,\xi)] \right)$ is merely monotone.
 
\textbf{Experiments and set-up.} We compare the performance of Algorithm~\ref{algorithm:aR-IP-SeG} with that of the two existing methods, namely \fyy{aRB-IRG} in~\cite{kaushik2021method} and the sequential regularization (SR) scheme (cf.~\cite{FacchineiPang2003,kaushik2021method}). Note that both the SR scheme and \fyy{aRB-IRG} can only use deterministic gradients. To apply these two \fyy{methods}, we use a sample average approximation scheme by assuming that the deterministic gradient is approximated using a batch size of $1000$ random samples. In Algorithm~\ref{algorithm:aR-IP-SeG}, however, we can use stochastic gradients (using a single sample $\xi$). In both Algorithm~\ref{algorithm:aR-IP-SeG} and \fyy{aRB-IRG}, we employ a randomized block-coordinate scheme with $N$ number of blocks, where $N$ is the number of firms. We consider four different settings in our simulation results, where they differ in terms of the choices of the initial stepsize, the initial regularization parameter \fyy{used in} \fyy{aRB-IRG}, and the initial penalty parameter. For each setting, we implement the three methods on \fyy{four} different Cournot games, one with $2$ players over a network with $2$ nodes, \fyy{one with $4$ players over a network with $5$ nodes, one with $10$ players over a network with $2$ nodes}, and another with $10$ players over a network with $10$ nodes. We assume that $\alpha_j(\xi)$ is uniformly distributed for all the agents. To compare the simulation results, we generate $15$ independent sample-paths for \fyy{any of the schemes that are stochastic and/or randomized}.

\fyy{ \begin{table}[H]
\centering{
\caption{The four settings for the algorithm parameters}
\label{table:settings}
\begin{tabular}{||c |  c | c c c c||} 
 \hline
Algorithm & Parameter(s)  & Setting 1  & Setting 2 & Setting 3  & Setting 4  \\ [0.5ex] 
 \hline\hline
 SR scheme & $\gamma_0$ & 0.1 &  0.1 & 1  & 1 \\ 
   \hline
 aRB-IRG & $(\gamma_0,\eta_0)$   & (0.1,0.1) &  (0.1,1)  &  (1,0.1)  &  (1,1) \\ 
 \hline
aR-IP-SeG & $(\gamma_0,\rho_0)$ & (0.01,10) & (0.1,1)  & (0.1,10)  & (1,1) \\ 
[1ex] 
 \hline
\end{tabular}}
\vspace{-.1in}
\end{table}}

 \begin{figure}
    \centering
    \includegraphics[width=0.3\textwidth]{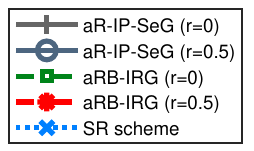}
    \caption{The figure legend used in the numerical experiments in Figures~\ref{fig:comparison1}--\ref{fig:comparison4}}
    \label{fig:legends}
\end{figure}

\textbf{Results and insights.} The simulation results are presented in Figures~\ref{fig:comparison1}-\ref{fig:comparison4}, and~\ref{fig:pos}. \fa{Note that the legend for Figures~\ref{fig:comparison1}-\ref{fig:comparison4} is presented in Figure~\ref{fig:legends}}. Several observations can be made: (i) As it can be seen in \fyy{Figures~\ref{fig:comparison1}-\ref{fig:comparison4}}, Algorithm~\ref{algorithm:aR-IP-SeG} outperforms the other two methods in almost all the scenarios. \fyy{We note that a smaller gap function value implies a smaller infeasibility for the solution iterate. However, because the solution iterate may be infeasible during the implementation of aRB-IRG and aR-IP-SeG , a smaller objective value may not necessarily imply a better solution. Instead, when comparing the objective function metric in the figures, it is important to observe how fast the objective value of each method reaches to a stable value.} (ii) Although both Algorithm~\ref{algorithm:aR-IP-SeG} and \fyy{aRB-IRG} are equipped with the same convergence speeds, Algorithm~\ref{algorithm:aR-IP-SeG} enjoys a better performance with respect to the run-time. This is because it uses stochastic gradients that are cheaper to compute in contrast with the sample average gradients used in \fyy{aRB-IRG}. (iii) We do observe that as the size of the problem increases in terms of the number of players and the size of the network, the performance of all the schemes is downgraded. However, Algorithm~\ref{algorithm:aR-IP-SeG} seems to stay robust across most settings and often outperforms the other two methods. (vi) In estimating the PoS in Figure~\ref{fig:pos}, \fyy{the methods seem to converge to a PoS smaller than one. This is because in this numerical experiment, we have considered the minimization of the negative of the profit function. As such, the optimal objective values of the minimization problems become negative. Consequently, the PoS is theoretically less than or equal to one. This is indeed aligned with the findings in Figure~\ref{fig:pos}.}   

\begin{table}[H]
\setlength{\tabcolsep}{0pt}
\centering{
 \begin{tabular}{c || c  c  c c}
  {\footnotesize {Setting}\ \ }& {\footnotesize  (1)} & {\footnotesize (2)} & {\footnotesize (3)  } & {\footnotesize (4)  }\\
 \hline\\
\rotatebox[origin=c]{90}{{\footnotesize {sample ave. gap}}}
&
\begin{minipage}{.22\textwidth}
\includegraphics[scale=.25, angle=0]{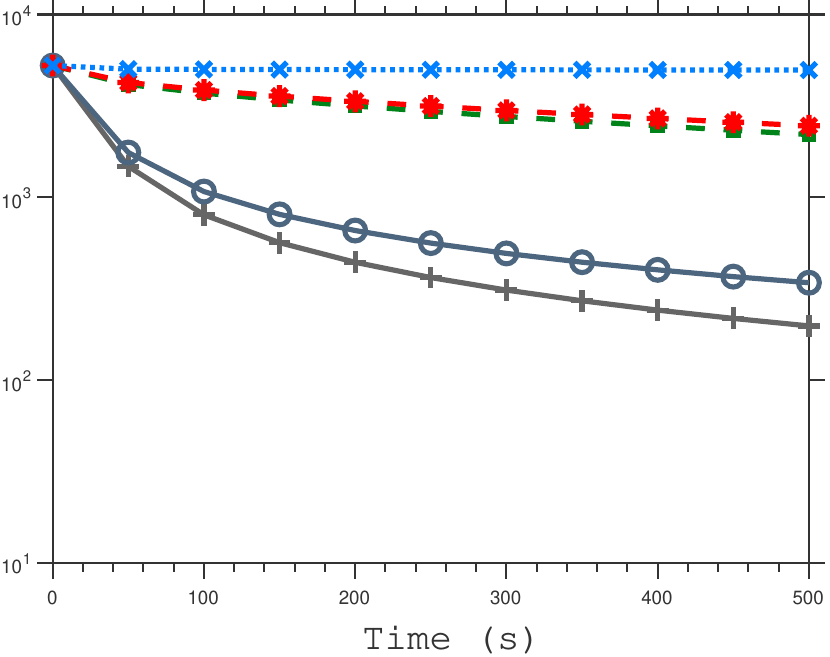}
\end{minipage}
&
\begin{minipage}{.22\textwidth}
\includegraphics[scale=.25, angle=0]{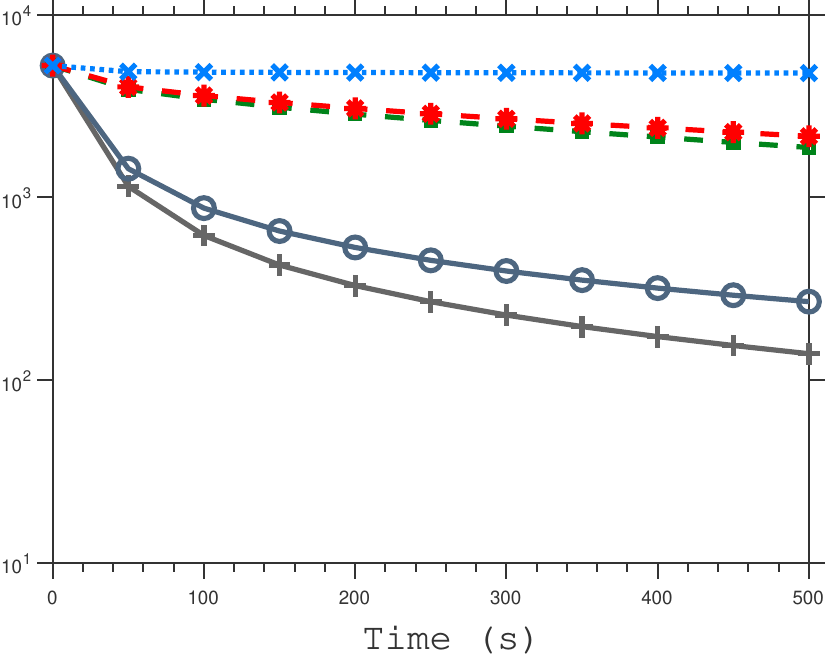}
\end{minipage}
	&
\begin{minipage}{.22\textwidth}
\includegraphics[scale=.25, angle=0]{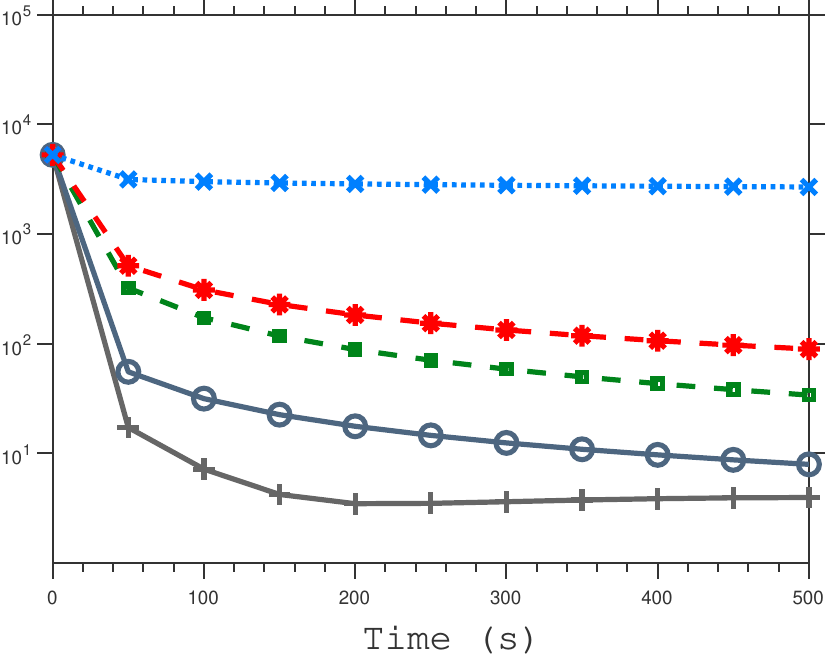}
\end{minipage}
&
\begin{minipage}{.22\textwidth}
\includegraphics[scale=.25, angle=0]{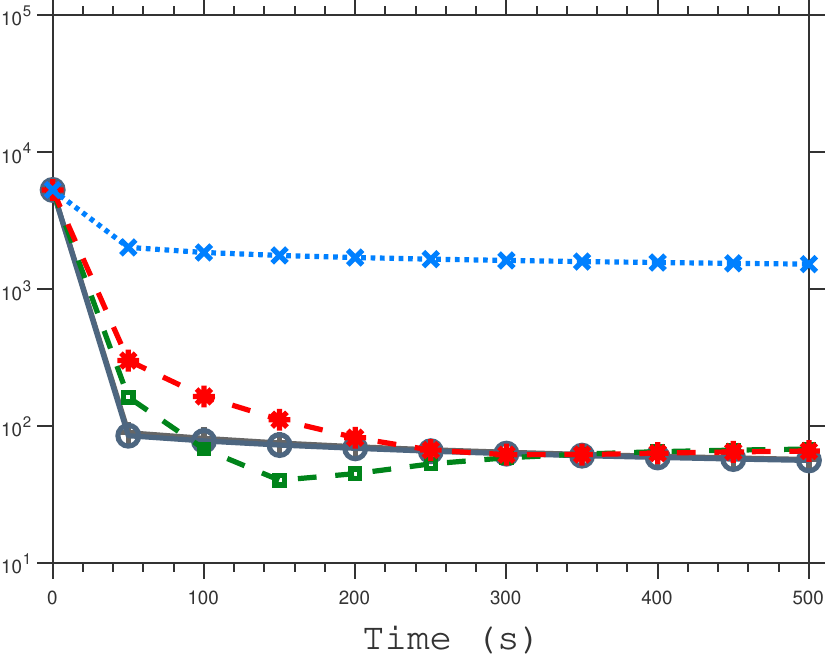}
\end{minipage}
\\
\hbox{}& & & &\\
 \hline\\
\rotatebox[origin=c]{90}{{\footnotesize sample ave. objective}}
&
\begin{minipage}{.22\textwidth}
\includegraphics[scale=.25, angle=0]{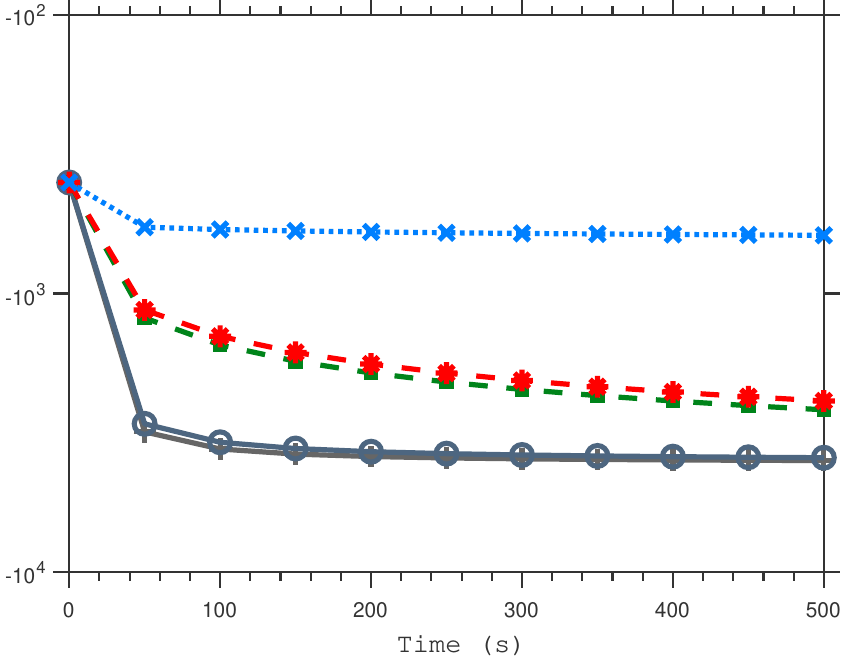}
\end{minipage}
&
\begin{minipage}{.22\textwidth}
\includegraphics[scale=.25, angle=0]{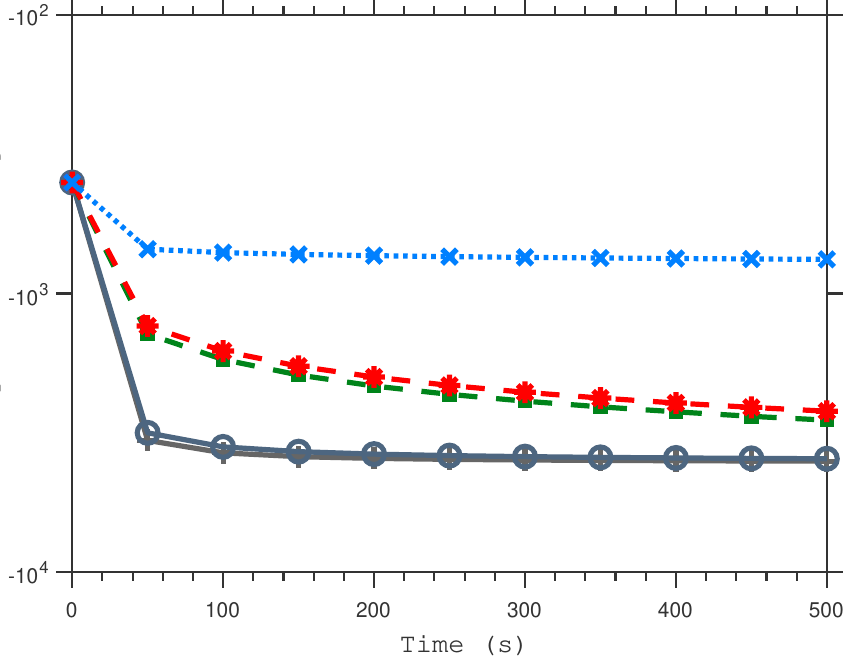}
\end{minipage}
&
\begin{minipage}{.22\textwidth}
\includegraphics[scale=.25, angle=0]{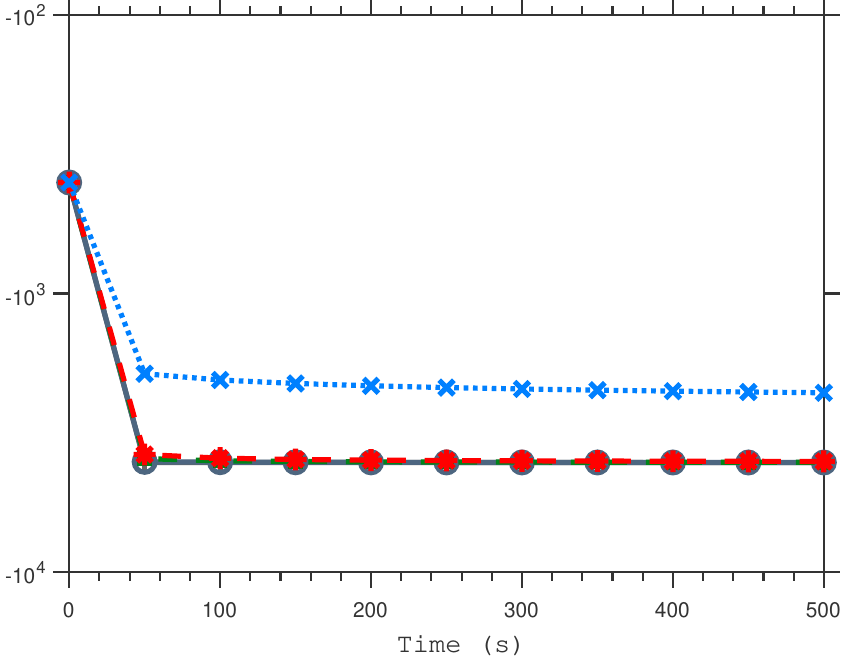}
\end{minipage}
&
\begin{minipage}{.22\textwidth}
\includegraphics[scale=.25, angle=0]{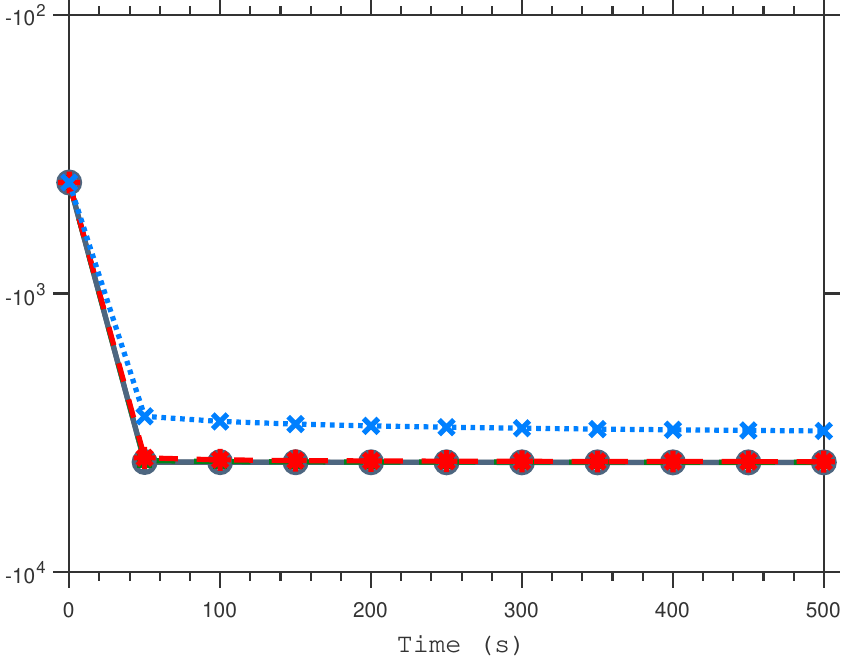}
\end{minipage}
\end{tabular}}
\captionof{figure}{Simulation results for a stochastic Nash Cournot game with 2 players over a network with 2 nodes, comparing Algorithm \eqref{algorithm:aR-IP-SeG} with other existing methods for solving problem \eqref{prob:sopt_svi}.}
\label{fig:comparison1}
\vspace{-.1in}
\end{table}

\begin{table}[H]
\setlength{\tabcolsep}{0pt}
\centering{
 \begin{tabular}{c || c  c  c c}
  {\footnotesize {Setting}\ \ }& {\footnotesize  (1)} & {\footnotesize (2)} & {\footnotesize (3)  } & {\footnotesize (4)  }\\
 \hline\\
\rotatebox[origin=c]{90}{{\footnotesize {sample ave. gap}}}
&
\begin{minipage}{.22\textwidth}
\includegraphics[scale=.25, angle=0]{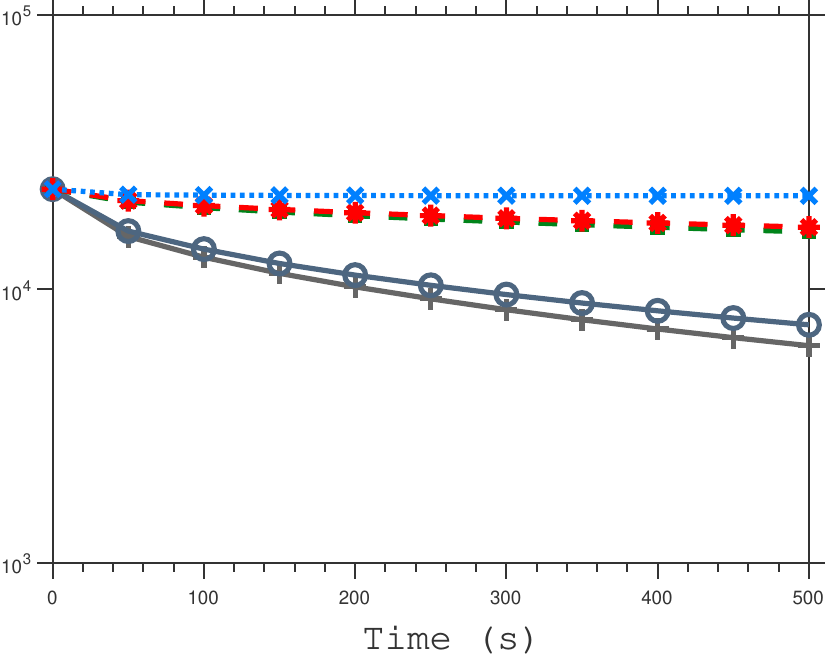}
\end{minipage}
&
\begin{minipage}{.22\textwidth}
\includegraphics[scale=.25, angle=0]{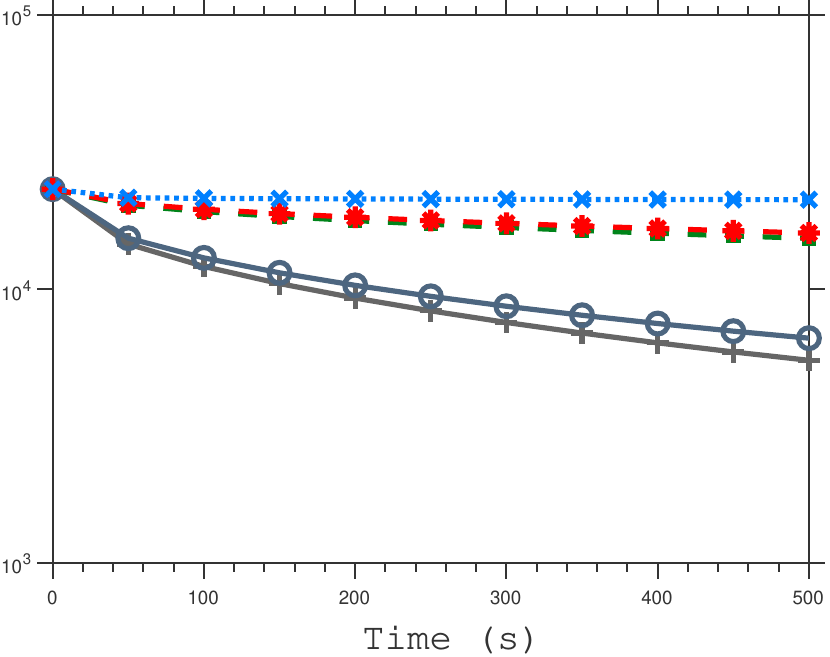}
\end{minipage}
	&
\begin{minipage}{.22\textwidth}
\includegraphics[scale=.25, angle=0]{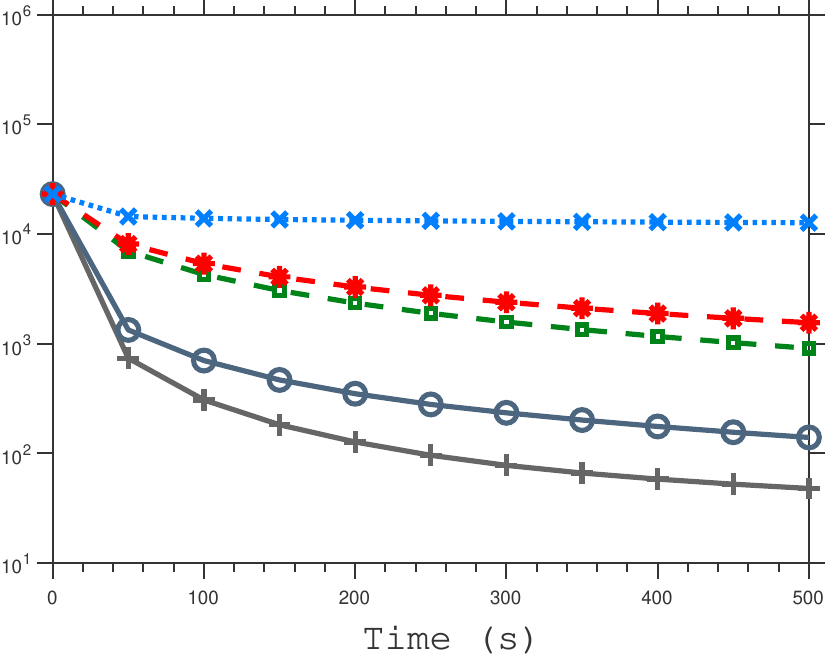}
\end{minipage}
&
\begin{minipage}{.22\textwidth}
\includegraphics[scale=.25, angle=0]{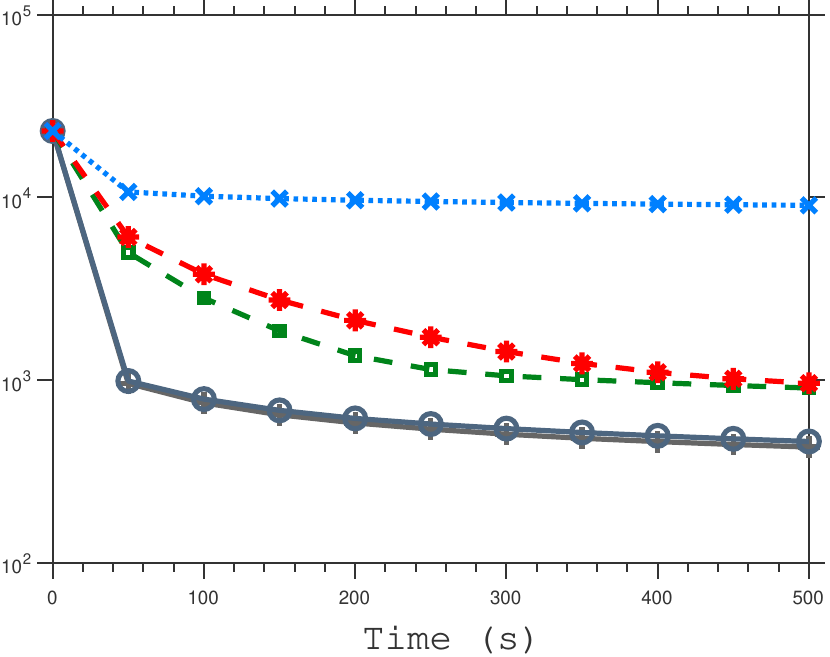}
\end{minipage}
\\
\hbox{}& & & &\\
 \hline\\
\rotatebox[origin=c]{90}{{\footnotesize sample ave. objective}}
&
\begin{minipage}{.22\textwidth}
\includegraphics[scale=.25, angle=0]{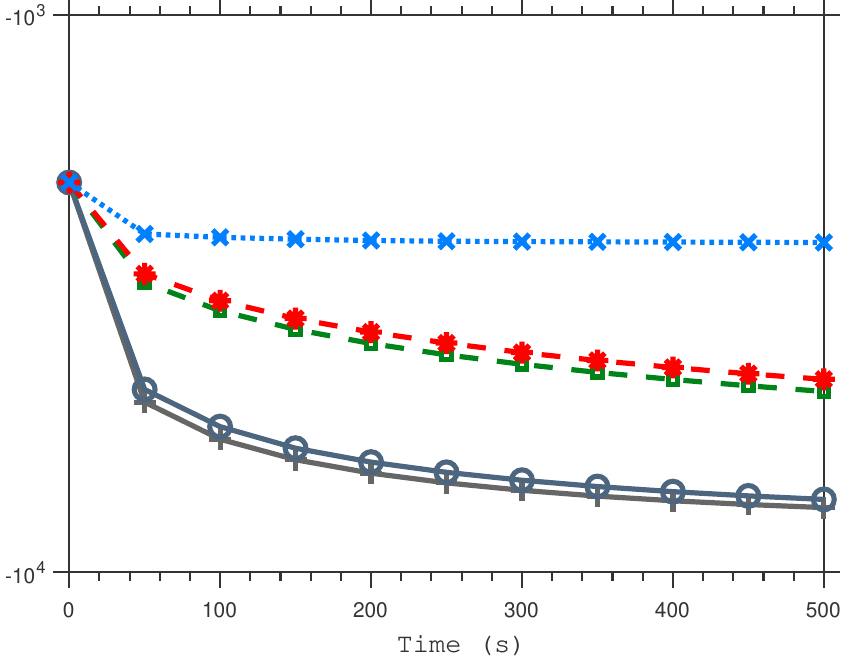}
\end{minipage}
&
\begin{minipage}{.22\textwidth}
\includegraphics[scale=.25, angle=0]{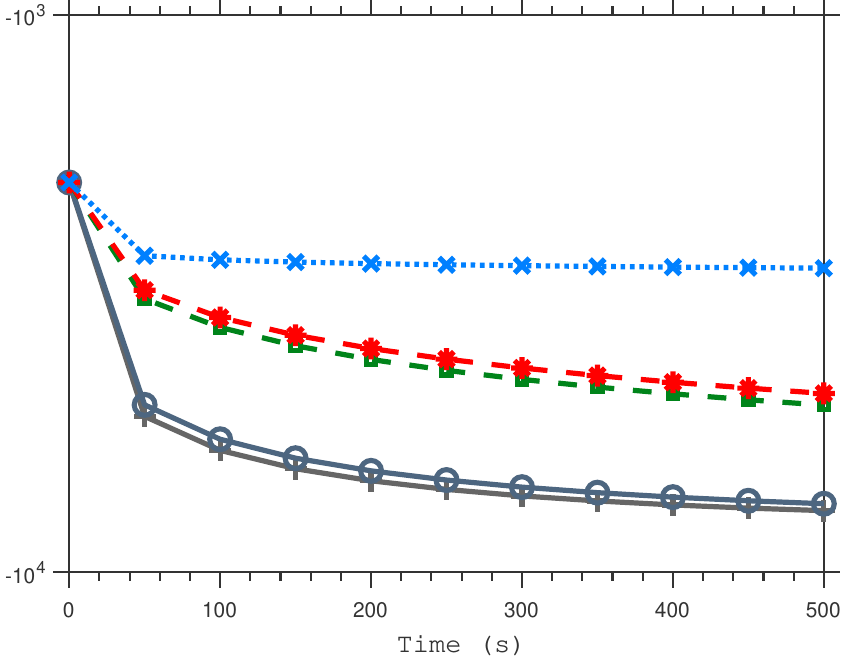}
\end{minipage}
&
\begin{minipage}{.22\textwidth}
\includegraphics[scale=.25, angle=0]{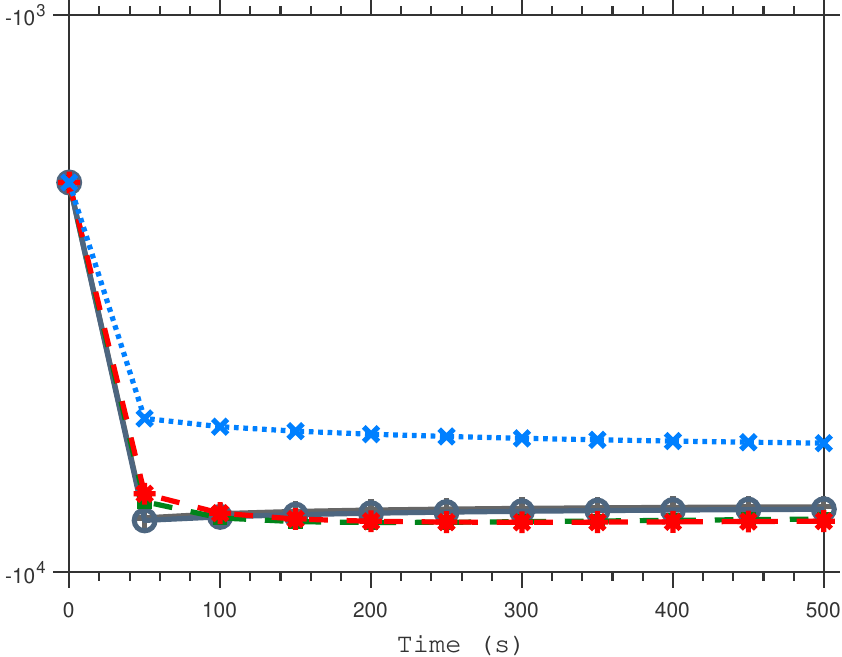}
\end{minipage}
&
\begin{minipage}{.22\textwidth}
\includegraphics[scale=.25, angle=0]{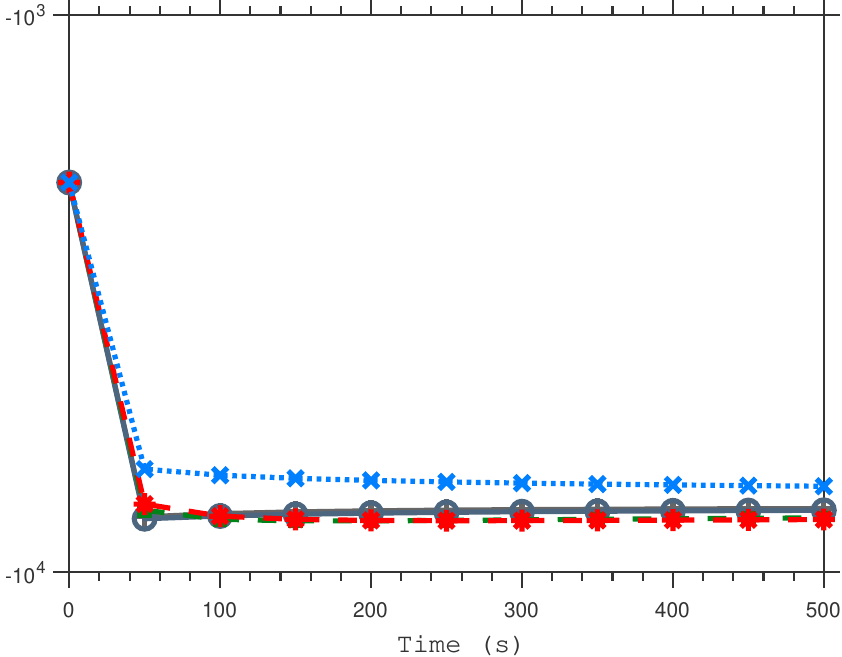}
\end{minipage}
\end{tabular}}
\captionof{figure}{Simulation results for a stochastic Nash Cournot game with 4 players over a network with 5 nodes, comparing Algorithm \eqref{algorithm:aR-IP-SeG} with other existing methods for solving problem \eqref{prob:sopt_svi}.}
\label{fig:comparison2}
\vspace{-.1in}
\end{table}

\begin{table}[H]
\setlength{\tabcolsep}{0pt}
\centering{
 \begin{tabular}{c || c  c  c c}
  {\footnotesize {Setting}\ \ }& {\footnotesize  (1)} & {\footnotesize (2)} & {\footnotesize (3)  } & {\footnotesize (4)  }\\
 \hline\\
\rotatebox[origin=c]{90}{{\footnotesize {sample ave. gap}}}
&
\begin{minipage}{.22\textwidth}
\includegraphics[scale=.25, angle=0]{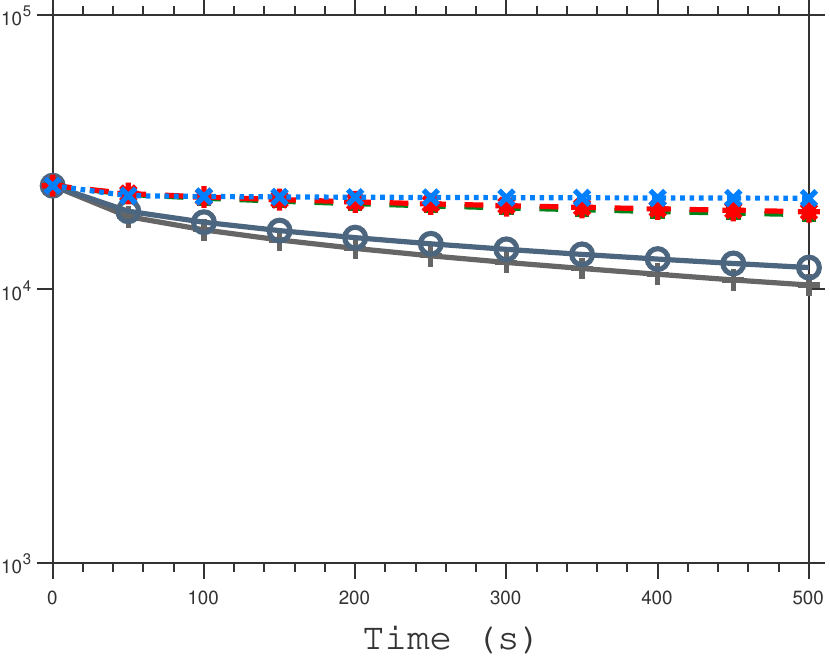}
\end{minipage}
&
\begin{minipage}{.22\textwidth}
\includegraphics[scale=.25, angle=0]{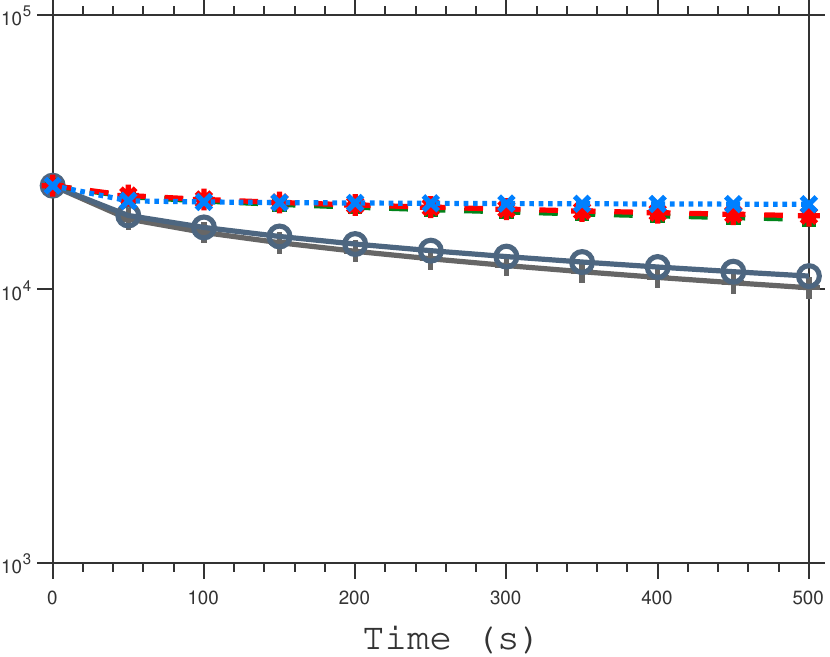}
\end{minipage}
	&
\begin{minipage}{.22\textwidth}
\includegraphics[scale=.25, angle=0]{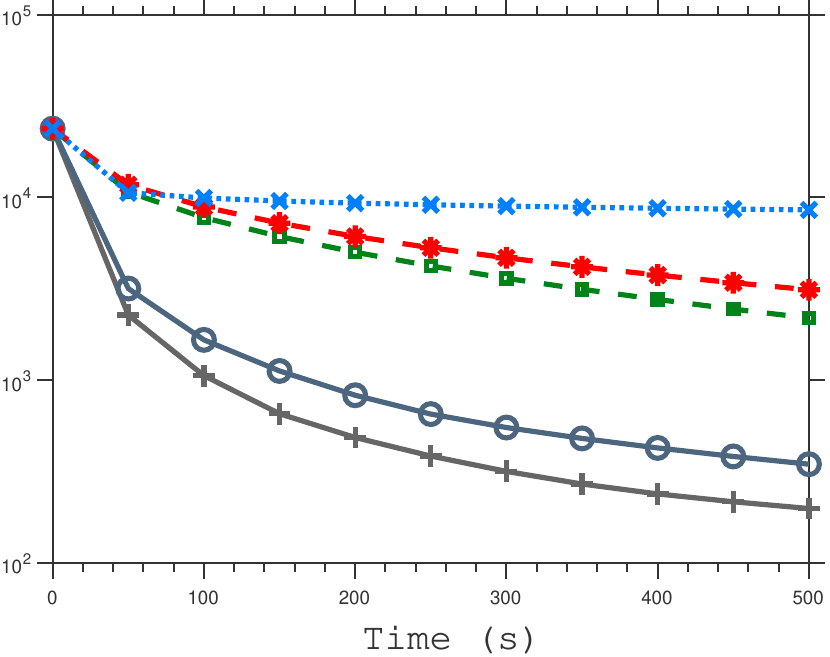}
\end{minipage}
&
\begin{minipage}{.22\textwidth}
\includegraphics[scale=.25, angle=0]{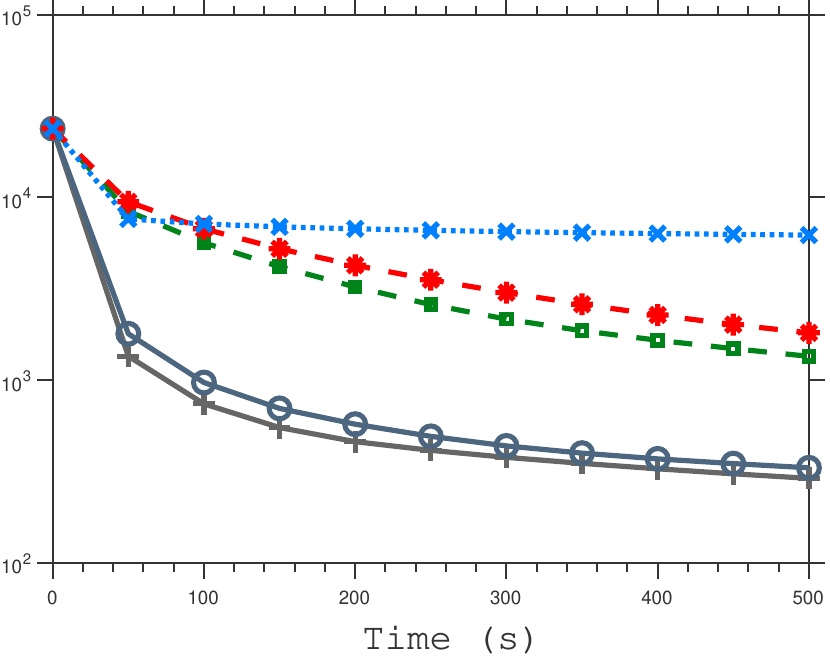}
\end{minipage}
\\
\hbox{}& & & &\\
 \hline\\
\rotatebox[origin=c]{90}{{\footnotesize sample ave. objective}}
&
\begin{minipage}{.22\textwidth}
\includegraphics[scale=.25, angle=0]{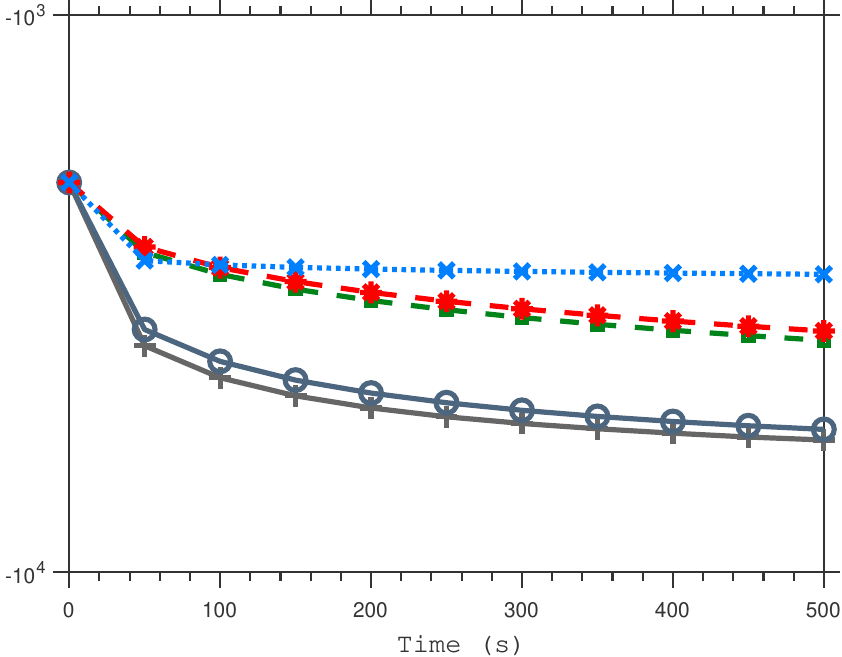}
\end{minipage}
&
\begin{minipage}{.22\textwidth}
\includegraphics[scale=.25, angle=0]{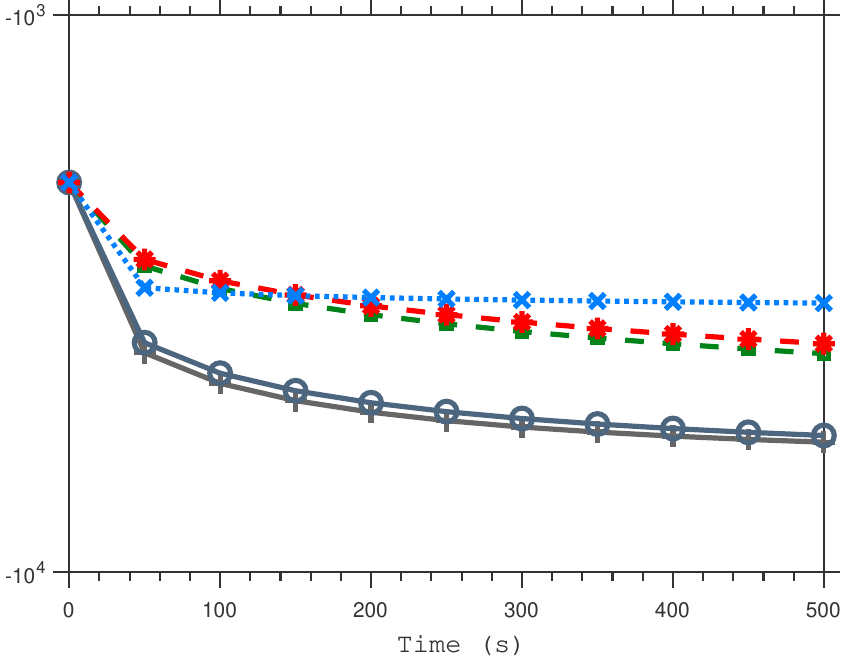}
\end{minipage}
&
\begin{minipage}{.22\textwidth}
\includegraphics[scale=.25, angle=0]{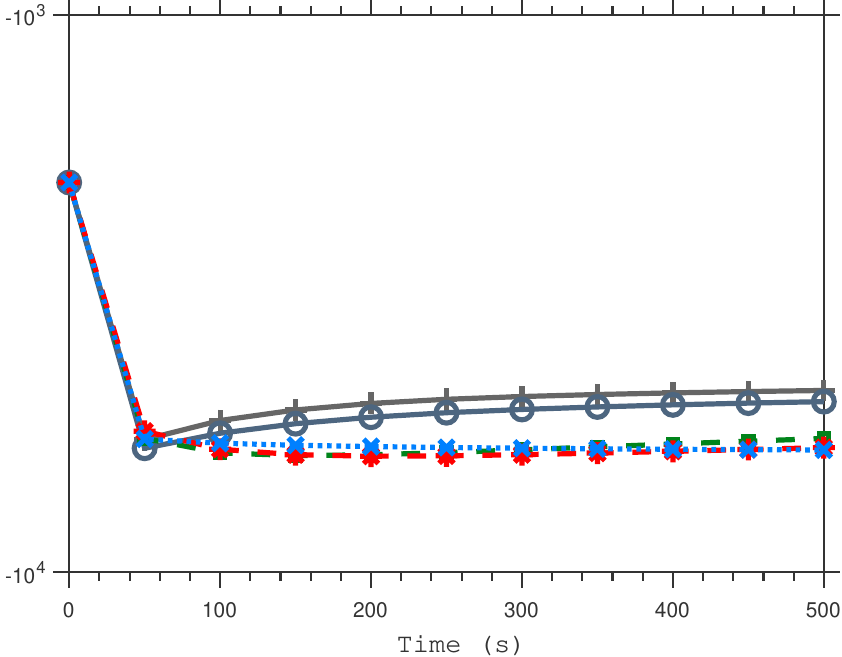}
\end{minipage}
&
\begin{minipage}{.22\textwidth}
\includegraphics[scale=.25, angle=0]{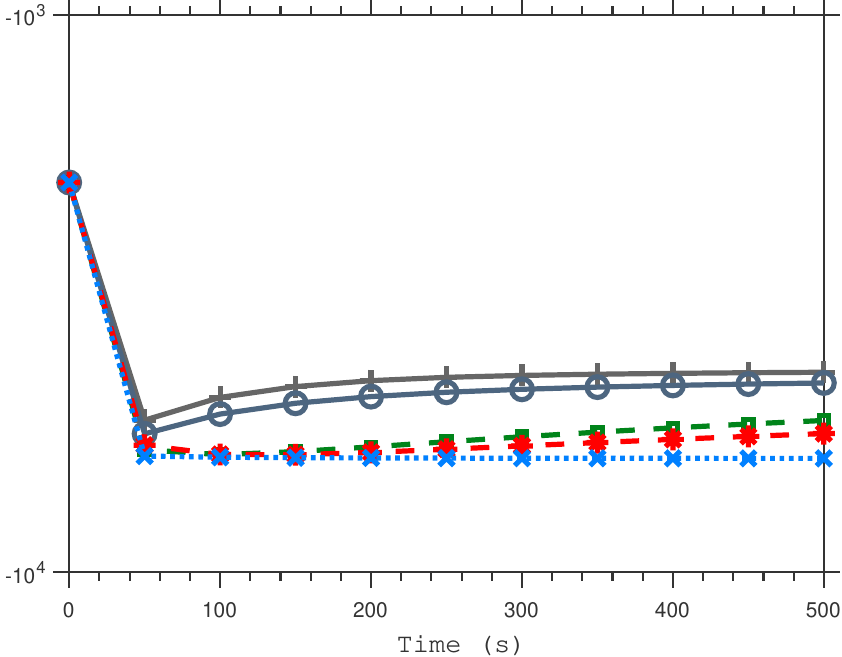}
\end{minipage}
\end{tabular}}
\captionof{figure}{Simulation results for a stochastic Nash Cournot game with 10 players over a network with 2 nodes, comparing Algorithm \eqref{algorithm:aR-IP-SeG} with other existing methods for solving problem \eqref{prob:sopt_svi}.}
\label{fig:comparison3}
\vspace{-.1in}
\end{table}

\begin{table}[H]
\setlength{\tabcolsep}{0pt}
\centering{
 \begin{tabular}{c || c  c  c c}
  {\footnotesize {Setting}\ \ }& {\footnotesize  (1)} & {\footnotesize (2)} & {\footnotesize (3)  } & {\footnotesize (4)  }\\
 \hline\\
\rotatebox[origin=c]{90}{{\footnotesize {sample ave. gap}}}
&
\begin{minipage}{.22\textwidth}
\includegraphics[scale=.25, angle=0]{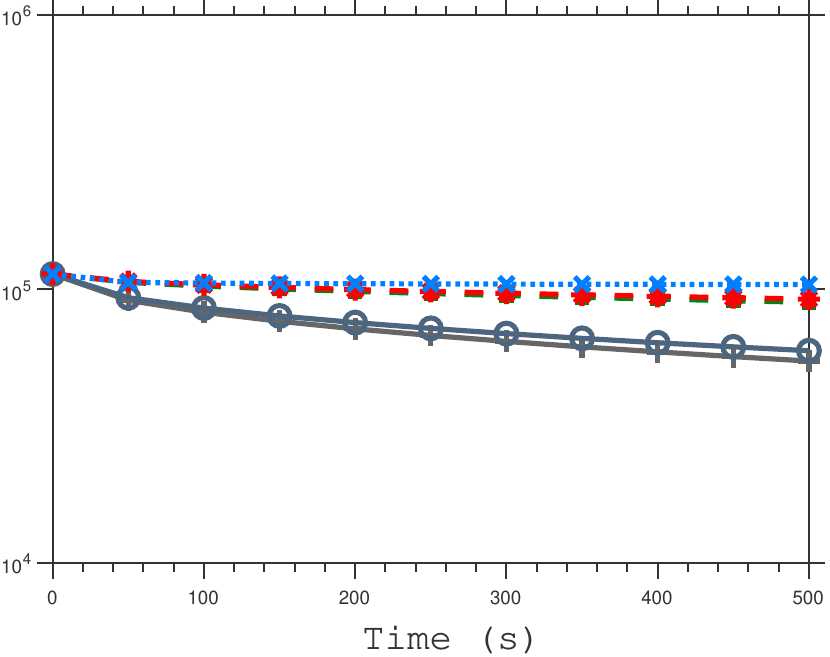}
\end{minipage}
&
\begin{minipage}{.22\textwidth}
\includegraphics[scale=.25, angle=0]{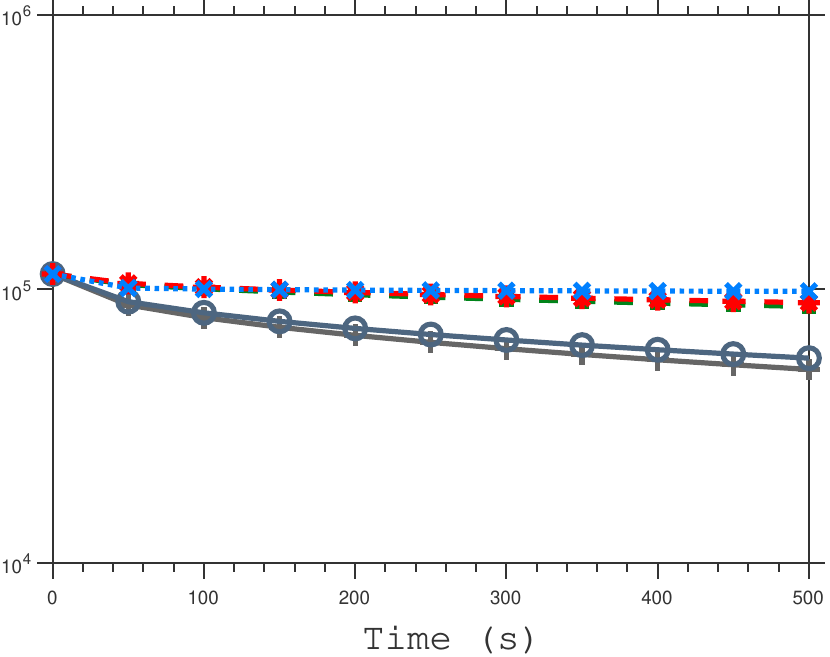}
\end{minipage}
	&
\begin{minipage}{.22\textwidth}
\includegraphics[scale=.25, angle=0]{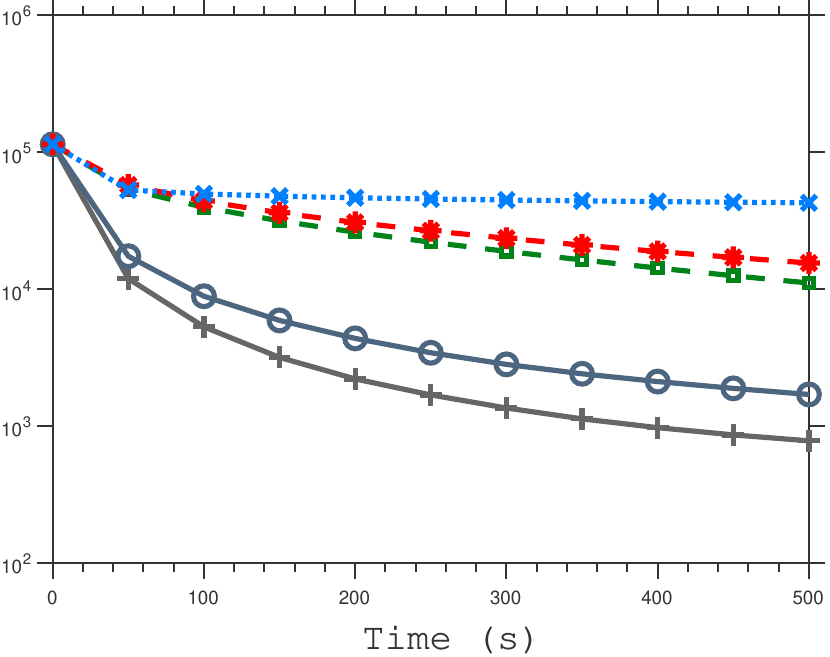}
\end{minipage}
&
\begin{minipage}{.22\textwidth}
\includegraphics[scale=.25, angle=0]{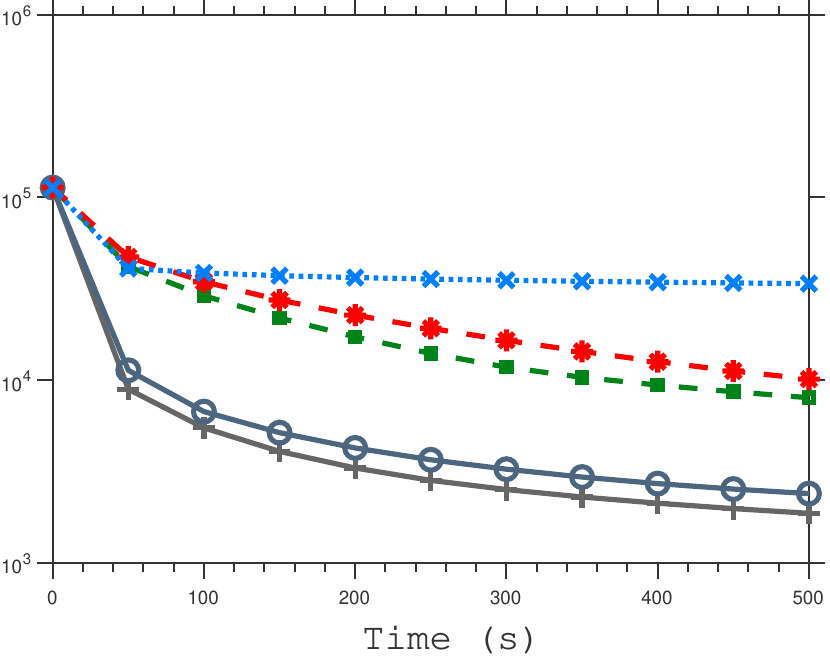}
\end{minipage}
\\
\hbox{}& & & &\\
 \hline\\
\rotatebox[origin=c]{90}{{\footnotesize sample ave. objective}}
&
\begin{minipage}{.22\textwidth}
\includegraphics[scale=.25, angle=0]{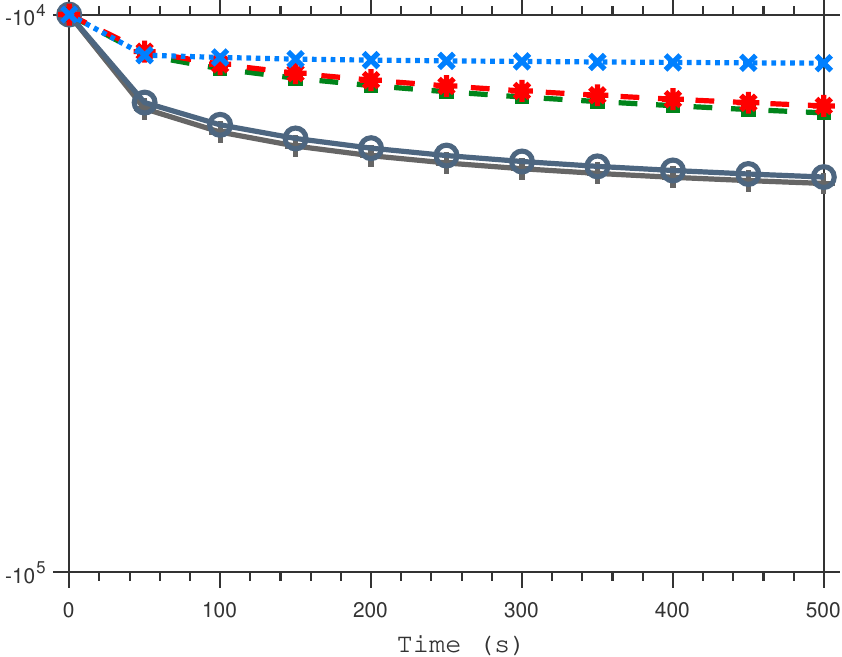}
\end{minipage}
&
\begin{minipage}{.22\textwidth}
\includegraphics[scale=.25, angle=0]{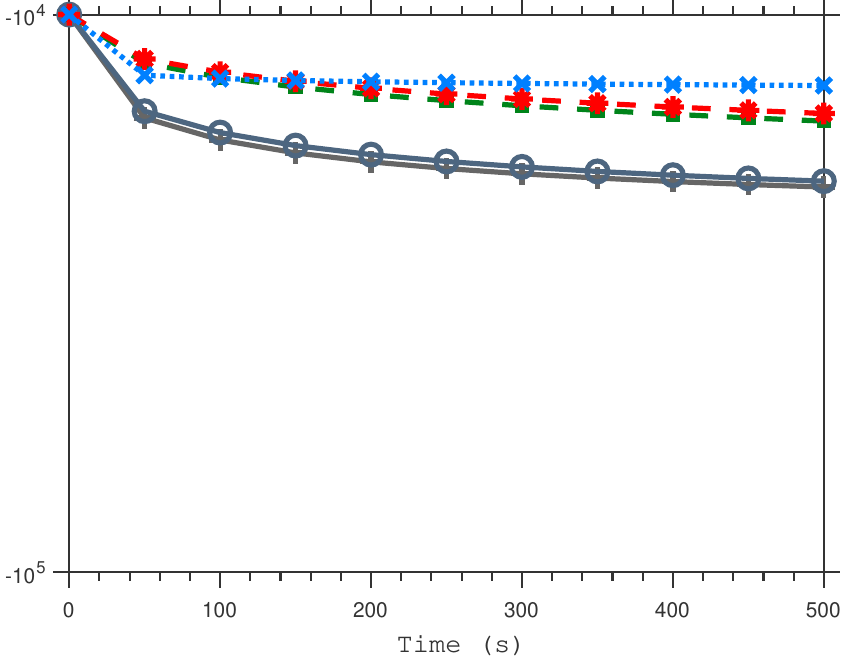}
\end{minipage}
&
\begin{minipage}{.22\textwidth}
\includegraphics[scale=.25, angle=0]{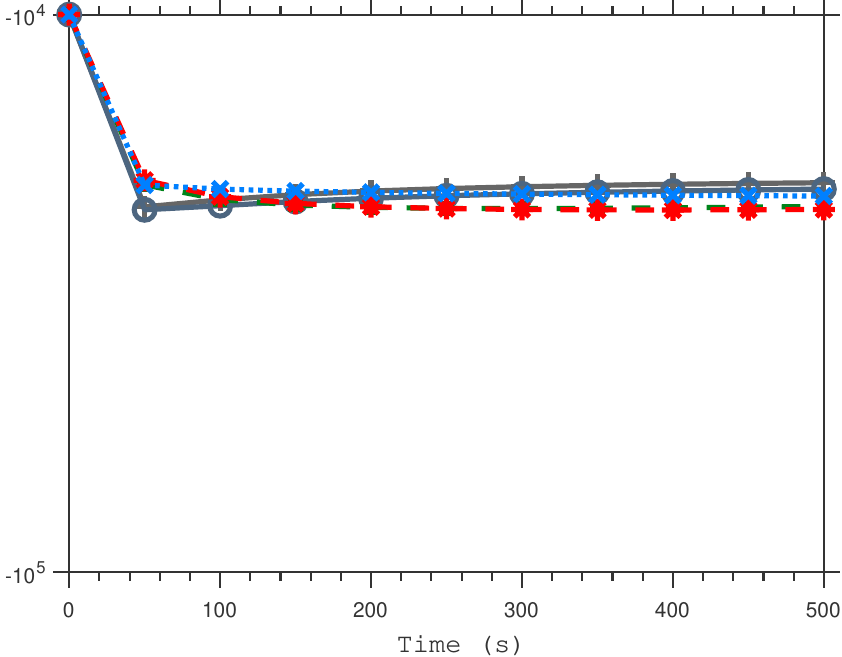}
\end{minipage}
&
\begin{minipage}{.22\textwidth}
\includegraphics[scale=.25, angle=0]{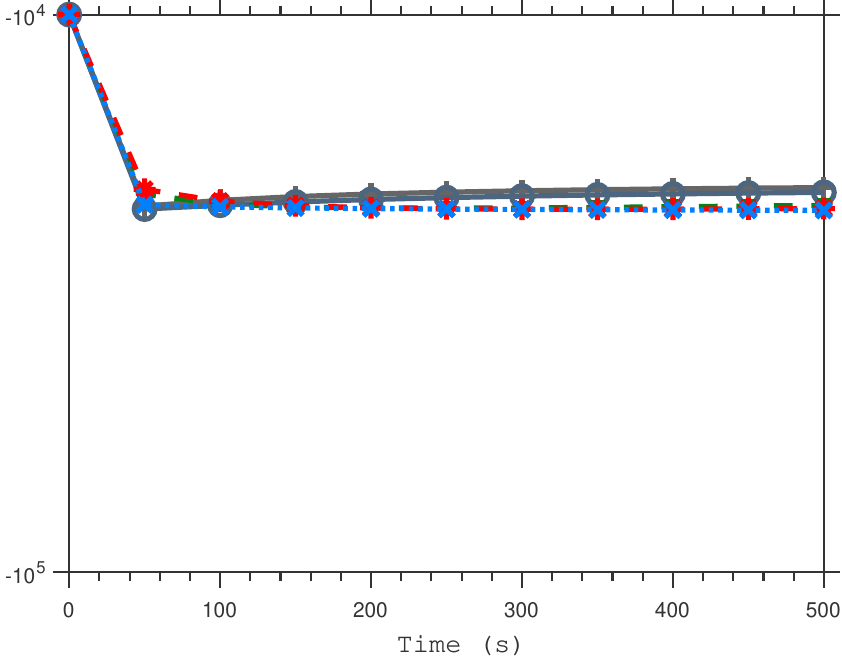}
\end{minipage}
\end{tabular}}
\captionof{figure}{Simulation results for a stochastic Nash Cournot game with 10 players over a network with 10 nodes, comparing Algorithm \eqref{algorithm:aR-IP-SeG} with other existing methods for solving problem \eqref{prob:sopt_svi}.}
\label{fig:comparison4}
\vspace{-.1in}
\end{table} 

\begin{figure}[H]
\centering
\begin{subfigure}{.5\textwidth}
  \centering
  \includegraphics[width=.8\linewidth]{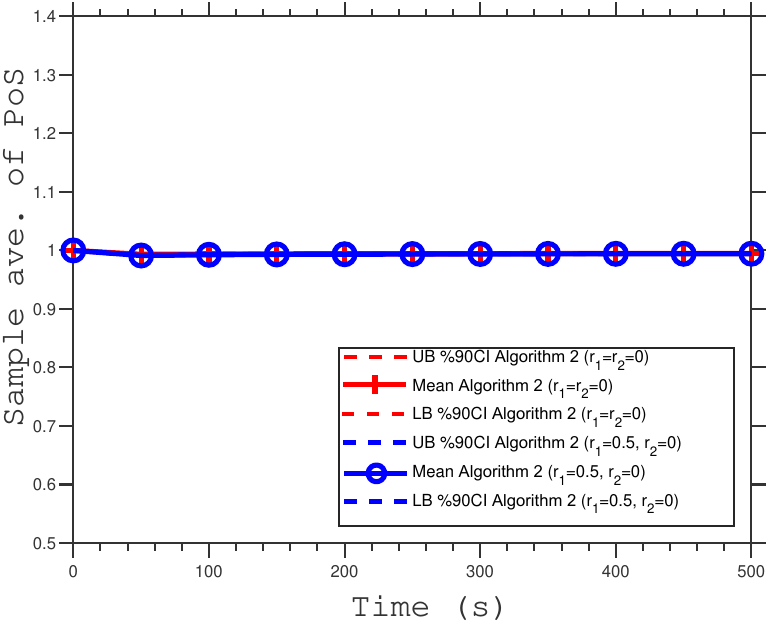}
  \caption{Cournot game with 2 players and network 2 nodes}
  \label{fig:sub1}
\end{subfigure}%
\begin{subfigure}{.5\textwidth}
  \centering
  \includegraphics[width=.8\linewidth]{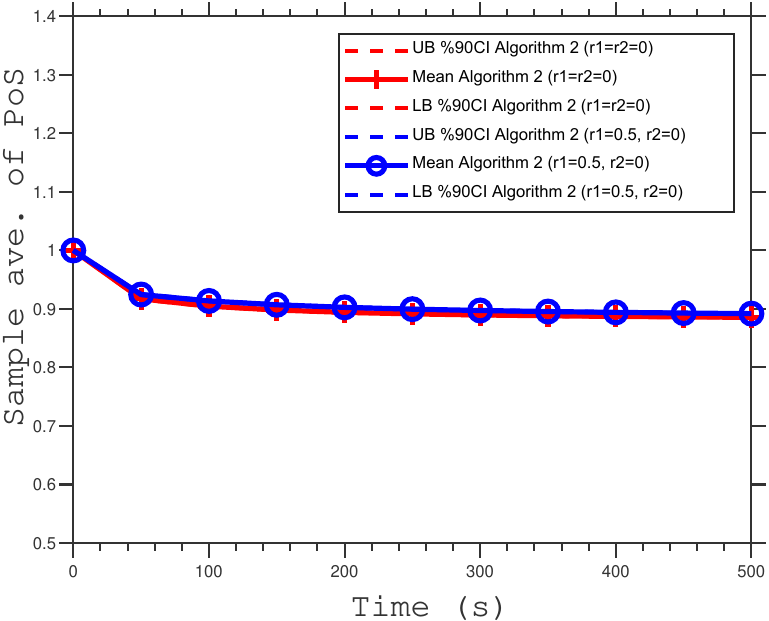}
  \caption{Cournot game with 4 players and network 5 nodes}
  \label{fig:sub2}
\end{subfigure}
\begin{subfigure}{.5\textwidth}
  \centering
  \includegraphics[width=.8\linewidth]{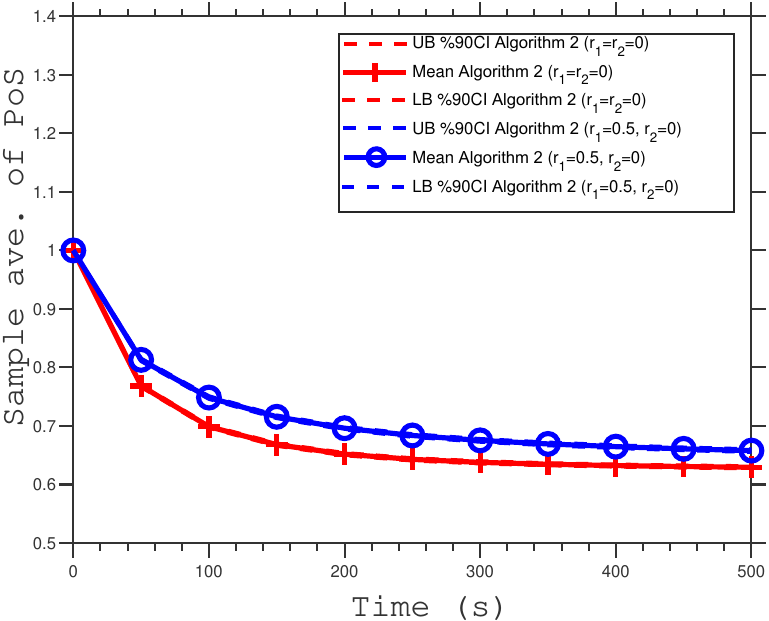}
  \caption{Cournot game with 10 players and network 2 nodes}
  \label{fig:sub2}
\end{subfigure}%
\begin{subfigure}{.5\textwidth}
  \centering
  \includegraphics[width=.8\linewidth]{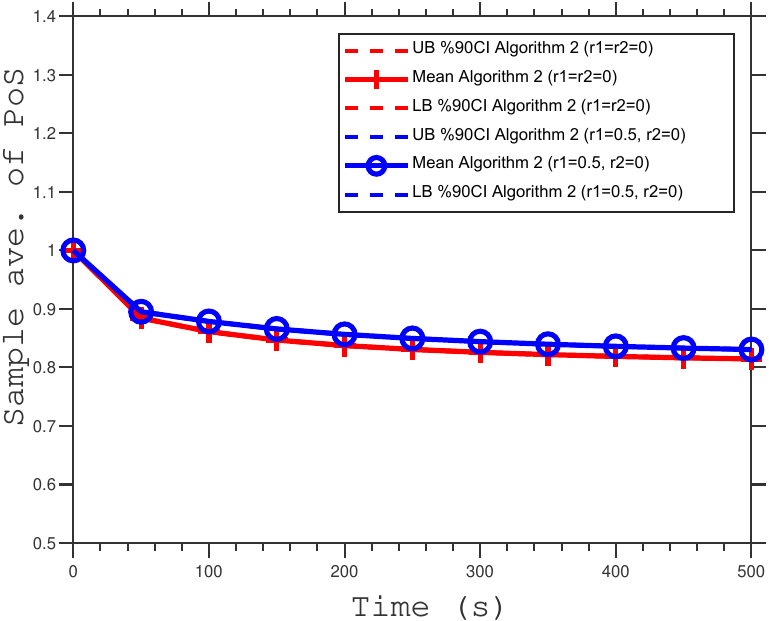}
  \caption{Cournot game with 10 players and network 10 nodes}
  \label{fig:sub2}
\end{subfigure}
\caption{Performance of Algorithm~\ref{algorithm:pos} in estimating PoS. 90\% confidence intervals become tighter as the scheme proceeds.}
\label{fig:pos}
\end{figure}

\section{Acknowledgments}
 This work is supported by the NSF CAREER grant ECCS-1944500, \fa{NSF grant ECCS-2231863,} and ONR grant N00014-22-1-2757.
\bibliographystyle{amsplain}
\bibliography{biblio,ref_fy}

\providecommand{\bysame}{\leavevmode\hbox to3em{\hrulefill}\thinspace}
\providecommand{\MR}{\relax\ifhmode\unskip\space\fi MR }
\providecommand{\MRhref}[2]{%
  \href{http://www.ams.org/mathscinet-getitem?mr=#1}{#2}
}
\providecommand{\href}[2]{#2}
\begin{thebibliography}{10}

\bibitem{PoS08}
E.~Anshelevich, A.~Dasgupta, J.~Kleinberg, E.~Tardos, T.~Wexler, and
  T.~Roughgarden, \emph{The price of stability for network design with fair
  cost allocation}, SIAM Journal on Computing \textbf{38} (2008), no.~4,
  1602–--1623.

\bibitem{bertsekas1997nonlinear}
D.~P. Bertsekas, \emph{Nonlinear programming}, Journal of the Operational
  Research Society \textbf{48} (1997), no.~3, 334--334.

\bibitem{bertsekas2003convex}
D.~P. Bertsekas, A.~Nedi{\'c}, and A.~Ozdaglar, \emph{Convex analysis and
  optimization}, vol.~1, Athena Scientific, 2003.

\bibitem{broadie2014multidimensional}
M.~Broadie, D.~M. Cicek, and A.~Zeevi, \emph{Multidimensional stochastic
  approximation: Adaptive algorithms and applications}, ACM Transactions on
  Modeling and Computer Simulation (TOMACS) \textbf{24} (2014), no.~1, 1--28.

\bibitem{censor2011subgradient}
Y.~Censor, A.~Gibali, and S.~Reich, \emph{The subgradient extragradient method
  for solving variational inequalities in hilbert space}, Journal of
  Optimization Theory and Applications \textbf{148} (2011), no.~2, 318--335.

\bibitem{censor2012extensions}
\bysame, \emph{Extensions of {K}orpelevich's extragradient method for the
  variational inequality problem in {E}uclidean space}, Optimization
  \textbf{61} (2012), no.~9, 1119--1132.

\bibitem{chen2017accelerated}
Y.~Chen, G.~Lan, and Y.~Ouyang, \emph{Accelerated schemes for a class of
  variational inequalities}, Mathematical Programming \textbf{165} (2017),
  no.~1, 113--149.

\bibitem{cui2016analysis}
S.~Cui and U.~V. Shanbhag, \emph{On the analysis of reflected gradient and
  splitting methods for monotone stochastic variational inequality problems},
  2016 IEEE 55th Conference on Decision and Control (CDC), IEEE, 2016,
  pp.~4510--4515.

\bibitem{cui2021analysis}
\bysame, \emph{On the analysis of variance-reduced and randomized projection
  variants of single projection schemes for monotone stochastic variational
  inequality problems}, Set-Valued and Variational Analysis \textbf{29} (2021),
  no.~2, 453--499.

\bibitem{distrobustfl20}
Y.~Deng, M.~M. Kamani, and M.~Mahdavi, \emph{Distributionally robust federated
  averaging}, Advances in Neural Information Processing Systems (H.~Larochelle,
  M.~Ranzato, R.~Hadsell, M.F. Balcan, and H.~Lin, eds.), vol.~33, Curran
  Associates, Inc., 2020, pp.~15111--15122.

\bibitem{Dubey86}
P.~Dubey, \emph{Inefficiency of {N}ash equilibria}, Mathematics of Operations
  Research \textbf{11} (1986), no.~1, 1--8.

\bibitem{FacchineiPang2003}
F.~Facchinei and J.-S. Pang, \emph{Finite-dimensional variational inequalities
  and complementarity problems. {V}ols. {I,II}}, Springer Series in Operations
  Research, Springer-Verlag, New York, 2003.

\bibitem{GAN14}
I.~Goodfellow, J.~Pouget-Abadie, Me. Mirza, B.~Xu, D.~Warde-Farley, S.~Ozair,
  A.~Courville, and Y.~Bengio, \emph{Generative adversarial nets}, Advances in
  Neural Information Processing Systems (Z.~Ghahramani, M.~Welling, C.~Cortes,
  N.~Lawrence, and K.Q. Weinberger, eds.), vol.~27, Curran Associates, Inc.,
  2014.

\bibitem{hsieh2019convergence}
Y.-G. Hsieh, F.~Iutzeler, J.~Malick, and P.~Mertikopoulos, \emph{On the
  convergence of single-call stochastic extra-gradient methods}, Advances in
  Neural Information Processing Systems \textbf{32} (2019).

\bibitem{iusem2017extragradient}
A.~N. Iusem, A.~Jofr{\'e}, R.~I. Oliveira, and P.~Thompson, \emph{Extragradient
  method with variance reduction for stochastic variational inequalities}, SIAM
  Journal on Optimization \textbf{27} (2017), no.~2, 686--724.

\bibitem{iusem2011korpelevich}
A.~N. Iusem and M.~Nasri, \emph{Korpelevich’s method for variational
  inequality problems in banach spaces}, Journal of Global Optimization
  \textbf{50} (2011), no.~1, 59--76.

\bibitem{jalilzadeh2019proximal}
A.~Jalilzadeh and U.~V. Shanbhag, \emph{A proximal-point algorithm with
  variable sample-sizes (ppawss) for monotone stochastic variational inequality
  problems}, 2019 Winter Simulation Conference (WSC), IEEE, 2019,
  pp.~3551--3562.

\bibitem{jalilzadeh2018smoothed}
A.~Jalilzadeh, U.~V. Shanbhag, J.~H. Blanchet, and P.~W. Glynn, \emph{Smoothed
  variable sample-size accelerated proximal methods for nonsmooth stochastic
  convex programs}, Stochastic Systems \textbf{12} (2022), no.~4, 373--410.

\bibitem{jiang2008stochastic}
H.~Jiang and H.~Xu, \emph{Stochastic approximation approaches to the stochastic
  variational inequality problem}, IEEE Transactions on Automatic Control
  \textbf{53} (2008), no.~6, 1462--1475.

\bibitem{JohariThesis}
R.~Johari, \emph{Efficiency loss in market mechanisms for resource allocation},
  Ph.D. thesis, MIT, 2004.

\bibitem{juditsky2011solving}
A.~Juditsky, A.~Nemirovski, and C.~Tauvel, \emph{Solving variational
  inequalities with stochastic mirror-prox algorithm}, Stochastic Systems
  \textbf{1} (2011), no.~1, 17--58.

\bibitem{KannanShanbhag2012}
A.~Kannan and U.~V. Shanbhag, \emph{Distributed computation of equilibria in
  monotone {N}ash games via iterative regularization techniques}, SIAM Journal
  on Optimization \textbf{22} (2012), no.~4, 1177--1205.

\bibitem{kannan2014pseudomonotone}
A.~Kannan and U.~V. Shanbhag, \emph{The pseudomonotone stochastic variational
  inequality problem: Analytical statements and stochastic extragradient
  schemes}, 2014 American Control Conference, IEEE, 2014, pp.~2930--2935.

\bibitem{kannan2019optimal}
\bysame, \emph{Optimal stochastic extragradient schemes for pseudomonotone
  stochastic variational inequality problems and their variants}, Computational
  Optimization and Applications \textbf{74} (2019), no.~3, 779--820.

\bibitem{kaushik2021method}
H.~D. Kaushik and F.~Yousefian, \emph{A method with convergence rates for
  optimization problems with variational inequality constraints}, SIAM Journal
  on Optimization \textbf{31} (2021), no.~3, 2171--2198.

\bibitem{korpelevich1976extragradient}
G.~M. Korpelevich, \emph{The extragradient method for finding saddle points and
  other problems}, Matecon \textbf{12} (1976), 747--756.

\bibitem{koshal12regularized}
J.~Koshal, A.~Nedi\'{c}, and U.~V. Shanbhag, \emph{Regularized iterative
  stochastic approximation methods for variational inequality problems}, IEEE
  Transactions on Automatic Control \textbf{58(3)} (2013), 594--609.

\bibitem{marcotte1998weak}
P.~Marcotte and D.~Zhu, \emph{Weak sharp solutions of variational
  inequalities}, SIAM Journal on Optimization \textbf{9} (1998), no.~1,
  179--189.

\bibitem{ned09}
A.~Nedi\'c and A.~Ozdaglar, \emph{Subgradient methods for saddle-point
  problems}, Journal of Optimization Theory and Applications \textbf{142}
  (2009), 205--228.

\bibitem{Nem04}
A.~Nemirovski, \emph{Prox-method with rate of convergence o(1/t) for
  variational inequalities with lipschitz continuous monotone operators and
  smooth convex-concave saddle point problems}, SIAM Journal on Optimization
  \textbf{15} (2004), no.~1, 229--251.

\bibitem{nemirovski_robust_2009}
A.~Nemirovski, A.~Juditsky, G.~Lan, and A.~Shapiro, \emph{Robust stochastic
  approximation approach to stochastic programming}, {SIAM} Journal on
  Optimization \textbf{19} (2009), no.~4, 1574--1609.

\bibitem{Polyak92acceleration}
B.~T. Polyak and A.~B. Juditsky, \emph{Acceleration of stochastic approximation
  by averaging}, SIAM J. Control Optim. \textbf{30} (1992), no.~4, 838--855.
  \MR{1167814 (93g:62110)}

\bibitem{robbins51sa}
H.~Robbins and S.~Monro, \emph{A stochastic approximation method}, Ann. Math.
  Statistics \textbf{22} (1951), 400--407.

\bibitem{roughgarden2010algorithmic}
T.~Roughgarden, \emph{Algorithmic game theory}, Communications of the ACM
  \textbf{53} (2010), no.~7, 78--86.

\bibitem{sibony1970methodes}
M.~Sibony, \emph{M{\'e}thodes it{\'e}ratives pour les {\'e}quations et
  in{\'e}quations aux d{\'e}riv{\'e}es partielles non lin{\'e}aires de type
  monotone}, Calcolo \textbf{7} (1970), no.~1, 65--183.

\bibitem{minimax28}
v.~Neumann, \emph{Zur theorie der gesellschaftsspiele}, Mathematische Annalens
  \textbf{19} (1928), no.~2, 295--320.

\bibitem{fairgan18}
D.~Xu, S.~Yuan, L.~Zhang, and X.~Wu, \emph{{Fairgan}: Fairness-aware generative
  adversarial networks}, 2018 IEEE International Conference on Big Data (Big
  Data), IEEE, 2018, pp.~570--575.

\bibitem{Farzad1}
F.~Yousefian, A.~Nedi{\'c}, and U.~V. Shanbhag, \emph{On stochastic gradient
  and subgradient methods with adaptive steplength sequences}, Automatica
  \textbf{48} (2012), no.~1, 56--67, An extended version of the paper available
  at: http://arxiv.org/abs/1105.4549.

\bibitem{yousefian2013regularized}
\bysame, \emph{A regularized smoothing stochastic approximation ({RSSA})
  algorithm for stochastic variational inequality problems}, 2013 Winter
  Simulations Conference (WSC), IEEE, 2013, pp.~933--944.

\bibitem{yousefian2014optimal}
\bysame, \emph{Optimal robust smoothing extragradient algorithms for stochastic
  variational inequality problems}, 53rd IEEE Conference on Decision and
  Control, IEEE, 2014, pp.~5831--5836.

\bibitem{yousefian2017smoothing}
\bysame, \emph{On smoothing, regularization, and averaging in stochastic
  approximation methods for stochastic variational inequality problems},
  Mathematical Programming \textbf{165} (2017), no.~1, 391--431.

\bibitem{yousefian2018stochastic}
\bysame, \emph{On stochastic mirror-prox algorithms for stochastic cartesian
  variational inequalities: Randomized block coordinate and optimal averaging
  schemes}, Set-Valued and Variational Analysis \textbf{26} (2018), no.~4,
  789--819.

\end{thebibliography}

\end{document}